\documentclass[10pt]{amsart}
\title[Classification of irreducible admissible genuine mod $p$ representations of $p$-adic $\widetilde{SL}_2$]{A classification of the irreducible admissible genuine mod $p$ representations of $p$-adic $\widetilde{SL}_2$} 
\author{Laura Peskin}

\usepackage{pdfsync}
\usepackage{amsmath} 
\usepackage{amsfonts} 
\usepackage[all]{xy}
\usepackage{amssymb} 
\usepackage{amsthm} 
\usepackage{fancyhdr}
\usepackage{xcolor}
\usepackage{graphicx}
\usepackage{url}

\newtheorem{thm}{Theorem}
\newtheorem*{*thm}{Theorem}
\newtheorem*{*cor}{Corollary}
\newtheorem{cor}[thm]{Corollary}
\newtheorem{lem}[thm]{Lemma}
\newtheorem{prop}[thm]{Proposition}

\newtheorem{rem}[thm]{Remark}

\newcommand{\NN}{\ensuremath{\mathbb{N}}}

\newcommand{\ZZ}{\ensuremath{\mathbb{Z}}}

\newcommand{\CC}{\ensuremath{\mathbb{C}}}
\newcommand{\QQ}{\ensuremath{\mathbb{Q}}}

\newcommand{\cI}{\ensuremath{\mathcal{I}}}

\newcommand{\OF}{\ensuremath{\mathcal{O}_{F}}}
\newcommand{\Gt}{\ensuremath{\widetilde{G}}}
\newcommand{\Tt}{\ensuremath{\widetilde{T}}}

\newcommand{\Bt}{\ensuremath{\widetilde{B}}}
\newcommand{\Bbt}{\ensuremath{\widetilde{\overline{B}}}}
\newcommand{\Kt}{\ensuremath{\widetilde{K}}}
\newcommand{\Ks}{\ensuremath{K^*}}
\newcommand{\sigr}{\ensuremath{\sigma_{\vec{r}}}}
\newcommand{\sigs}{\ensuremath{\sigma_{\vec{s}}}}
\newcommand{\sigrt}{\ensuremath{\tilde{\sigma}_{\vec{r}}}}
\newcommand{\sigst}{\ensuremath{\tilde{\sigma}_{\vec{s}}}}
\newcommand{\SHA}{\ensuremath{\mathcal{H}(\widetilde{G}, \widetilde{K}, \tilde{\sigma}_{\vec{r}})}}
\newcommand{\SHAs}{\ensuremath{\mathcal{H}(\widetilde{G}, \widetilde{K}, \tilde{\sigma}_{\vec{r}}, \tilde{\sigma}_{\vec{s}})}}
\newcommand{\SHT}{\ensuremath{\mathcal{H}(\widetilde{T}, \widetilde{T} \cap \widetilde{K}, (\tilde{\sigma}_{\vec{r}})_{\Ubark})}}

\newcommand{\SHAps}{\ensuremath{\mathcal{H}(\widetilde{G}, \widetilde{K}', \tilde{\sigma}_{\vec{r}}^{\tilde{\alpha}}, \tilde{\sigma}_{\vec{s}}^{\tilde{\alpha}})}}

\newcommand{\Vsigt}{\ensuremath{V_{\vec{r}}}}

\newcommand{\phit}{\ensuremath{\varphi}}

\newcommand{\UK}{\ensuremath{U \cap K}}
\newcommand{\Ubar}{\ensuremath{\overline{U}}}
\newcommand{\Bbar}{\ensuremath{\overline{B}}}
\newcommand{\Uk}{\ensuremath{U(\mathfrak{k})}}
\newcommand{\Ubark}{\ensuremath{\overline{U}(\mathfrak{k})}}
\newcommand{\UKt}{\ensuremath{U^* \cap K^*}}
\newcommand{\TKt}{\ensuremath{\widetilde{T}\cap \widetilde{K}}}
\newcommand{\ind}{\text{ind}}
\newcommand{\Ind}{\text{Ind}}
\newcommand{\End}{\text{End}}
\newcommand{\Hom}{\text{Hom}}
\numberwithin{equation}{section}
\numberwithin{thm}{section}

\begin{document}
\maketitle

\begin{abstract}
We classify the irreducible, admissible, smooth, genuine mod $p$ representations of the metaplectic double cover of $SL_2(F)$, where $F$ is a $p$-adic field and $p \neq 2$. We show, using a generalized Satake transform, that each such representation is isomorphic to a certain explicit quotient of a compact induction from a maximal compact subgroup by an action of a spherical Hecke operator, and we define a parameter for the representation in terms of this data. We show that our parameters distinguish genuine nonsupercuspidal representations from genuine supercuspidals, and that every irreducible genuine nonsupercuspidal representation is in fact an irreducible principal series representation. In particular, the metaplectic double cover of $p$-adic $SL_2$ has no genuine special mod $p$ representations. 
\end{abstract}

\renewcommand{\baselinestretch}{0.6}\normalsize
\setcounter{tocdepth}{1}
\tableofcontents
\renewcommand{\baselinestretch}{1.1}\normalsize

\section{Introduction}

\subsection{Context}
Central extensions of algebraic groups appear in several key roles in number theory, perhaps first in representation-theoretic studies of theta functions and more generally of automorphic forms of half-integral weight. In that context, the relevant covering group is a \textit{metaplectic group}, a central extension $\widetilde{Sp}_{2n}$ of an even-dimensional symplectic group by the square roots of unity. The work of Weil \cite{weil:unitaires} on the representations of metaplectic groups, in particular a distinguished (Weil) representation, led to the theory of dual reductive pairs and theta correspondence. A central work in this area is the parametrization by Waldspurger of the irreducible genuine representations of $\widetilde{SL}_2(F)$, where $F$ is a local field of characteristic 0 (different from $\CC$) and a \textit{genuine} representation is one which does not factor through a representation of $SL_2(F)$. This parametrization, by the irreducible representations of $PGL_2(F)$ and of the units of the projectivization of a quaternion algebra, has been shown by Waldspurger and others (cf. $\S$2.2 of \cite{gelbartps:cuspforms}) to encode the local factors of automorphic representations which appear in the statement of local-global compatibility for the classical local Langlands correspondence (LLC). 

The past several years have seen much investigation of the role of covering groups in the Langlands Program. Recently, Waldspurger's parametrization has been generalized to $\widetilde{Sp}_{2n}(F)$ (for $F$ a local field of characteristic 0 and odd residual characteristic) by Gan and Savin \cite{gansavin:epsilon}, who also define a local Langlands correspondence for $\widetilde{Sp}_{2n}$ via the existing LLC for the special orthogonal groups whose representations provide the parametrizing set. In the general setting of Brylinski-Deligne covers, Weissman \cite{weissman:metaplecticlgps} has constructed $L$-groups and $L$-parameters for covers of split groups. The recent preprint \cite{gangao:langlandsweissman} of Gan and Gao, incorporating Weissman's refinement of his $L$-group construction, defines an unramified LLC for this class of covering groups. 

Another recent focus of activity has been the mod $p$ local Langlands program, which aims to relate the irreducible mod $p$ representations of a reductive group over a $p$-adic field $F$ to certain mod $p$ representations of the absolute Galois group $\text{Gal}(\overline{F}/F).$ Such a correspondence exists for $GL_2(\QQ_p)$ by work of Barthel and Livn\'e \cite{barthellivne:irredmodp}, \cite{barthellivne:ordunram} and Breuil \cite{breuil:modpgl2i}, and has been shown to be induced by Colmez's functor \cite{colmez:gl2qpphigamma} which realizes the $p$-adic LLC for $GL_2(\QQ_p)$. For higher-rank groups, and even for $GL_2(F)$ when $F$ is a proper extension of $\QQ_p$, it quickly becomes difficult even to formulate a precise conjecture. As a prerequisite, one should have a classification of the irreducible admissible mod $p$ representations of the desired reductive group. Such a classification exists up to supercuspidals for split reductive groups due to Abe \cite{abe:irredmodp}, building on Herzig's  classification (likewise up to supercuspidals) for $GL_n(F)$ \cite{herzig:modpgln}, and has recently been generalized to connected reductive groups by Abe, Henniart, Herzig, and Vign\'eras \cite{abehenniartherzigvigneras:irredmodp}. 

A parallel question, whose relationship to the mod $p$ Langlands program is so far unclear but intriguing, is the existence of a mod $p$ theta correspondence. Shin \cite{shin:avweilrepn} notes that several objects appearing in the classical theta correspondence fail to carry over to the mod $p$ setting, and even once reconstructed in geometric terms, appear to behave differently from their classical counterparts \cite{shin:modptheta}. However, a weak version of Howe duality for unramified representations, defined in terms of the respective unramified spherical Hecke algebras, holds in the case of a type II dual reductive pair (\cite{shin:modptheta} Theorem 5.14). 

Our main goal in this paper is to provide a detailed study of the mod $p$ representations of $\widetilde{SL}_2(F)$, where $F$ is a $p$-adic field of odd residual characteristic. As $SL_2(F)$ is a rank-one group with a unique topological central extension of degree 2, it is possible to work very concretely while keeping track of the effects of all choices involved in the parametrization of representations. In particular, we give a classification of the irreducible admissible genuine mod $p$ representations of $\widetilde{SL}_2(F)$ along the lines of Barthel and Livn\'e's classification of the irreducible mod $p$ representations of $GL_2(F)$. This classification may be viewed as a generalized Satake parametrization, incorporating all admissible irreducible genuine representations (not only the unramified ones). Similar parametrizations already exist for $PGL_2(F)$ (by Barthel-Livn\'e \cite{barthellivne:irredmodp} and \cite{barthellivne:ordunram}), $SL_2(F)$ (by Abdellatif \cite{abdellatif:thesis}), and $PD^\times$ where $D$ is a quaternion algebra (by Cheng, \cite{cheng:modpjl}). Thus we hope this work will provide most of the information needed to use $\widetilde{SL}_2(F)$ as a test case for conjectures in the above-mentioned mod $p$ programs. 

This hope has informed our choice of techniques. For example, Hecke algebras and their modules play a key role both in the mod $p$ Langlands program and in Gan-Savin's formulation of a local Shimura correspondence \cite{gansavin:heckecorr}. On the way to our classification, we give explicit descriptions of the genuine spherical mod $p$ Hecke algebras of $\widetilde{SL}_2(F)$ using a generalized Satake transform. 

\subsection{Main Results}

Let $F$ be a $p$-adic field with $p \neq 2$ and residue field of order $q = p^f$, let $E$ be an algebraic closure of $\mathbb{F}_q$, and from now on write $\Gt$ for the metaplectic group $\widetilde{SL}_2(F)$. (See $\S$\ref{subsubsec:covgp} for the construction of the metaplectic group.) The main result is a classification theorem for irreducible admissible genuine mod $E$-representations of $\Gt$. To each such representation, we attach a parameter $(\vec{r}, \lambda)$ consisting of a vector $\vec{r} \in \{0, \dots, p-1\}^f$ and a value $\lambda \in E$. 

In fact we attach two such parameters, one for each of the two conjugacy classes of maximal compact subgroups in $\Gt$. We briefly explain how this is done. For any maximal compact subgroup $\widetilde{\mathcal{K}}$ of $\Gt$, the irreducible genuine representations (which we call \textit{weights}) of $\widetilde{\mathcal{K}}$ are indexed by vectors $\vec{r} \in \{ 0 , \dots, p-1\}^f$ (Proposition \ref{prop:ktweights} and Lemma \ref{lem:kt'weights}). Let $\tilde{\sigma}_{\vec{r}}$ denote the weight of $\widetilde{\mathcal{K}}$ with index $\vec{r}$. The endomorphism algebra $\End_{\Gt}(\ind_{\widetilde{\mathcal{K}}}^{\Gt}\tilde{\sigma}_{\vec{r}})$, a \textit{genuine spherical Hecke algebra}, acts on the compact induction $\ind_{\widetilde{\mathcal{K}}}^{\Gt}\tilde{\sigma}_{\vec{r}}$. We show (Corollary \ref{cor:injimag}) that each genuine spherical Hecke algebra is a polynomial algebra in a single operator $\mathcal{T}$, the precise form of which depends on $\vec{r}$ and $\widetilde{\mathcal{K}}$. The generator $\mathcal{T}$ of the genuine spherical Hecke algebra then must act by a scalar on the compact induction. An irreducible genuine representation $\pi$ is given the parameter $(\vec{r}, \lambda)$ with respect to $\widetilde{\mathcal{K}}$ if $\pi$ is isomorphic to a quotient of the $\Gt$-module $\frac{\ind_{\widetilde{\mathcal{K}}}^{\Gt}\tilde{\sigma}_{\vec{r}}}{(\mathcal{T} - \lambda)}$. If $\pi$ is admissible, then (Proposition \ref{prop:possparamexists}) $\pi$ has a parameter with respect to any maximal compact subgroup.  

The parameter attached to $\pi$ does not depend on the choice of maximal compact subgroup within a given conjugacy class, but does depend on the choice of conjugacy class. (This dependence is described for nonsupersingular parameters by Theorem \ref{thm:dictionary} below.) A parameter is called \textit{supersingular} if $\lambda \neq 0$. We show that if $\pi$ has a supersingular parameter with respect to one maximal compact subgroup, then all of its parameters (with respect to any maximal compact subgroup) are supersingular. Let $\Tt$ denote the preimage in $\Gt$ of the diagonal torus of $G$, and say that an irreducible admissible genuine representation of $\Gt$ is \textit{supercuspidal} if it is not isomorphic to a subquotient of the parabolic induction of an irreducible genuine representation of $\Tt$. Then, under the admissibility hypothesis, the property of having a supersingular parameter is equivalent to supercuspidality. This is our first main result: 

\begin{*thm}[Theorem \ref{thm:irredadmissclassn}] 
The smooth, genuine, irreducible, admissible $E$-representations of $\Gt$ fall into two disjoint classes:
\begin{enumerate}
\item those which have only nonsupersingular parameters,
\item those which have only supersingular parameters.
\end{enumerate}
The representations in the first class are exactly the genuine principal series representations of $\Gt$. All representations in the second class are supercuspidal. 
\end{*thm}

In particular, we show that all genuine principal series representations of $\Gt$ are irreducible. We are able to give a more refined description of these representations, proving (Theorem \ref{thm:pswtisom}) that a representation having a nonsupersingular parameter $(\vec{r}, \lambda)$ with respect to a maximal compact  $\widetilde{\mathcal{K}}$ is isomorphic to the cokernel module $\frac{\ind_{\widetilde{\mathcal{K}}}^{\Gt}\sigrt}{\mathcal{T}-\lambda}$ mentioned above. 

We fix representatives $\Kt$ and $\Kt'$ for the two conjugacy classes of maximal compact subgroups, where $\Kt$ is the preimage in $\Gt$ of $SL_2(\OF)$. We denote the cokernel module $\frac{\ind_{\Kt}^{\Gt}\sigrt}{\mathcal{T} - \lambda}$ by $\pi(\vec{r}, \lambda)$, while the analogous module with respect to $\Kt'$ is denoted by $\pi'(\vec{r}, \lambda)$; a cokernel module is called supersingular if $\lambda = 0$ and nonsupersingular otherwise. As part of the proof of Theorem \ref{thm:irredadmissclassn}, we show that no supersingular cokernel module is isomorphic to a nonsupersingular one. The following theorem gives all equivalences between nonsupersingular cokernel modules.

\begin{*thm}[Theorem \ref{thm:dictionary}] Given $\vec{r} \in \{0, \dots, p-1\}^f$, let $\vec{r'}$ denote any vector in $\{0, \dots, p-1\}^f$ such that 
 $$\sum_{i = 0}^{f-1}r_i'p^i \equiv \sum_{i = 0}^{f-1}\left(r_i + \frac{p-1}{2}\right)p^i \pmod{q-1}.$$ 
Let $\overrightarrow{p-1}$ (resp., $\overrightarrow{\frac{p-1}{2}}$) denote the vector in $\{0, \dots, p-1\}^f$ with all entries equal to $p-1$ (resp., to $\frac{p-1}{2}$). (Then $\vec{r'}$ is uniquely determined if $\vec{r}\neq \overrightarrow{\frac{p-1}{2}}$, and $\vec{r'} \in \{ \vec{0}, \overrightarrow{p-1}\}$ if $\vec{r} = \overrightarrow{\frac{p-1}{2}}$.) 
\begin{enumerate}
\item For any $\vec{r} \in \{ 0, \dots, p-1\}^f$ and $\lambda \in E^\times$, $$\pi(\vec{r}, \lambda) \cong \pi'(\vec{r'}, \lambda).$$
Then, in particular, $$\pi(\vec{0}, \lambda) \cong \pi(\overrightarrow{p-1}, \lambda) \cong \pi'\left(\overrightarrow{\frac{p-1}{2}}, \lambda\right)$$ and $$\pi'(\vec{0}, \lambda) \cong \pi'(\overrightarrow{p-1}, \lambda) \cong \pi\left(\overrightarrow{\frac{p-1}{2}}, \lambda\right).$$
\item The isomorphisms of (1) are the only equivalences between nonsupersingular cokernel modules. 
\end{enumerate}
\end{*thm}
Thus a genuine principal series representation $\pi$ is in one of the following three situations: (i) $\pi$ has a unique parameter with respect to each of $\Kt$ and $\Kt'$, or (ii) $\pi$ has a unique parameter with respect to $\Kt'$ and has exactly two parameters with respect to $\Kt$, or (iii) $\pi$ has a unique parameter with respect to $\Kt$ and has exactly two parameters with respect to $\Kt'$. The metaplectic cover is split over $K$ (resp., over $K'$); we denote the image of this splitting in $\Gt$ by $K^*$ (resp., by $(K')^*$). Then case (ii) occurs exactly when $\pi$ has a $K^*$-fixed vector, and case (iii) occurs exactly when $\pi$ has a $(K')^*$-fixed vector.  

Alternatively, one can parametrize genuine principal series representations by genuine characters of the metaplectic torus $\Tt$. In fact the genuine principal series representations can be parametrized by characters of $F^\times$, but there is an ambiguity equivalent to the choice of a class in $(F^\times)/(F^\times)^2.$  This choice appears also in the complex representation theory of $\widetilde{SL}_2(F)$, where it is tied to the choice of a nontrivial additive character $\psi$ of $F$. In the complex theory, this character $\psi$ determines a genuine character $\chi_{\psi}$ of $\Tt$, which then gives a bijection between genuine characters of $\Tt$ and characters of $F^\times$. Any other nontrivial additive character of $F$ is equal to $\psi_a := \left( x \mapsto \psi(ax)\right)$ for some $a \in F^\times$, and $\psi$ and $\psi_a$ give identical parametrizations of the genuine characters of $\Tt$ if and only if $a \in (F^\times)^2$. 

The additive character $\psi$ plays an important role in Waldspurger's correspondence, and indeed the correspondence depends on the choice of $\psi$. But a nontrivial additive continuous character of $F$ does not even exist in the mod $p$ setting, much less the theory of Whittaker models which motivates the use of $\psi$ in the complex case. On the other hand, it is still possible (Lemma \ref{lem:gencharparam}) to define a \textit{multiplicative $E$-valued} genuine character $\chi_{\psi}$ of $\Tt$ using a \textit{complex-valued additive} character $\psi$ of $F$. One can then proceed to parametrize the mod $p$ genuine characters of $\Tt$ by mod $p$ characters of $F^\times$ using $\chi_{\psi}$. The dependence of this parametrization on the choice of $\psi$ is the same as in the complex case. 

We give a complete cross-indexing between the extant  parametrizations of a genuine principal series representation, namely the two parametrizations with respect to each of the two conjugacy classes of maximal compact subgroups, and the parametrizations by characters of $F^\times$ with respect to nontrivial additive $\CC^*$-valued characters. In the following statement, $(\cdot, \cdot)_F$ denotes the degree-2 Hilbert symbol on $F^\times \times F^\times$, $\varpi$ is a choice of uniformizer of $F$ (fixed throughout the paper), and $\Bbt$ is the full preimage in $\widetilde{SL}_2(F)$ of the lower-triangular Borel subgroup of $SL_2(F)$. 

\begin{*cor}[Corollary \ref{cor:dictionary2}]
Let $\psi: F \rightarrow \CC^\times$ be a nontrivial additive character of conductor $m$, let $\mu: F^\times \rightarrow E^\times$ be a smooth multiplicative character, let $\chi_{\psi}$ be the genuine $E$-valued character of $\Tt$ defined in (\ref{eqn:genchar}), and let $\lambda_{\mu, \psi}$ be the element of $E^\times$ defined in (\ref{eqn:lambdamupsi}). 
\begin{enumerate}
\item Suppose either that $2 \big \vert m$ and $\mu \big \vert_{\OF^\times} = 1$, or that $2 \not \big \vert m$ and $\mu \big \vert_{\OF^\times} = (-, \varpi)_F$. Then the parameters of $\Ind_{\Bbt}^{\Gt}(\mu \cdot \chi_{\psi})$ with respect to $\Kt$ are $(\vec{0}, \lambda_{\mu, \psi})$ and $(\overrightarrow{p-1}, \lambda_{\mu, \psi})$, and $\Ind_{\Bbt}^{\Gt}(\mu \cdot \chi_{\psi})$ has the unique parameter $(\overrightarrow{\frac{p-1}{2}}, \lambda_{\mu, \psi})$ with respect to $\Kt'$. 

\item Suppose either that $2 \big \vert m$ and $\mu \big \vert_{\OF^\times} = (-, \varpi)_F$, or that $2 \not \big \vert m$ and $\mu \big \vert_{\OF^\times} = 1$. Then $\Ind_{\Bbt}^{\Gt}(\mu \cdot \chi_{\psi})$ has the unique parameter $(\overrightarrow{\frac{p-1}{2}}, \lambda_{\mu, \psi})$ with respect to $\Kt$, and the parameters of $\Ind_{\Bbt}^{\Gt}(\mu \cdot \chi_{\psi})$ with respect to $\Kt'$ are $(\vec{0}, \lambda_{\mu, \psi})$ and $(\overrightarrow{p-1}, \lambda_{\mu, \psi})$.  

\item Otherwise, $\Ind_{\Bbt}^{\Gt}(\mu \cdot \chi_{\psi})$ has a unique parameter with respect to $\Kt$, and this parametrization is of the form $(\vec{r}, \lambda_{\mu, \psi})$ for some $\vec{r} \notin \{\vec{0}, \overrightarrow{\frac{p-1}{2}}, \overrightarrow{p-1}\}$. $\Ind_{\Bbt}^{\Gt}(\mu\cdot \chi_{\psi})$ also has a unique parameter with respect to $\Kt'$, equal to $(\vec{r'}, \lambda_{\mu, \psi})$ where $\vec{r'}$ is the unique vector in $\{0, \dots, p-1\}^f$ such that $\sum_{i = 0}^{f-1}r_i'p^i \equiv \sum_{i = 0}^{f-1}\left(r_i + \frac{p-1}{2}\right) p^i \pmod{q-1}$. 
\end{enumerate}
\end{*cor}

In particular, we show that the choice of parity of the conductor of $\psi$ in the latter parametrization is completely equivalent to the choice of a maximal compact in the former parametrization (Corollary \ref{cor:dictionary3}), and that the remaining dependence on $\psi$ is encoded by the second components of the parameters with respect to $\Kt$ and $\Kt'$.
Therefore, one can index genuine principal series representations by $E^\times$-valued characters of $F^\times$ and express the dependence of this indexing on the choice of a class in $(F^\times)/(F^\times)^2$ without any reference to a $\CC^\times$-valued character. 

As for supercuspidal (equivalently, admissible supersingular) genuine representations of $\widetilde{SL}_2(F)$, we show that supercuspidals must appear as quotients of the cokernel modules $\pi(\vec{r}, 0)$ or $\pi'(\vec{r}, 0)$. (We can confirm that $\pi(\vec{r}, 0)$ and $\pi'(\vec{r}, 0)$ have smooth irreducible quotients, but not that they have admissible irreducible quotients.) This paper lays the ground for classifying the irreducible quotients of these cokernel modules for $\widetilde{SL}_2(\QQ_p)$ along the lines of Breuil \cite{breuil:modpgl2i} and for studying their admissibility, and we plan to do this in future work. However, the case of $GL_2(F)$, $F \neq \QQ_p$, leads us to expect difficulties in classifying supercuspidal representations oustide the case of $\widetilde{SL}_2(\QQ_p)$. 

\subsection{Plan of the paper} 
\label{subsec:plan}
The general outline of the paper is modelled on that of Barthel and Livn\'e's founding work  \cite{barthellivne:ordunram}, \cite{barthellivne:irredmodp} on the mod $p$ representations of $GL_2(F)$, and of Herzig's classification of the mod $p$ representations of $GL_n(F)$ up to supercuspidal representations \cite{herzig:modpgln}. Sections 2, 3, and 4 develop the technical prerequisites for applying the methods of \cite{herzig:modpgln} in the new context of the covering group $\widetilde{SL}_2(F)$. The main points which have to be developed are the following:

\begin{enumerate}
\item \textit{Weight theory} ($\S$\ref{sec:weights}). A key fact of mod $p$ representation theory is that every smooth representation of a pro-$p$ group on a vector space of characteristic $p$ must have a fixed vector. Therefore every smooth mod $p$ representation of a maximal compact $\mathcal{K} \subset SL_2(F)$ is tamely ramified. A smooth mod $p$ irreducible representation of $\mathcal{K}$ (or equivalently, an irreducible mod $p$ representation of $SL_2(\mathbb{F}_q)$) is called a \textit{weight}, and there are exactly $q$ inequivalent such representations. We have defined a weight of a maximal compact subgroup $\widetilde{\mathcal{K}} \subset \widetilde{SL}_2(F)$ to be a smooth irreducible \textit{genuine} mod $p$ representation of $\widetilde{\mathcal{K}}$. Then the weights of a maximal compact  $\widetilde{\mathcal{K}} \subset \widetilde{SL}_2(F)$ have a simple classification in terms of the weights of $\mathcal{K}$, and we recover a useful criterion (Proposition \ref{prop:ktwtsignif}) for irreducibility of a smooth genuine mod $p$ representation of $\widetilde{SL}_2(F)$. 

\item \textit{Satake transform} (Proposition \ref{prop:gensataketransformi}). We define a Satake transform from a genuine spherical Hecke algebra of $\widetilde{SL}_2(F)$ to a certain genuine spherical Hecke algebra of the torus $\Tt$. (In fact we define the transform for a more general class of genuine spherical Hecke bimodules, which is needed to establish the ``change-of-weight'' isomorphism discussed in the next point). The definition of the transform is almost identical to that of the map $'\mathcal{S}_G$ defined by Herzig in \cite{herzig:modpgln}, and we have used ideas from \cite{henniartvigneras:cptparabolic} and \cite{herzig:glnnotes} to streamline some arguments. We use the comparison of Cartan and Iwasawa decompositions mentioned above in order to give an explicit formula for the transform in terms of Hecke operators for $\Tt$. This formula is then used to find explicit generators of the spherical Hecke algebras of $\Gt$ and to define the change-of-weight map (below). Various properties of the mod $p$ Satake transform for unramified reductive groups are discussed in \cite{herzig:modpsatake} and apply to our transform as well: in particular, the image is not invariant under the Weyl group.

\item \textit{Change of weight} (Lemma \ref{lem:changeofweight}). A comparison of compact and parabolic inductions (Proposition \ref{prop:wtisomuniv}) suffices to prove irreducibility of all parabolic inductions which do not contain a one-dimensional weight. A similar situation occurs in the classification of mod $p$ representations of $GL_n(F)$, and in \cite{herzig:modpgln} Herzig proceeds by defining a ``change-of-weight'' map. This map is usually an isomorphism and allows one to embed a higher-dimensional weight into the parabolic induction, thus proving it to be irreducible. The obstruction to injectivity of this change-of-weight map is defined in terms of the explicit Satake transform, and in the case of $GL_2$ injectivity fails only when the representation under consideration is the induction of a character of form $\chi \otimes \chi$, $\chi:F^\times \rightarrow \bar{\mathbb{F}}_p^\times$. In our situation, however, there is no obstruction to injectivity: due to the behavior of the spherical Hecke operators for $\widetilde{SL}_2(F)$, the change-of-weight map is \textit{always} an isomorphism. This allows us to prove irreducibility of \textit{all} parabolic inductions, and explains the symmetry between the nonsupersingular parameters $(\vec{0}, \lambda)$ and $(\overrightarrow{p-1}, \lambda)$. 
\end{enumerate}

Once these pieces are in place, the statement of Theorem \ref{thm:pswtisom} (and its refinements) follow with only minor adaptations to the strategy of \cite{herzig:modpgln}. Consequently, the techniques of this paper should generalize well, modulo the above three points, to tame covering groups of higher rank and/or higher degree. 

Along the way, we also develop the following point:

\begin{enumerate}
\item[(iv)] \textit{Local systems on the tree of $SL_2(F)$} ($\S$\ref{subsec:tree}). Given a weight $\sigma$ of a maximal compact subgroup $\mathcal{K} \subset SL_2(F)$, one has a local system on the Bruhat-Tits tree of $SL_2(F)$. In the work of Barthel-Livn\'e, the local system defined by $\sigma$ (a weight of the maximal compact of $GL_2(F)$ in their context) provides a useful framework for understanding the spherical Hecke algebra attached to $\sigma$. For each weight of a maximal compact $\widetilde{\mathcal{K}}$ we define a local system on the tree of $SL_2(F)$, and use this to deduce several properties of the genuine spherical Hecke algebras (e.g. Proposition \ref{prop:SHAfree}).
\end{enumerate}

This final point is useful in the frameworks of Barthel-Livn\'e and Breuil for studying the supersingular representations, and we plan to use it in future work on the supersingular cokernel modules of $\widetilde{SL}_2(F)$.

\subsection{Acknowledgments}
This work grew out of my Ph.D. thesis \cite{peskin:thesis}, and many thanks are due to my advisor Dinakar Ramakrishnan for suggesting the thesis problem. Professors Rachel Ollivier, Nike Vatsal, and Haruzo Hida have each made very helpful comments at various stages of this project. Thanks also to Karol Koziol for several useful comments on a previous draft, including a correction to the statement of Lemma \ref{lem:cartaniwasawa} and suggestion for simplifying its proof. 

\pagebreak

\section{Preliminaries}
\label{sec:prelim}

\subsection{Notation and basic definitions}
\label{subsec:notation}
\subsubsection{The base field} Let $F$ be a $p$-adic field with finite residue field $\mathfrak{k}$ of order $q = p^f$. We assume throughout that $p \neq 2.$ For $n \geq 2$ we denote the group of $n^{\text{th}}$ roots of unity in $F^\times$ by $\mu_n(F)$. Let $\OF$ denote the ring of integers of $F$ and fix a uniformizer $\varpi$, hence also an identification of $\OF/(\varpi\OF)$ with $\mathfrak{k}$. The image of $a \in \OF$ under the reduction map $\text{red}: \OF \rightarrow \mathfrak{k}$ will be denoted by $\overline{a}$. The valuation $v_F$ on $F$ is normalized so that $v_F(\varpi) = 1$. 

The Teichm\"uller lift of $x \in \mathfrak{k}$ is denoted by $[x]$. For $m \geq 1$, let 
$$\cI_{m} = \{ \sum_{i = 0}^{m-1} [\lambda_i]\varpi^i: \, (\lambda_0, \dots, \lambda_{m-1}) \in \mathfrak{k}^m\}.$$ For $m' \geq m \geq 1$, let $[\cdot]_{m}: \cI_{m'} \rightarrow \cI_{m}$ denote the truncation map defined by
$$\Big[\sum_{i = 0}^{m'-1}[\lambda_i]\varpi^i\Big]_{m} = \sum_{i = 0}^{m-1}[\lambda_i]\varpi^i.$$
For $m \geq 1$, $\kappa = \sum_{i = 0}^{m-1}[\kappa_i]\varpi^i \in \cI_m$, and $\lambda = [\lambda_0] \in \cI_1$, let 
$$\kappa_{+\lambda} = [\kappa]_{m-1} + [\kappa_{m-1} + \lambda_0]\varpi^{m-1}.$$

\subsubsection{The Hilbert symbol} The symbol $( \cdot, \cdot)_F$ will denote the degree-2 Hilbert symbol on $F$. We will frequently use the following formula for $(\cdot, \cdot)_F$ which may be found in, e.g., \cite{neukirch:ant} Prop. V.3.4.
For $x \in \OF^\times$, let $\omega(x)$ denote the unique element of $\mu_{q-1}(F)$ such that $x = \omega(x)\cdot \langle x \rangle $ with $\langle x \rangle \in \OF^\times$ congruent to 1 $\pmod{\varpi\OF}$. Since $p \neq 2$, the degree-2 Hilbert symbol is tame, given by 
\begin{equation} \label{eqn:tamesymbol}
(a, b)_F = \omega\left((-1)^{v_F(a)v_F(b)}\frac{b^{v_F(a)}}{a^{v_F(b)}}\right)^{\frac{q-1}{2}}
\end{equation}
for $a$, $b \in F^\times$.
\begin{rem}
\label{rem:hilbsymprops}
 We will make use of the following properties of $(\cdot, \cdot)_F$:
\begin{enumerate}
\item $(\cdot, \cdot)_F$ is symmetric and bilinear.
\item $(-a, a)_F = 1$ for all $a \in F^\times.$ Together with (1), this implies $(a, a)_F = (-1, a)_F$ for all $a \in F^\times.$
\item $(\cdot, \cdot)_F$ is unramified, i.e., $(a, b)_F = 1$ when $a$, $b \in \OF^\times.$
\item $(d, c)_F = 1$ for all $d \in F^\times$ if and only if $c \in (F^\times)^2.$ In particular, $(-1, \varpi)_F = 1$ when $q \equiv 1 \pmod{4}$.
\item Let $(\lambda_0, \dots, \lambda_{n-1}) \in \mathfrak{k}^n$ with $\lambda_0 \neq 0$, so that $\lambda = \sum_{i = 0}^{n-1}[\lambda_i]\varpi^i$ belongs to $\OF^\times.$ Then $(\lambda, \varpi)_F = 1$ if and only if $\lambda_0$ is a square in $\mathfrak{k}$. 
\item As a consequence of (5), for any $n \geq 1$ the set $\{ \lambda \in \cI_n \cap \OF^\times: \left(\lambda, \varpi\right)_F = 1\}$ forms a subset of index 2 in $\cI_n \cap \OF^\times$, and the set $\{ \lambda \in \OF^\times: (\lambda, \varpi)_F = 1\} $ forms a subgroup of index 2 in $\OF^\times.$ 
\end{enumerate}
\end{rem}

\subsubsection{The coefficient field} 
\label{subsubsec:coefffield}
Let $E$ be an algebraically closed field of characteristic $p$ which admits an embedding of $\mathfrak{k}$. We fix (but suppress from the notation) an embedding $\mathfrak{k} \hookrightarrow E$. Unless specifically mentioned otherwise, all representations should be assumed to have coefficients in $E$. 

In particular, given a vector $\vec{r} = (r_0, \dots, r_{f-1})$ with $0 \leq r_i \leq p-1$ for each $0 \leq i \leq f-1$, we define a tamely ramified character $\delta_{\vec{r}}: \OF^\times \rightarrow E^\times$ as the composition of the character
\begin{align*}
\OF^\times & \longrightarrow \mathfrak{k}^\times\\
a & \mapsto (\overline{a})^{\sum_{i = 0}^{f-1}r_ip^i}
\end{align*}
with the fixed embedding $\mathfrak{k} \hookrightarrow E$.

\subsubsection{The covering group} 
\label{subsubsec:covgp}
Let $G = SL_2(F)$. In this part we summarize the construction of the metaplectic group, which is a certain extension of $G$ by the square roots of 1. Everything said here is well-known, and much of the summary is specialized (to the case $r = n = 2$, $F$ $p$-adic, $p \neq 2$) from $\S$0.1 of \cite{kazhdanpatterson:metaplecticforms}. 

Since $H^2_{meas}(G, \mu_2(F)) = \mu_2$ (cf. \cite{kubota:topcovering}, also \cite{moore:gpextensions}) there is a unique nontrivial topological central extension of $G$ by $\mu_2(F)$, called the \textit{metaplectic group} and denoted here by $\Gt$. Kubota \cite{kubota:topcovering} produced a cocycle $\Delta$ of nontrivial class in $H^2_{meas}(G, \mu_2(F))$. We define $\Gt$ concretely as the set $G \times \mu_2$ with multiplication 
$$(g, \zeta) \cdot (g', \zeta') = (gg', \zeta \zeta' \Delta(g, g')),$$
where $\Delta \in H^2_{meas}(G, \mu_2(F))$ is Kubota's cocycle, namely 
\begin{equation} \label{eqn:sl2cocycle} 
\Delta(g, g') = \left( \frac{X(gg')}{X(g)}, \frac{X(gg')}{X(g')}\right)_F, \, \, X\left( 
  \left( \begin{array}{cc}
a      & b \\
c      & d 
    \end{array} \right) \right) = \begin{cases}
c & \text{ if } c \neq 0\\
d & \text{ if } c = 0.
\end{cases}
\end{equation}
(This formula for $\Delta$ is simpler than the one appearing in \cite{kubota:topcovering} but is equivalent, cf. \cite{kazhdanpatterson:metaplecticforms} or \cite{mcnamara:psreps}.) Let $Pr: \Gt \rightarrow G$ denote the projection $(g, \zeta) \mapsto g.$ Given a subgroup $H \subset G$, we denote its full preimage in $\Gt$ by $\widetilde{H}.$ 

The covering $\Gt \rightarrow G$ may be extended to a nontrivial double cover of $GL_2(F)$ defined by the following cocycle (which we also denote by $\Delta$):
\begin{equation} \label{eqn:gl2cocycle}
\Delta(g, g')  = \left( \frac{X(gg')}{X(g)}, \frac{X(gg')}{X(g')}\right)_F \cdot \left(\text{det}(g), \frac{X(gg')}{X(g)}\right)_F.
\end{equation} The extension $\widetilde{GL}_2(F) \rightarrow GL_2(F)$ corresponding to $\Delta$ is the unique nontrivial topological extension of $GL_2(F)$ by $\mu_2(F)$ if and only if $F$ contains a primitive fourth root of unity, but all such extensions are indistinguishable in applications for which $\mu_2(F)$ is identified with a subgroup of a field of coefficients containing a primitive fourth root of unity (cf. \cite{kazhdanpatterson:metaplecticforms} $\S$0.1, \textit{Remarks}, p.41-42). Hence there is no loss of generality, from the point of view of representation theory over a sufficiently large field of coefficients, in fixing $\widetilde{GL}_2(F)$ as the preferred double cover of $GL_2(F)$. 

\subsubsection{Splittings and preferred lifts} 
\label{subsubsec:splittings}
The covering $\Gt \rightarrow^{Pr} G$ is canonically split over any unipotent subgroup $\mathcal{U} \subset G$ by the map $u \mapsto (u, 1)$, whose image we denote by $\mathcal{U}^*$. We write $U$ (resp., $\Ubar$) for the upper-triangular (resp., lower-triangular) unipotent subgroup of $\Gt$, and write $u(x)$ (resp., $\bar{u}(x)$) for the element of $U$ (resp., $\Ubar$) with off-diagonal entry equal to $x \in F$. We write $\tilde{u}(x)$ for $(u(x), 1) \in U^*$ and $\tilde{\bar{u}}(x)$ for $(\bar{u}(x), 1) \in \Ubar^*$. 

For $x \in F^\times$, elements of the form $w(x) := 
\left( \begin{array}{cc}
0    & x \\
-x^{-1}    & 0 
  \end{array} \right)$ are generated by unipotent elements of $G$, so we may fix a preferred lift of $w(x)$ using the canonical sections of $U$ and $\Ubar$. This preferred lift is 
$$\tilde{w}(x) := \tilde{u}(x)\cdot \tilde{\bar{u}}(-x^{-1}) \cdot \tilde{u}(x) = (w(x), 1).$$

The restriction of the covering $\Gt \rightarrow^{Pr} G$ does not split over the diagonal torus $T$. Letting $h(x)$ denote the element of $T$ with diagonal entries $(x, x^{-1})$, we have $\Delta(h(x), h(y)) = (x, y)_F$. Since $h(x) = w(x)w(-1),$ we again choose a preferred section $T \rightarrow \Gt$ using the canonical sections of the unipotent subgroups: for $x \in F^\times$, put
$$\tilde{h}(x) := \tilde{w}(x)\tilde{w}(-1) = (h(x), (-1, x)).$$

For future convenience, we define a map $\phi: \ZZ \rightarrow \mu_2$ by 
\begin{equation}
\label{eqn:phidef}
\tilde{h}(\varpi)^n = \left(h(\varpi^n), \phi(n) \right).
\end{equation}

\begin{rem}
\label{rem:phiprops}
Some easy properties of $\phi$ are collected here.
\begin{enumerate}
\item $\phi(n) = \begin{cases}
(-1, \varpi)_F & \text{ if } n \equiv 1 \text{ or } 2 \pmod{4},\\
1 & \text{ if } n \equiv 0 \text{ or } 3 \pmod{4}.
\end{cases}$\\
In particular, $\phi(n) = 1$ for all $n$ if and only if $q \equiv 1 \pmod{4}.$
\item $\phi(n)\phi(-n) = (-1, \varpi)_F^n$ for all $n$. 
\item $\phi(n+1)\phi(-n) = (-1, \varpi)_F$ for all $n$.
\end{enumerate} 
\end{rem}

Let $B = TU$ and $\Bbar = T\Ubar$, and let $\Bt$ and $\Bbt$ denote the respective preimages in $\Gt$. We have $\Bt = \Tt U^*$ and $\Bbt = \Tt \Ubar^*$.

Finally we describe splittings over maximal compact subgroups of $G$. Let $K = SL_2(\OF)$, and let $\alpha = 
\left( \begin{array}{cc}
1    & 0 \\
0    & \varpi 
  \end{array} \right) \in GL_2(F)$ and $K' = 
\alpha K \alpha^{-1} \subset G$. Then $K$ and $K'$ represent the two conjugacy classes of maximal compact subgroups in $G$, and likewise $\Kt$ and $\Kt'$ represent the two conjugacy classes of maximal compact subgroups in $\Gt$. There exists a unique splitting $K \rightarrow \Kt$, whose image we denote by $K^*$, of the restriction of $Pr$ to $\Kt$. Concretely, if $k = 
\left( \begin{array}{cc}
a    & b \\
c    & d 
  \end{array} \right) \in K$ then this splitting sends $k$ to $(k, \theta(k))$, where
\begin{equation} \label{eqn:thetadef}
\theta(k) := \begin{cases}
1 & \text{ if } c = 0 \text{ or } c \in \OF^\times\\
(c, d)_F & \text{ otherwise.}
\end{cases}
\end{equation}

From $\theta$ one can also construct a unique splitting of the extension over $K'$, as follows. Let $\tilde{\alpha}$ be any lift of $\alpha$ to $\widetilde{GL}_2(F)$.  For any $k \in K$, the product $\tilde{\alpha}\cdot (k, \theta(k))\cdot(\tilde{\alpha})^{-1}$, taken in $\widetilde{GL}_2(F)$ using the cocycle (\ref{eqn:gl2cocycle}), is independent of the choice of lift of $\alpha$ and we may define 
$$(\alpha k \alpha^{-1}, \theta'(\alpha k \alpha^{-1})) := \tilde{\alpha}\cdot (k, \theta(k))\cdot (\tilde{\alpha})^{-1}.$$
Then $g \mapsto (g, \theta'(g))$ defines a splitting of $\Gt \rightarrow^{Pr} G$ over $K'$. Explicitly, for $k' \in K'$ and $k = \alpha^{-1} k' \alpha = 
\left( \begin{array}{cc}
a    & b \\
c    & d 
  \end{array} \right) \in K$, we have
\begin{align*}
\theta'(k') & = \begin{cases}
 (c, d)_F & \text{ if } c \neq 0 \text{ and } c \notin \OF^\times,\\
 (d, \varpi)_F & \text{ if } c = 0,\\
 1 & \text{ if } c \in \OF^\times.
\end{cases}
\end{align*}

\begin{rem}
\label{rem:kspl}
The canonical splitting of the covering over a unipotent subgroup agrees with $(-, \theta(-))$ (resp., with $(-, \theta'(-))$) on the intersection of $K$ (resp., of $K'$) with that unipotent subgroup. However, $\theta$ and $\theta'$ do not agree on all of $K \cap K'$: for example if $x \in \OF^\times$ is an element whose reduction modulo $\varpi$ is a nonsquare in $\mathfrak{k}$, and if $k = h(x) \in K \cap K'$, then 
$\theta(k) = 1$ while $\theta'(k) = (x, \varpi)_F = -1.$ There is no conflict between these statements, since $T \cap K = T \cap K'$ is generated by unipotent elements which lie in $K$ but not in $K'$. 
\end{rem}

\subsubsection{Conventions on representations} As already mentioned, all representations should be assumed (unless noted otherwise) to have coefficients in the field $E$ of characteristic $p$. A representation of $\Gt$ is called \textit{smooth} if the subgroup of $\Gt$ fixing each vector is open, and called \textit{genuine} if it does not factor through a representation of $G$. In terms of our explicit construction of $\Gt$, a representation $\rho$ is genuine if and only if $\rho((g, \zeta)) = \zeta \rho((g, 1))$ for all $g \in G$, $\zeta \in \mu_2$. We say that a function $f$ on $\Gt$ is genuine if the same relation holds of $f$. We will work only with smooth, genuine representations of $\Gt$. We do not assume that representations are admissible unless this hypothesis is specifically mentioned.

The symbol $\Ind$ will denote smooth induction (of a smooth representation of a closed subgroup). We do not normalize smooth induction. The symbol $\ind$ will denote compact induction (of a smooth representation of an open subgroup). We use the following standard notation for certain elements of a compact induction: given a smooth representation $\sigma$ of a group $H$ which is an open subgroup of a group $H'$, $g \in H$, and $v \in \sigma$, the symbol $[g, v]$ denotes the element of $\ind_{H}^{H'}\sigma$ which is supported on $Hg^{-1}$ and which takes the value $[g, v](h) = \sigma(hg)\cdot v$ on $h \in Hg^{-1}$.

\subsection{Commutators in $\Gt$}
\label{subsec:commutators}
 For $H$ a subgroup of $\Gt$, let $[H, H]$ denote the subgroup of $\Gt$ generated by commutators in $H$. The following fact is  well-known. 

\begin{lem} \label{lem:commutators}
\begin{enumerate}
\item $[\Tt, \Tt] = 1$. 
\item $[\Bt, \Bt] = U^*$. 
\item $[\Bbt, \Bbt] = \Ubar^*$. 
\item $[\Gt, \Gt] = \Gt$. 
\end{enumerate}
\end{lem}

\begin{proof}
(1), (2), and (3) are straightforward calculations using the cocycle (\ref{eqn:sl2cocycle}). For (4), recall that $[G, G] = G$, so it suffices to show that $(1, -1)$ is also generated by commutators in $\Gt$. By (2) and (3) the subgroups $U^*$ and $\Ubar^*$ are generated by commutators in $\Gt$, which implies that for any $a \in F^\times$ we have $\tilde{h}(a) \in [\Gt, \Gt].$ 

Pick $u \in \OF^\times$ so that $(u, \varpi)_F = -1$ (such a unit exists by Remark \ref{rem:hilbsymprops} (6)). We have 
$$\tilde{h}(u) \cdot \tilde{h}(\varpi)  = 
\left(h(u\varpi), (-1, u\varpi)_F (u, \varpi)_F \right) \in [\Gt, \Gt],$$
and also 
$$\tilde{h}(u\varpi)^{-1} = \left( h(u\varpi)^{-1}, (-1, u\varpi)_F(u\varpi, u\varpi)_F \right) \in [\Gt, \Gt],$$
so 
$$\left(h(u\varpi), (-1, u\varpi)_F (u, \varpi)_F \right) \left(h(u\varpi), (-1, u\varpi)_F \right)^{-1} = (1, (u, \varpi)_F) = (1, -1) \in [\Gt, \Gt].$$
\end{proof}
Since $\Bt = \Tt U^* = U^* \Tt$ (resp., $\Bbt = \Tt \Ubar^* = \Ubar^* \Tt$), we obtain:
\begin{cor}
\label{cor:abelianization}
The abelianization of $\Bt$ (resp., of $\Bbt$) is $U^*$ (resp., $\Ubar^*$), and the abelianization of $\Gt$ is trivial. 
\end{cor}

\subsection{Cartan decompositions of  $\Gt$}
\label{subsec:cartandecomp}
Recall that $G$ has the following Cartan decompositions:
\begin{equation} \label{eqn:gcartandecomp}
G = \coprod_{n \geq 0} Kh(\varpi)^{-n}K  = \coprod_{n \geq 0} K' h(\varpi)^{-n}K'.
\end{equation}

Lifting these decompositions to $\Gt$, we have
\begin{equation}\label{eqn:gtktcartandecomp} \Gt = \coprod_{n \geq 0} \Kt \tilde{h}(\varpi)^{n}\Kt = \coprod_{n \geq 0} \Kt'\tilde{h}(\varpi)^{-n}\Kt'.
\end{equation}
It will be useful to refine (\ref{eqn:gtktcartandecomp}) to a disjoint union of $K^*$- (resp., $K'^*$-) double cosets.

\begin{lem} \label{lem:gtkscartandecomp}
$\Gt$ has the Cartan decompositions
$$\Gt = \coprod_{\substack{n \geq 0\\ \zeta \in \mu_2}} K^* \tilde{h}(\varpi)^{-n}(1, \zeta) K^* = \coprod_{\substack{n \geq 0\\ \zeta \in \mu_2}} K'^* \tilde{h}(\varpi)^{-n}(1, \zeta)K'^*.$$
\end{lem}
 
\begin{proof}[Proof of Lemma \ref{lem:gtkscartandecomp}]
We first prove the decomposition with respect to $K^*$. By (\ref{eqn:gtktcartandecomp}) we have
$$ \Gt  = \coprod_{n \geq 0} \left( \bigcup_{\zeta \in \mu_2}K^{*}\tilde{h}(\varpi)^{-n}(1, \zeta)K^{*} \right).$$
We will show that the union over $\mu_2$ is also disjoint, using a similar idea to that of \cite{mcnamara:psreps} Theorem 9.2. Let $K_n$ denote the intersection $K \cap h(\varpi)^{-n}Kh(\varpi)^{n}$, and define a map $\psi_n: K_n \rightarrow \mu_2$ as follows: $\psi(k) = \zeta$, where $\zeta \in \mu_2$ is the unique sign such that $(k, \theta(k)) \tilde{h}(\varpi)^{-n} = \tilde{h}(\varpi)^{-n}(k', \zeta\theta(k'))$ for some $k' \in K$. The intersection $K^* \tilde{h}(\varpi)^{-n}K^* \cap K^* \tilde{h}(\varpi)^{-n}(1, -1)K^*$ is empty if and only if $\psi_n$ is trivial. The map $\psi_n$ is a group homomorphism since $\theta$ is a splitting of the covering $\Gt \rightarrow G$ over $K$. It therefore suffices to show that $\psi_n$ is trivial on a set of generators for $K_n$; for example, on $(U \cap K_n) \cup (\Ubar \cap K_n)$. Let $u \in U \cap K_n$; then there exists a unique $k \in K$ such that $\tilde{h}(\varpi)^{n}(u, \theta(u)) \tilde{h}(\varpi)^{-n}= (k, \psi_n(u)\theta(k))$. Since $\Tt$ normalizes $U^*$ (by \ref{lem:commutators} (2)), we must have $(k, \psi_n(u)\theta(k)) \in U^* \cap \widetilde{K}$. And since $U^* \cap \widetilde{K} = K^*$ (Remark \ref{rem:kspl}), it follows that $\psi_n(u) = 1$. The same argument (using \ref{lem:commutators} (3) this time) shows that $\psi_n(u) = 1$ for all $u \in \Ubar \cap K_n$. Hence $K^* \tilde{h}(\varpi)^{-n}K^* \cap K^* \tilde{h}(\varpi)^{-n}(1, -1)K^* = \emptyset$ for each $n \geq 0$. 

Now suppose that there exist $k_1$, $k_2 \in K$ and $n \geq 0$ such that 
$$\tilde{h}(\varpi)^{-n}(1, -1) = (\alpha k_1\alpha^{-1}, \theta'(\alpha k_1 \alpha^{-1})) \cdot \tilde{h}(\varpi)^{-n} \cdot (\alpha k_2\alpha^{-1}, \theta'(\alpha k_2 \alpha^{-1})).$$
Then 
\begin{align*}
(\tilde{\alpha})^{-1}\tilde{h}(\varpi)^{-n}(1, -1)\tilde{\alpha} & = (k_1, \theta(k_1))\cdot (\tilde{\alpha})^{-1}\tilde{h}(\varpi)^{-n}\tilde{\alpha}\cdot (k_2, \theta(k_2)),
\end{align*}
or equivalently
\begin{align*}
\tilde{h}(\varpi)^{-n}(1, -(\varpi, \varpi^n)_F) & = (k_1, \theta(k_1)) \cdot \tilde{h}(\varpi)^{-n}(1, (\varpi, \varpi^n)_F)\cdot (k_2, \theta(k_2)),
\end{align*}
which is impossible since $K^*\tilde{h}(\varpi)^{-n}K^* \cap K^*\tilde{h}(\varpi)^{-n}(1, -1)K^* = \emptyset.$ Hence $K'^*\tilde{h}(\varpi)^{-n}K'^* \cap K'^*\tilde{h}(\varpi)^{-n}(1, -1)K'^* = \emptyset$ for all $n \geq 0.$ 
\end{proof}

\subsection{Bruhat decompositions of $\Gt$ and $K^*$}
\label{subsec:bruhatdecomp}
Recall the standard and refined (respectively) Bruhat decompositions of $G$:
\begin{equation}
\label{eqn:gbruhat}
G = B \amalg B w(1) B = B \amalg B w(1) U.
\end{equation}
Lifting to $\Gt$, we have the decompositions
\begin{equation}
\label{eqn:gtbruhat}
\Gt = \Bt \amalg \Bt \tilde{w}(1)\Bt = \Bt \amalg \Bt \tilde{w}(1)U^*. 
\end{equation}
In particular, using the cocycle (\ref{eqn:sl2cocycle}) one finds 
\begin{equation}
\label{eqn:ubarbruhat}
\tilde{\bar{u}}(x\varpi) = \left(
  \left( \begin{array}{cc}
-(x\varpi)^{-1}      & -1 \\
0      & -x\varpi 
    \end{array} \right), (-1, x \varpi)_F \right) \cdot \tilde{w}(1) \cdot \tilde{u}((x \varpi)^{-1}).
\end{equation}

Let $I$ denote the inverse image in $K$ of $B(\mathfrak{k})$ under the map given by reduction modulo $\varpi$. The Bruhat decompositions of $G(\mathfrak{k})$ induce decompositions
\begin{equation}
\label{eqn:kbruhat}
K = I \amalg I w(1)I = I \amalg (U \cap I) w(1) I = I \amalg \left(\coprod_{x \in \mathfrak{k}} u \left(\left[x\right] \right)w(1)I\right),
\end{equation}
which lift to the following decompositions of $K^*$:
\begin{equation}
\label{eqn:ktbruhat}
K^* = I^* \amalg I^* \tilde{w}(1)I^* = I^* \amalg (U \cap I)^* \tilde{w}(1)I^* = I^* \amalg \left(\coprod_{x \in \mathfrak{k}}\tilde{u}\left( \left[x\right]\right)\tilde{w}(1)I^* \right). 
\end{equation}

\subsection{A system of $\Ks$-coset representatives in $\Gt$}
\label{subsec:ktcosetdecomp}
For $n \geq 0$ and $\lambda \in \cI_{2n-1}$, let 
\begin{equation}
\label{eqn:etadef}
\eta_n(\lambda) := \phi(-v_F(\lambda) - 1) \cdot \left(\lambda, \varpi^{-v_F(\lambda) - 1} \right)_F \cdot (-1, \varpi^n)_F. 
\end{equation}

For $n \geq 1$ and $\lambda \in\cI_{2n}$ let
\begin{equation}
\label{eqn:ht0def}
\tilde{h}^0_{2n, \lambda} := \tilde{h}(\varpi)^n\tilde{u}(\lambda\varpi^{-2n})(1, (-1, \varpi^n)_F),
\end{equation}
and for $n \geq 1$ and $\lambda \in\cI_{2n-1}$ let
\begin{equation}
\label{eqn:ht1def}
\tilde{h}^1_{2n-1, \lambda} = \begin{cases}
\tilde{h}(\varpi)^{-n} & \text{ if } \lambda =0,\\
\tilde{h}(\varpi)^{n - v_{F}(\lambda) - 1} \tilde{u} \left(\lambda \varpi^{-2n + 1} \right) (1, \eta_n(\lambda)) & \text{ if } \lambda \neq 0.  
\end{cases}
\end{equation}

Let $S_0 := \{ (1, 1)\}$, and for $n \geq 1$ define
\begin{align*}
S^0_n &:= \{ \tilde{h}^0_{2n, \lambda} : \, \lambda \in \cI_{2n}\},\\
S^1_{n} &:= \{ \tilde{h}^1_{2n-1, \lambda} : \, \lambda \in\cI_{2n-1}\},\\
S_n &:= S^0_n \amalg S^1_n.
\end{align*} 
Finally, let $S := \bigcup_{n \geq 0} S_n$.

\begin{lem}
\label{lem:leftksreps}
For each $n \geq 1$ we have $K^* \tilde{h}(\varpi)^{-n}K^* = \coprod_{g \in S_n} gK^* ,$ and also $\Gt = \coprod_{\substack{g \in S\\ \zeta \in \mu_2}} g(1, \zeta)K^*.$
\end{lem} 

\begin{proof}[Proof of Lemma \ref{lem:leftksreps}]

Let $n \geq 1$. Applying (\ref{eqn:ktbruhat}), we have
$$K^*\tilde{h}(\varpi)^{-n}K^* = I^* \tilde{h}(\varpi)^{-n}K^* \amalg I^* \tilde{w}(1) I^* \tilde{h}(\varpi)^{-n}K^* = I^*\tilde{h}(\varpi)^{-n}K^* \amalg \left(\coprod_{\lambda \in \mathfrak{k}}\tilde{u}\left( \left[\lambda\right]\right)\tilde{w}(1)I^*\tilde{h}(\varpi)^{-n}K^* \right).$$ 
The Iwahori decomposition $I = (I \cap \Ubar)(I \cap T)(I \cap U)$ lifts to give $I^* = (I \cap \Ubar)^* (I \cap T)^* (I \cap U)^*$, where the three factors may be taken in any order. Furthermore, for each $n \geq 0$, $\tilde{h}(\varpi)^{-n}$ normalizes $(I \cap \Ubar)^*$ while $\tilde{h}(\varpi)^n$ normalizes $(I \cap U)^*$, and $\tilde{w}(1)\tilde{h}(\varpi)^{\pm n} = \tilde{h}(\varpi)^{\mp n}(1, (-1, \varpi^n)_F)\tilde{w}(1)$. Applying these facts and calculating, we get
\begin{align*}
\coprod_{x \in \mathfrak{k}} \tilde{u}([x]) \tilde{w}(1)I^*\tilde{h}(\varpi)^{-n}K^* & = \coprod_{ x \in \mathfrak{k}} \tilde{u}([x]) \tilde{w}(1)(I \cap \Ubar)^*(I \cap T)^*(I \cap U)^*\tilde{h}(\varpi)^{-n}K^*\\
& = \coprod_{x \in \mathfrak{k}} \tilde{u}([x])\tilde{w}(1)(I \cap \Ubar)^*\tilde{h}(\varpi)^{-n}K^*\\
& = \coprod_{x \in \mathfrak{k}} \bigcup_{y \in \varpi\OF} \tilde{u}([x])\tilde{u}(-y)\tilde{w}(1)\tilde{h}(\varpi)^{-n}K^*\\
& = (I \cap U)^*\tilde{w}(1)\tilde{h}(\varpi)^{-n}K^*\\
& = (I \cap U)^*\tilde{h}(\varpi)^{n}(1, (-1, \varpi^n)_F)K^*.
\end{align*}

If $\tilde{u} \in (I \cap U)^*$ then $\tilde{u} = \tilde{u}(z)$ for some $z \in \OF$. We have $\tilde{u}(z) \tilde{h}(\varpi)^n = \tilde{h}(\varpi)^n \tilde{u}(z\varpi^{-2n})$, and if $z$, $z' \in \OF$ then $\tilde{h}(\varpi)^n\tilde{u}(z\varpi^{-2n})(1, (-1, \varpi^n)_F)K^* = \tilde{h}(\varpi)^n\tilde{u}(z'\varpi^{-2n})(1, (-1, \varpi^n)_F)K^*$ if and only if $z - z' \in \varpi^{2n}\OF$. Hence
\begin{align*}
(I \cap U)^*\tilde{h}(\varpi)^{n}(1, (-1, \varpi^n)_F)K^* & = \coprod_{\lambda \in \cI_{2n}} \tilde{h}(\varpi)^{n}\tilde{u}\left(\lambda\varpi^{-2n}\right)(1, (-1, \varpi^n)_F)K^* = \coprod_{g \in S^0_n} gK^*. 
\end{align*}
We also have
\begin{align*}
I^*\tilde{h}(\varpi)^{-n}K^* = (I \cap \Ubar)^*(I \cap T)^*(I \cap U)^* \tilde{h}(\varpi)^{-n}K^* = (I \cap \Ubar)^*\tilde{h}(\varpi)^{-n}K^*.
\end{align*}
If $\tilde{\bar{u}} \in (I \cap \Ubar)^*$ then $\tilde{\bar{u}} = \tilde{\bar{u}}(-z\varpi)$ for some $z \in \OF$. For $z, z' \in \OF$ we have $\tilde{\bar{u}}(-z\varpi) \tilde{h}(\varpi)^{-n}K^* = \tilde{\bar{u}}(-z'\varpi) \tilde{h}(\varpi)^{-n}K^*$ if and only if $z - z' \in \varpi^{2n-1}\OF$, so 
\begin{align*}
(I \cap \Ubar)^*\tilde{h}(\varpi)^{-n}K^* & = \tilde{h}(\varpi)^{-n}K^* \amalg \left( \coprod_{\lambda \in \cI_{2n-1}\setminus\{0\}} \tilde{\bar{u}}(-\lambda\varpi) \tilde{h}(\varpi)^{-n}K^*\right).
\end{align*}

For $\lambda \in \cI_{2n-1}\setminus \{0\}$, a calculation using the explicit Bruhat decomposition (\ref{eqn:ubarbruhat}) of $\tilde{\bar{u}}(-\lambda\varpi)$ implies that
\begin{align*}
\tilde{\bar{u}}(-\lambda\varpi) \tilde{h}(\varpi)^{-n} & = \left(
  \left( \begin{array}{cc}
(\lambda\varpi)^{-1}      & -1 \\
0      & \lambda\varpi 
    \end{array} \right), (-1, \lambda\varpi)_F \right)\tilde{w}(1)\tilde{u}(-(\lambda\varpi)^{-1})\tilde{h}(\varpi)^{-n}\\
& = \left(
  \left( \begin{array}{cc}
(\lambda\varpi)^{-1}      & -1 \\
0      &  \lambda\varpi
    \end{array} \right), (-1, \lambda\varpi^{n+1})_F \right)\tilde{h}(\varpi)^n\tilde{w}(1)\tilde{u}(-\lambda^{-1}\varpi^{2n-1}).
\end{align*}
Further manipulation of the expression just above gives
\begin{align*}
\tilde{\bar{u}}(-\lambda\varpi) \tilde{h}(\varpi)^{-n} & \in \tilde{h}^1_{2n-1, \underline{\lambda}}K^*,
\end{align*}
where $\underline{\lambda}$ denotes the unique element of $\cI_{2n-1} \setminus \{0\}$ such that $\underline{\lambda} \equiv -\left( \frac{\lambda}{\varpi^{v_F(\lambda)}}\right)^{-1} \varpi^{v_F(\lambda)} \pmod{\varpi^{2n-1}}.$ It follows that
$$I^*\tilde{h}(\varpi)^{-n}K^* = \tilde{h}(\varpi)^{-n}K^* \coprod_{\lambda \in \cI_{2n-1}\setminus\{0\}} \tilde{h}^1_{2n-1, \underline{\lambda}} K^*.$$
Since $\lambda \mapsto \underline{\lambda}$ is a (valuation-preserving) bijection $\cI_{2n-1}\setminus \{0\}\rightarrow \cI_{2n-1}\setminus \{0\}$, we obtain
$$I^*\tilde{h}(\varpi)^{-n}K^* = \coprod_{g \in S^1_n} gK^*.$$
Hence for each $n \geq 1$ we have $K^*\tilde{h}(\varpi)^{-n}K^* = \coprod_{g \in S_n} gK^*$, and applying the Cartan decomposition of $\Gt$ we get $\Gt = \coprod_{\substack{g \in S\\ \zeta \in \mu_2}} g(1, \zeta)K^*$.   
\end{proof}

\begin{rem}
\label{rem:poscartan}
Since $\tilde{w}(\pm 1) \in K^*$ and $\tilde{w}(1)\tilde{h}(\varpi)^n\tilde{w}(-1) = (1, (-1, \varpi^n)_F)\tilde{h}(\varpi)^{-n}$ for each $n \in \ZZ$, we have $K^*\tilde{h}(\varpi)^{-n}K^* = K^*\tilde{h}(\varpi)^n(1, (-1, \varpi^n)_F)K^*$. Hence $\Gt$ also has the Cartan decompositions 
\begin{align*}
\Gt & = \coprod_{\substack{n \geq 0\\ \zeta \in \mu_2}} K^* \tilde{h}(\varpi)^n(1, \zeta)K^* = \coprod_{\substack{n \geq 0 \\ \zeta \in \mu_2}}K'^*\tilde{h}(\varpi)^{n}(1, \zeta) K'^*,
\end{align*}
and for $n \geq 0$, 
\begin{align*}
K^* \tilde{h}(\varpi)^nK^* & = \coprod_{g \in S_n}(1, (-1, \varpi^n)_F ) gK^*. 
\end{align*}
\end{rem}

\subsection{Comparison of Cartan and Iwasawa decompositions in $\Gt$}
\label{subsec:cartaniwasawa}
The Iwasawa decomposition $G = BK$ lifts to the decomposition $\Gt = \Bt K^*$. For later applications, in particular the derivation of an explicit formula for a Satake-type transform (Proposition \ref{prop:gensataketransformi}), we need to compare this Iwasawa decomposition of $\Gt$ with the Cartan decomposition with respect to $\Ks$ given in Lemma \ref{lem:gtkscartandecomp}. 

\begin{rem}
\label{rem:ciintro}
In this section we work with $\ZZ$-coefficients rather than with $E$-coefficients as elsewhere in this paper. Although in the proof of Proposition \ref{prop:gensataketransformi} we only take advantage of the reduction mod $q$ of the calculations in Lemmas \ref{lem:cartaniwasawa} and \ref{lem:citranspose}, the lemmas may also be useful for determining the image of the Satake transform for $\Gt$ when the coefficient field has characteristic 0 or dividing $q-1$. 
\end{rem}

\begin{lem} \label{lem:cartaniwasawa}
Let $n \geq 0$, $m \in \ZZ$, $\zeta \in \mu_2$, and let 
$$S_{n, m, \zeta} : = \{\tilde{u} \in U^*/(U \cap K)^*: \, \tilde{h}(\varpi)^m\tilde{u}K^* \subset K^*\tilde{h}(\varpi)^{-n}K^*\} .$$ Then 
$$|S_{n, m, \zeta}| =
\begin{cases}
1 & \text{ if } m = -n \text{ and } \zeta = 1,\\

q^{n+m-1}(q-1) & \text{ if } -n < m < n \text{ and } 2 \big \vert (m - n) \text{ and } \zeta = (-1, \varpi)_F^{\frac{n+m}{2}} \\
q^{n+m-1}\left(\frac{q-1}{2} \right) & \text{ if } -n < m < n \text{ and } 2 \not \big \vert (m - n),\\
q^{2n} & \text{ if } m = n \text{ and } \zeta = (-1, \varpi^n)_F,\\
0 & \text{ otherwise.} 
\end{cases}$$

\end{lem}

\begin{proof}[Proof of Lemma \ref{lem:cartaniwasawa}]
First we suppose that $|m| > n$ and show that $|S_{n, m, \zeta}| = 0$.  Suppose that $|S_{n, m, \zeta}| > 0$; then there exists some $\tilde{u} \in U^*/(U \cap K)^*$ such that $\tilde{h}(\varpi)^{m} \tilde{u}K^* \subset K^* \tilde{h}(\varpi)^{-n}(1, \zeta)K^*$. Let $u = Pr(\tilde{u})$ and view $u$ as a representative of $U/(U \cap K)$; then $h(\varpi)^m u K \subset Kh(\varpi)^{-n}K$. But by Theorem 2.6.11(3) of \cite{macdonald:sphericalfns}, the set $\{u \in U/(U \cap K): h(\varpi)^{m}uK \subset Kh(\varpi)^{-n}K\}$ is empty if $|m| > n$. 

Consequently $|S_{0, m, \zeta}| \neq 0$ only if $m = 0$, and we have $|S_{0, 0, 1}| =| \{\tilde{u} \in U^*/(U \cap K)^*: \, \tilde{u} \in K^*\}| = 1$ while $|S_{0, 0, -1}| = 0.$

Now suppose that $n \geq 1$. By Lemma \ref{lem:leftksreps}, 
$$K^*\tilde{h}(\varpi)^{-n}(1, \zeta)K^* = \left(\coprod_{\lambda \in\cI_{2n}} \tilde{h}^{0}_{2n, \lambda}(1, \zeta)K^*\right)
\amalg \left(\coprod_{\lambda \in\cI_{2n-1}} \tilde{h}^1_{2n-1, \lambda}(1, \zeta)K^* \right).$$

It is clear from the definitions (\ref{eqn:ht0def}, \ref{eqn:ht1def}) that for $\lambda \in\cI_{2n}$ we have $\tilde{h}^0_{2n, \lambda} \in (1, (-1, \varpi^n)_F)\tilde{h}(\varpi)^nU^*$, while for $\lambda \in\cI_{2n-1}$ we have $\tilde{h}^1_{2n-1, \lambda} \in (1, \eta_n(\lambda))\tilde{h}(\varpi)^{n - v_F(\lambda)- 1}U^*$ if $\lambda \neq 0$ and $\tilde{h}^1_{2n-1, 0} \in \tilde{h}(\varpi)^{-n}U^*$. Since $S$ is a complete set of representatives for $\Gt/\Kt$, we have the following facts: if $\lambda$ and $\lambda'$ are distinct elements of $\cI_{2n}$, then $\tilde{h}(\varpi)^{-n}\tilde{h}^0_{2n, \lambda}\Kt \neq \tilde{h}(\varpi)^{-n}\tilde{h}^0_{2n, \lambda'}\Kt$, and if $\lambda$ and $\lambda'$ are distinct elements of $\cI_{2n-1}$ such that $0 \leq v_F(\lambda) = v_F(\lambda') \leq 2n-2$ then $\tilde{h}(\varpi)^{-n +v_F(\lambda) + 1}\tilde{h}^1_{2n-1, \lambda}\Kt \neq \tilde{h}(\varpi)^{-n +v_F(\lambda) + 1}\tilde{h}^1_{2n-1, \lambda'}\Kt$. In particular, $\tilde{h}(\varpi)^{-n}\tilde{h}^0_{2n, \lambda}(\widetilde{U} \cap \widetilde{K}) \neq \tilde{h}(\varpi)^{-n}\tilde{h}^0_{2n, \lambda'}(\widetilde{U} \cap \widetilde{K})$ if $\lambda$ and $\lambda'$ are distinct elements of $\cI_{2n}$, and $\tilde{h}(\varpi)^{-n + v_F(\lambda) + 1}\tilde{h}^1_{2n-1, \lambda}(\widetilde{U} \cap \widetilde{K}) \neq \tilde{h}(\varpi)^{-n + v_F(\lambda) + 1}\tilde{h}^1_{2n-1, \lambda'}(\widetilde{U} \cap \widetilde{K})$ if $\lambda$ and $\lambda'$ are distinct elements of $\cI_{2n-1}$ such that $0 \leq v_F(\lambda) = v_F(\lambda') \leq 2n-2$. Thus 
\begin{align*}
\sum_{\zeta \in \mu_2}|S_{n, n, \zeta}| &= |S^0_n| = |\cI_{2n}| = q^{2n},\\
\sum_{\zeta \in \mu_2}|S_{n, m, \zeta}| & = |\{\lambda \in \cI_{2n-1}: \, v_F(\lambda) = n - m - 1\}| = q^{n + m - 1} (q-1) \text{ if } -n < m < n,\\
\sum_{\zeta \in \mu_2}|S_{n, -n, \zeta}| & = |\{\tilde{h}^1_{2n-1,0}\}| = 1.
\end{align*}

Next, considering the sign of $\tilde{h}^0_{2n, \lambda}$ modulo $\tilde{h}(\varpi)^nU^*$ (respectively, the sign of $\tilde{h}^1_{2n-1, 0}$ modulo $\tilde{h}(\varpi)^{-n}U^*$), we easily see that
\begin{align*}
|S_{n, n, \zeta}| & = \begin{cases}
q^{2n} &\text{ if } \zeta = (-1, \varpi^n)_F,\\
0 & \text{ otherwise};
\end{cases}
\end{align*}
\begin{align*}
|S_{n, -n, \zeta}| & = \begin{cases}
1 &\text{ if } \zeta = 1, \\
0 & \text{ otherwise}.
\end{cases}
\end{align*}

Now let $-n < m < n$ and let $\lambda \in\cI_{2n-1}$ such that $v_F(\lambda) = n - m - 1$; then $\tilde{h}^1_{2n-1, \lambda} \subset (1, \eta_n(\lambda))\tilde{h}(\varpi)^{m}U^*$. Hence 
$$|S_{n, m, \zeta}| = |\{\lambda \in\cI_{2n-1}:\, v_F(\lambda) = n - m - 1 \text{ and } \eta_n(\lambda) = \zeta\}|.$$
Fix $n\geq 1$, $-n < m < n$, and $\zeta \in \mu_2$, and consider $\lambda \in\cI_{2n-1}$ such that $v_F(\lambda) = n - m - 1$. From the definition (\ref{eqn:etadef}) and a short calculation with the Hilbert symbols, we have $\eta_n(\lambda) = \zeta$ if and only if 
\begin{equation}
\label{eqn:etacond}
\left(\lambda\varpi^{m-n+1}, \varpi^{m - n}\right)_F = \zeta \cdot (-1, \varpi^n)_F\cdot \phi(m - n). 
\end{equation} 
If $2 \big \vert (m - n)$ then the left-hand side of (\ref{eqn:etacond}) is always equal to 1, so in that case $\eta_n(\lambda) = \zeta$ if and only if $\zeta = (-1, \varpi^n)_F \cdot \phi(m - n)$. From (1) of Remark \ref{rem:phiprops}, when $2 \big \vert (m - n)$ we have $(-1, \varpi^n)_F \cdot \phi(m - n) = (-1, \varpi^n)_F \cdot (-1, \varpi)_F^{\frac{m - n}{2}} = (-1, \varpi)_F^{\frac{n + m}{2}}.$

Suppose that $-n < m < n$ and $2 \not \big \vert (m - n)$. Then the left-hand side of (\ref{eqn:etacond}) is equal to $\left(\lambda\varpi^{m-n+1}, \varpi\right)_F$, which by (6) of Remark \ref{rem:hilbsymprops} is equal to 1 for exactly $q^{n + m - 1}\left(\frac{q-1}{2}\right)$ of the $q^{n + m - 1}(q-1)$ elements in the set $\{\lambda \in \cI_{2n-1}: \, v_F(\lambda) = n - m - 1\}$. Therefore if $-n< m < n$ and $2 \not \big \vert(m - n)$, then $|S_{n, m, \zeta}| = q^{n + m - 1} \left(\frac{q-1}{2} \right)$ for each $\zeta \in \mu_2$. 
\end{proof}

In applications (e.g. Proposition \ref{prop:satakeproperties}) it will also be convenient to have a comparison of the Cartan decomposition $\Gt = \amalg_{n \geq 0}\Ks \tilde{h}(\varpi)^{-n} \Ks$ with the ``opposite'' Iwasawa decomposition $\Gt = \Bbt \Ks$. 

The transpose operation on $G$ gives a bijection between the sets 
$\{u \in U/(\UK): \, h(\varpi)^muK \subset Kh(\varpi)^{-n}K\}$ and $ \{ \bar{u} \in (\Ubar \cap K) \setminus \Ubar: \, K\bar{u}h(\varpi)^m \subset Kh(\varpi)^{-n}K\}$. The following lemma gives the analogous statement for $\Gt$.

\begin{lem} \label{lem:citranspose}
Let $n \geq 0$, $m \in \ZZ$, and $\zeta \in \mu_2$. Then 
$$\#\{ \tilde{\bar{u}} \in (\Ubar \cap K)^* \setminus \Ubar^*: \, K^*\tilde{\bar{u}}\tilde{h}(\varpi)^m \subset K^*\tilde{h}(\varpi)^{-n}(1, \zeta)K^*\} = |S_{n, m, \zeta}|.$$
\end{lem}

\begin{proof}[Proof of Lemma \ref{lem:citranspose}]
Since $\Gt$ is generated by $U^*$ and $\Ubar^*$ (Lemma \ref{lem:commutators} (3)), we may define a transpose operation on $\Gt$ by setting
${}^{t}(u, 1) := ({}^tu, 1)$ and ${}^t(\bar{u}, 1) := ({}^t\bar{u}, 1)$ for $u \in U$ and $\bar{u} \in \Ubar$, and then defining 
$${}^{t}(g, \zeta) : = {}^t(u_{j}, 1)\cdot \cdot \cdot{}^t(u_1, 1)$$  
where $u_i \in U \cup \Ubar$ and $(g, \zeta) = \prod_{i = 1}^{i = j} (u_i, 1)$. Then ${}^t(g_1g_2) = {}^tg_2{}^tg_1$ for all $g_1$ and $g_2 \in \Gt$, ${}^tk \in K^*$ for all $k \in K^*$, and ${}^t\tilde{h}(a) = \tilde{h}(a)(1, (-1, a)_F)$ for all $a \in F^\times$. 

Two elements $\tilde{u}$, $\tilde{u}' \in U^*$ represent distinct cosets in $U^*/(U \cap K)^*$ if and only if ${}^t\tilde{u}, \, {}^t\tilde{u}' \in \Ubar^*$ represent distinct cosets in $(\Ubar \cap K)^* \setminus \Ubar^*$. Furthermore we have  $\tilde{u} \in U^*$ and $\tilde{h}(\varpi)^m\tilde{u}K^* \subset K^* \tilde{h}(\varpi)^{-n}(1, \zeta)K^*$ if and only if ${}^t\tilde{u} \in \Ubar^*$ and $K^* ({}^t\tilde{u}) \tilde{h}(\varpi)^m(1, (-1, \varpi^m)_F) \subset K^*\tilde{h}(\varpi)^{-n}(1, (-1, \varpi^{n})_F)K^*$. Hence the transpose operation on $\Gt$ gives a bijection between $\{\tilde{u} \in U^* / (U \cap K)^*: \, \tilde{h}(\varpi)^m\tilde{u}K^* \subset K^* \tilde{h}(\varpi)^{-n}K^*\}$ and $\{\tilde{\bar{u}} \in (\Ubar \cap K)^* \setminus \Ubar^*: \, K^* \tilde{\bar{u}}\tilde{h}(\varpi)^m \subset K^* \tilde{h}(\varpi)^{-n}(1, (-1, \varpi^{n+m})_F)K^*\}.$ Lemma \ref{lem:citranspose} then follows from Lemma \ref{lem:cartaniwasawa} and the fact that $(-1, \varpi^{n+m})_F = 1$ whenever $2 \big \vert (m - n).$ 
\end{proof}

\section{Weights}
\label{sec:weights}

\subsection{Weights of $\Kt$} 
\label{subsec:ktweights}
A \textit{weight of} $\Kt$ is a smooth, genuine, irreducible representation of $\Kt$ on an $E$-vector space. The extension which defines the metaplectic cover of $G$ is split over $K$, so the classification of weights of $\Kt$ reduces to the classification of weights of $K$, i.e., of smooth irreducible $E$-representations of $K$. Recall that $K^*$ denotes the image of the map $k \mapsto (k, \theta(k))$ (defined in (\ref{eqn:thetadef})) which uniquely splits the extension over $K$, that $\epsilon$ denotes the embedding $\mu_2(F) \rightarrow E^\times$, and that $f$ is the residual degree of $F/\QQ_p$.

\begin{prop} \label{prop:ktweights}
For each vector $\vec{r} = (r_0, \dots, r_{f-1}) \in \{0, \dots, p-1\}^f$, let $\sigr$ denote the inflation to $K$ of the following $E$-representation of $G(\mathfrak{k})$, likewise denoted by $\sigr$:
$$\sigr = \bigotimes_{i = 0}^{f-1}(Sym^{r_i}E^2)^{Fr^i}$$
where $Fr$ is the Frobenius map $x \mapsto x^p$. 
\begin{enumerate}
\item Any smooth irreducible representation of $K^*$ is isomorphic to $\sigr$ for exactly one $\vec{r} \in \{0, \dots, p-1\}^{f},$ and the weights for $\Kt$ are exactly the representations
$$\sigrt := \sigr \otimes \epsilon, \, \, \vec{r} \in \{0, \dots, p-1 \}^f.$$
\item Let $\rho \neq 0$ be a smooth, genuine representation of $\Gt$. Then for any weight $\sigrt$ of $\Kt$,
$$\text{dim}_E \, \text{Hom}_{\Kt}(\sigrt, \rho\big \vert_{\Kt}) = \text{dim}_E \, \text{Hom}_{\Ks}(\sigr, \rho \big \vert_{\Ks})$$
(where $\sigr$ is viewed as a representation of $K^*$ on the right-hand side). In particular, there is a weight $\sigrt$ of $\Kt$ such that 
$$\text{dim}_E \, \text{Hom}_{\Kt}(\sigrt, \rho \big \vert_{\Kt}) \geq 1.$$
\end{enumerate}
\end{prop}

\begin{proof}[Proof of Proposition \ref{prop:ktweights}]
\begin{enumerate}
\item Let $\pi$ be a weight of $\Kt \cong K^* \times \mu_2$. Then $\pi \big \vert_{\Ks}$ is a smooth irreducible representation of $K^* \cong K$. By the classification of weights of $K$ (a reference is \cite{abdellatif:thesis} Lemme 3.5.1), we have $\pi \big \vert_{\Ks \times \{1\}} \cong \sigr$ for a unique $\vec{r} \in \{0, \dots, p-1\}^f.$ Since $\pi$ is genuine, the restriction $\pi \big \vert_{\{1\} \times \mu_2}$ is nontrivial. Thus $\pi \cong \sigr \otimes \epsilon$ for a unique $r \in \{0, \dots, p-1\}^f$. Conversely, given $\vec{r} \in \{0, \dots, p-1\}^f$, the inflation of $\sigr$ to $K^*$ is smooth and irreducible. The product $\sigr \otimes \epsilon$ is likewise smooth and irreducible as a representation of $\Kt$, and moreover is genuine, so $\sigrt = \sigr \otimes \epsilon$ is a weight of $\Kt$. 

\item Every $E$-represention of $K$ contains a weight of $K$, so, identifying $K$ with $\Ks$,
$$\text{dim}_E\,  \text{Hom}_{\Ks}(\sigr, \rho\big \vert_{\Ks}) \geq 1$$
for some weight $\sigr$ of $K$. A map of $\Ks$-representations has a unique extension to a map of genuine $\Kt$ representations, so
$$\text{dim}_E \, \text{Hom}_{\Kt}(\sigrt, \rho \big \vert_{\Kt}) \geq 1.$$
\end{enumerate}
\end{proof}

From now on we parametrize the weights of $\Kt$ by $\vec{r} \in \{0, \dots, p-1 \}^f.$ The underlying vector space of $\sigrt$ (equivalently, of $\sigr$) will be denoted by $V_{\vec{r}}$. 

Proposition \ref{prop:ktweights} (2) provides a convenient irreducibility criterion. The following is a standard argument in mod $p$ representation theory:

\begin{prop} \label{prop:ktwtsignif} 
Suppose that $\rho \neq 0$ is a smooth, genuine representation of $\Gt$, and let $\tilde{\sigma}_{\vec{r}}$ be a weight of $\Kt$ contained in $\rho$. If $$\text{dim}_E \, \Hom_{\Kt}(\tilde{\sigma}_{\vec{s}}, \rho\big \vert_{\Kt}) = \begin{cases}
1 & \text{ if } \vec{s} = \vec{r},\\
0 & \text{ otherwise, }
\end{cases}$$ 
then the image of the inclusion $\tilde{\sigma}_{\vec{r}} \hookrightarrow \rho$ generates an irreducible $\Gt$-subrepresentation of $\rho$. 
\end{prop}

\begin{proof}[Proof of Proposition \ref{prop:ktwtsignif}]
Let $\Gt \cdot \tilde{\sigma}_{\vec{r}}$ denote the $\Gt$-subrepresentation of $\rho$ generated by the image of $\tilde{\sigma}_{\vec{r}}$. Suppose that $\pi \subset \Gt \cdot \tilde{\sigma}_{\vec{r}}$ is again a $\Gt$-subrepresentation. Then $\pi$ is a smooth genuine $\Gt$-representation, so by Proposition \ref{prop:ktweights}, $\pi$ contains a $\Kt$-weight. This $\Kt$-weight must be $\sigrt$, since it is contained in $\rho$. Since also $\text{dim}_E\, \Hom_{\Kt}(\tilde{\sigma}_{\vec{r}}, \rho \big \vert_{\Kt}) = 1$, the $\Gt$-subrepresentation of $\pi$ generated by $\sigrt$ is equal to $\Gt \cdot \tilde{\sigma}_{\vec{r}}$. Thus $\Gt\cdot \tilde{\sigma}_{\vec{r}}$ is irreducible. 
\end{proof}

\subsection{Weights of $\Kt'$}
\label{subsec:kt'weights}
Recall that $\Kt' = \tilde{\alpha}\Kt\tilde{\alpha}^{-1}$ where $\tilde{\alpha}$ is a lift to $\widetilde{GL}_2(F)$ of $\alpha = 
\left( \begin{array}{cc}
1    & 0 \\
0    & \varpi 
  \end{array} \right),$ and that the extension defining $\Gt$ is split over $\Kt'$. We define a \textit{weight of} $\Kt'$ to be a smooth, irreducible, genuine representation of $\Kt'$. The classification of weights of $\Kt'$ (in Lemma \ref{lem:kt'weights} below) will reduce to that of the weights of $\Kt$, following the argument used in \cite{abdellatif:thesis} Cor. 3.5.2 to classify weights of $K'$ in terms of weights of $K$. 

Given a representation $\pi$ of a subgroup $H$ of $\widetilde{GL}_2(F)$ and an element $g \in \widetilde{GL}_2(F)$, we define a conjugate representation $\pi^g$ of $gHg^{-1}$ by
$$\pi^g(h) = \pi(g^{-1}hg)$$
for $h \in H$. 

\begin{lem} \label{lem:kt'weights}
Let $\tilde{\alpha}$ be any lift of $\alpha =
  \left( \begin{array}{cc}
1      & 0 \\
0      & \varpi 
    \end{array} \right)$ to $\widetilde{GL}_2(F)$. 
\begin{enumerate}
\item Any weight of $\Kt'$ is isomorphic to $(\sigrt)^{\tilde{\alpha}}$ for a unique vector $\vec{r} \in \{ 0, \dots, p-1\}^{f}$. (The conjugate representation does not depend on the choice of lift of $\alpha$.)
\item Let $\rho \neq 0$ be a smooth, genuine representation of $\Gt$. Then for any weight $\sigrt^{\tilde{\alpha}}$ of $\Kt'$, 
$$\Hom_{\Kt'}(\sigrt^{\tilde{\alpha}},\, \rho \big \vert_{\Kt'}) =  \Hom_{\Kt}\left(\sigrt,\, \rho^{(\tilde{\alpha})^{-1}}\big \vert_{\Kt}\right).$$
In particular, there is at least one weight $\sigrt^{\tilde{\alpha}}$ of $\Kt'$ such that 
$$\text{dim}_E \, \Hom_{\Kt'}(\sigrt^{\tilde{\alpha}}, \rho \big \vert_{\Kt'}) \geq 1.$$
\end{enumerate}
\end{lem}

\begin{proof}[Proof of Lemma \ref{lem:kt'weights}]
\begin{enumerate}
\item Let $\pi$ be a weight of $\Kt'$. Then $\pi^{(\tilde{\alpha})^{-1}}$ is a weight of $\Kt$, so there is a unique vector $\vec{r} \in \{0, \dots, p-1\}^f$ such that $\pi^{(\tilde{\alpha})^{-1}} \cong \sigrt$.
Then $\pi = (\pi^{(\tilde{\alpha})^{-1}})^{\tilde{\alpha}} \cong (\sigrt)^{\tilde{\alpha}}.$
\item From the definition of the conjugate representations it follows that each element of $\Hom_{\Kt'}(\sigrt^{\tilde{\alpha}},\, \rho\big \vert_{\Kt'})$ belongs to $\Hom_{\Kt}\left( \sigrt, \, \rho^{(\tilde{\alpha})^{-1}} \big \vert_{\Kt}\right)$, and vice versa. By Proposition \ref{prop:ktweights} (2) the $\Gt$-representation $\rho^{(\tilde{\alpha})^{-1}}$ contains some weight $\sigrt$ of $\Kt$, so $\rho$ contains the weight $\sigrt^{\tilde{\alpha}}$ of $\Kt'$. 
\end{enumerate} 
\end{proof}

\subsection{$\Ubark$-coinvariants of $\Kt$- and $\Kt'$-weights}

Let $\Uk$ denote the upper-triangular unipotent subgroup of $G(\mathfrak{k})$, and let $\sigrt$ be a weight of $\Kt$. The restriction of $\sigrt$ to $U^* \cap \Kt$ factors through the composition 
\begin{align*}
\UKt& \longrightarrow \UK  \longrightarrow^{red} \Uk\\
\, \, \hspace{.2cm}  (u(x), 1) \hspace{.2cm}\,  & \longmapsto  \hspace{.3cm} u(x)\, \,  \longmapsto \, \, u(\bar{x}),
\end{align*}
so we may view $\sigrt \big \vert_{(U \cap K)^*}$ as a representation of $\Uk$. The same is true if $U$ is replaced with $\overline{U}$, so we may likewise view $\sigrt\big \vert_{(\Ubar \cap K)^*}$ as a representation of $\Ubark$. Let $(\sigrt)^{\Uk}$ denote the subspace of $\Uk$-invariants in $\Vsigt$, let $(\sigrt)_{\Ubark}$ denote the $\Ubark$-coinvariants of $\sigrt$, and let $p_{\Ubark}$ denote the projection $\Vsigt \rightarrow (\sigrt)_{\Ubark}$. 

As an $E$-vector space, $(\sigrt)^{\Uk}$ is equal to the invariants of $\Vsigt$ by the action of $\Uk$ under $\sigma_{\vec{r}}$ via inflation through $red$ alone; this $\Uk$-invariant space is well-known to be the one-dimensional highest weight space of $\sigma_{\vec{r}}$. Furthermore, $(\sigma_{\vec{r}})^{\Ubark}$ is stable by $T \cap K$, and is isomorphic to (the inflation to $T \cap K$ of) the character $\delta_{\vec{r}}$ of $\OF^\times$. The $\Ubark$-coinvariants $(\sigma_{\vec{r}})_{\Ubark}$ likewise carry a $T \cap K$-representation isomorphic to $\delta_{\vec{r}}$. Let $\tilde{\delta}_{\vec{r}}$ denote the genuine character $\delta_{\vec{r}} \otimes \epsilon$ of $\TKt$. The following lemma is a consequence of the preceding comment together with well-known facts for weights of $K$ (cf. \cite{herzig:modpsatake} Lemma 2.5 for (2) in the general setting of mod $p$ representations of unramified reductive groups). 

\begin{lem} \label{lem:ukinvariants}
\begin{enumerate} 
\item $(\sigrt)^{\Uk}$ is generated as an $E$-vector space by the highest-weight vector of $\sigma_{\vec{r}}$, and $$(\sigrt)^{\Uk} \cong \tilde{\delta}_{\vec{r}}$$ as $\TKt$-representations.
\item The composition $$j_{\vec{r}}: (\sigrt)^{\Uk} \hookrightarrow^{i^{\Uk}_{\vec{r}}} \sigrt \twoheadrightarrow (\sigrt)_{\Ubark}$$ is an isomorphism of
 $\TKt$-representations. 
\end{enumerate}
\end{lem}

Since $\delta_{\vec{r}} \cong \delta_{\vec{s}}$ if and only if $\sum_{i = 0}^{f-1}r_ip^i \equiv \sum_{i = 0}^{f-1}s_ip^i \pmod{q-1}$, we obtain:

\begin{cor} \label{cor:ukisoms}
Let $\sigrt$ and $\sigst$ be two weights of $\Kt$. There exists a $\Tt \cap \Kt$-linear isomorphism 
$$(\sigrt)_{\Ubark} \longrightarrow (\sigst)_{\Ubark}$$
if and only if $\vec{r} = \vec{s}$ or $\{\vec{r}, \vec{s}\} = \{\vec{0}, \overrightarrow{p-1}\}.$
\end{cor}

Now consider the $\Kt'$-weight $\sigrt^{\tilde{\alpha}}$. The subgroup $(U^* \cap K^*)^{\tilde{\alpha}}$ of $\Kt'$ consists of the elements $(\tilde{\alpha})(u(x), 1)\tilde{\alpha}^{-1} = (u(\varpi^{-1}x), (-1, \varpi)_F)$ such that $x \in \OF$, and $\sigrt^{\tilde{\alpha}}(u(\varpi^{-1}x), 1) = \sigrt(u(x), 1)$. Hence the restriction $\sigrt^{\tilde{\alpha}} \big \vert_{(\Ubar^* \cap K^*)^{\tilde{\alpha}}}$ is the inflation of a representation of $\Uk$ through the composition 
\begin{align*}
(U^* \cap K^*)^{\tilde{\alpha}} &\longrightarrow U \cap K \longrightarrow^{red} \Uk\\
(u(\varpi^{-1} x), 1) & \, \,\hspace{.1cm} \mapsto u(x)\, \,\hspace{.1cm}  \mapsto u(\bar{x}),
\end{align*}
and we write $(\sigrt^{\tilde{\alpha}})^{\Uk}$ for $(\sigrt^{\tilde{\alpha}})^{(U^* \cap K^*)^{\tilde{\alpha}}}$. Likewise, the restriction $\sigrt^{\tilde{\alpha}} \big \vert_{(\Ubar^*\cap K^*)^{\tilde{\alpha}}}$ is the inflation of a representation of $\Ubark$ through the composition
\begin{align*}
(\Ubar^* \cap K^*)^{\tilde{\alpha}} &\longrightarrow \Ubar \cap K \longrightarrow^{red} \Ubark\\
(\bar{u}(\varpi x), 1) & \, \,\hspace{.1cm} \mapsto \bar{u}(x)\, \,\hspace{.1cm}  \mapsto \bar{u}(\bar{x}),
\end{align*} 
and we denote $(\sigrt^{\tilde{\alpha}})_{(\Ubar^* \cap K^*)^{\tilde{\alpha}}}$ by $(\sigrt^{\tilde{\alpha}})_{\Ubark}.$ 

\begin{lem} \label{lem:kt'coinvs}
\begin{enumerate}
\item $(\sigrt^{\tilde{\alpha}})^{\Uk}$ is generated as an $E$-vector space by the highest-weight vector of $\sigma_{\vec{r}}^{\alpha}$.
\item If $\vec{r'} \in \{0, \dots, p-1\}^f$ then there is an isomorphism of $\Tt \cap \Kt$-representations
$$(\sigrt^{\tilde{\alpha}})^{\Uk} \cong \tilde{\delta}_{\vec{r'}}$$  
if and only if $\vec{r'}$ satisfies $$\sum_{i = 0}^{f-1}r'_ip^i \equiv \sum_{i = 0}^{f-1} \left(r_i + \frac{p-1}{2} \right) \pmod{q - 1}.$$
\item The composition 
$$j_{\vec{r}}^{\tilde{\alpha}}: (\sigrt^{\tilde{\alpha}})^{\Uk} \hookrightarrow\sigrt^{\tilde{\alpha}} \twoheadrightarrow (\sigrt^{\tilde{\alpha}})_{\Ubark}$$
is an isomorphism of $\Tt \cap \Kt$-representations. 
\end{enumerate}
\end{lem}

\begin{proof}[Proof of Lemma \ref{lem:kt'coinvs}]
All three statements follow from Lemma \ref{lem:ukinvariants} under conjugation by $\tilde{\alpha}$. The calculation for (2) is the following: let $a \in \OF^\times$ and $\zeta \in \mu_2(F)$, so that $(t(a), \zeta) \in \Tt \cap \Kt$, and let $v \in (\sigrt^{\tilde{\alpha}})^{\Uk}$. Then 
\begin{align*}
\sigrt^{\tilde{\alpha}}(t(a), \zeta)v &= \sigrt\left( (\tilde{\alpha})^{-1}(t(a), \zeta)\tilde{\alpha}\right)v \\
& = \sigrt\left((t(a), \zeta \cdot (a, \varpi)_F)\right)v\\
& = \zeta \cdot (a, \varpi)_F \cdot \delta_{\vec{r}}(a)v.
\end{align*}
By (\ref{eqn:tamesymbol}), the character $a \mapsto (a, \varpi)_F$ of $\OF^\times$ is equal to $a \mapsto \overline{(a^{-1})}^{\frac{q-1}{2}}$, which in terms of our parametrization of smooth characters of $\OF^\times$ is equal to $\delta_{\overrightarrow{\frac{p-1}{2}}}$. Hence
$$\sigrt^{\tilde{\alpha}}(t(a), \zeta)v = \zeta \cdot \delta_{\overrightarrow{\frac{p-1}{2}}}(a) \cdot \delta_{\vec{r}}(a)v.$$
We have $\delta_{\overrightarrow{\frac{p-1}{2}}}(a) \cdot \delta_{\vec{r}}(a) = \delta_{\overrightarrow{r'}}(a)$ for all $a \in \OF^\times$ if and only if $\vec{r'}$ satisfies the condition given in (2). 
\end{proof}

Finally we state a lemma which will be useful in the definition and application of a Satake transform (Proposition \ref{prop:gensataketransformi}). It is the appropriate analogue of the fact that, in the context of representations on $\CC$-vector spaces, the Jacquet functor is left adjoint to parabolic induction. The proof is taken with only minor adaptations from notes of a course by Herzig (\cite{herzig:glnnotes}, Lemma 26), where a similar statement is proven for mod $p$ representations of $GL_n(F)$. 

\begin{lem} \label{lem:natlisom}
\begin{enumerate}
\item Let $\sigrt$ be any weight of $\Kt$. There is a natural isomorphism 
$$\Hom_{\Gt}\left(\ind_{\Kt}^{\Gt}\sigrt, \, \Ind_{\Bbt}^{\Gt}(-)\right) \cong \Hom_{\Tt}\left(\ind_{\Tt \cap \Kt}^{\Tt}\left((\sigrt)_{\Ubark}\right),\, -\right)$$
of functors from the category of smooth genuine $\Tt$-representations to the category of $E$-vector spaces.

\item Given a smooth genuine $\Tt$-representation $\pi$ and $f \in \Hom_{\Gt}(\ind_{\Kt}^{\Gt}\sigrt, \, \Ind_{\Bbt}^{\Gt}(\pi)),$ let $f_{\Tt}$ denote the image of $f$ in $\Hom_{\Tt}(\ind_{\Tt \cap \Kt}^{\Tt}((\sigrt)_{\Ubark}),\, \pi)$. Let $f'$ denote the element of $\Hom_{\Kt}(\sigrt, \, \Ind_{\Bbt}^{\Gt}(\pi) \big \vert_{\Kt})$ which corresponds to $f$ by Frobenius reciprocity, and let $(f_{\Tt})'$ denote the element of $\Hom_{\TKt}((\sigrt)_{\Ubark},\, \pi \big \vert_{\TKt})$ which corresponds to $f_{\Tt}$ by Frobenius reciprocity. Then for every $v \in \Vsigt$, 
$$f'(v)(1) = (f_{\Tt})'(p_{\Ubark}(v)).$$
\end{enumerate}
\end{lem}

\begin{proof}[Proof of Lemma \ref{lem:natlisom}]
\begin{enumerate}
\item 
Compact Frobenius reciprocity gives a natural isomorphism of functors
\begin{align*}
\Hom_{\Gt}\left(\ind_{\Kt}^{\Gt}\sigrt, \, \Ind_{\Bbt}^{\Gt}( - )\right) \cong \Hom_{\Kt}\left( \sigrt, \, \Ind_{\Bbt}^{\Gt}(-) \big \vert_{\Kt} \right).
\end{align*}
The Mackey decomposition of the second factor with respect to $\Gt = \Bbt\Kt$ gives another natural isomorphism
\begin{align*}
\Hom_{\Kt}\left( \sigrt, \, \Ind_{\Bbt}^{\Gt}(-) \big \vert_{\Kt} \right) \cong \Hom_{\Kt} \left( \sigrt, \, \Ind_{\Bbt \cap \Kt}^{\Kt}( - \big \vert_{\Bbt \cap \Kt})\right).
\end{align*}
By smooth Frobenius reciprocity, 
\begin{align*}
\Hom_{\Kt} \left( \sigrt, \, \Ind_{\Bbt \cap \Kt}^{\Kt}( - \big \vert_{\Bbt \cap \Kt})\right) \cong \Hom_{\Bbt \cap \Kt} \left( \sigrt \big \vert_{\Bbt \cap \Kt},\, ( - \big \vert_{\Bbt \cap \Kt}) \right)
\end{align*}
The latter is a functor defined on $\Tt$-representations viewed as $\Bbt$-representations by inflation. Since $\Bbt \cap \Kt \cap \Tt = \TKt$, we may replace the restriction to $\Bbt \cap \Kt$ in the second argument with restriction to $\Tt \cap \Kt$. As for the first argument, recall that $\Bbt \cap \Kt = (\TKt) \cdot (\Ubar \cap K)^*$ and that $\sigrt \big \vert_{(\Ubar \cap K)^*}$ is the inflation of a representation of $\Ubark$. Hence by the universal property of the $\Ubark$-coinvariants we have a natural isomorphism
\begin{align*}
\Hom_{\Bbt \cap \Kt}\left( \sigrt \big \vert_{\Bbt \cap \Kt}, \, (- \big \vert_{\TKt})\right) \cong \Hom_{\TKt} \left( (\sigrt)_{\Ubark}, \, (- \big \vert_{\TKt}) \right).
\end{align*}
Finally, compact Frobenius reciprocity gives a natural isomorphism
\begin{align*}
\Hom_{\TKt} \left( (\sigrt)_{\Ubark}, \, (- \big \vert_{\TKt}) \right) \cong \Hom_{\Tt} \left( \ind_{\TKt}^{\Tt}\left((\sigrt)_{\Ubark} \right), \, - \right). 
\end{align*}

\item We trace the progression of $f' \in \Hom_{\Kt} \left( \sigrt, \, \Ind_{\Bbt}^{\Gt}(-) \big \vert_{\Kt}\right)$ through the second, third, and fourth isomorphisms in the proof of Part (1) of the present lemma. The Mackey isomorphism is simply restriction of functions in the second argument. By smooth Frobenius reciprocity, $f'$ corresponds to the map $f'' \in \Hom_{\Bbt \cap \Kt}\left( \sigrt \big \vert_{\Bbt \cap \Kt}, \, (- \big \vert_{\TKt})\right)$ defined by 
$f''(v) = f'(v)(1).$
Finally $(f_{\Tt})'$ is the image of $f''$ in $\Hom_{\TKt} \left( (\sigrt)_{\Ubark}, \, (-\big \vert_{\Tt \cap \Kt})\right)$ via the universal property of the $\Ubark$-coinvariants, i.e.,
$$(f_{\Tt})'(p_{\Ubark}(v)) = f''(v).$$
Thus $f'(v)(1) = (f_{\Tt})'(p_{\Ubark}(v))$. 
\end{enumerate}
\end{proof}

\begin{lem} \label{lem:natlisomkt}
The statement of Lemma \ref{lem:natlisom} holds when $\Kt$ is replaced by $\Kt'$ everywhere and $\sigrt$ is replaced by a weight $\sigrt^{\tilde{\alpha}}$ of $\Kt'$. 
\end{lem}

\begin{proof}[Proof of Lemma \ref{lem:natlisomkt}]
To prove the $\Kt'$-analogue of Lemma \ref{lem:natlisom} (1), one replaces the Mackey decomposition with respect to the decomposition $\Gt = \Bbt \Kt$ with the Mackey decomposition with respect to the alternative Iwasawa decomposition $\Gt = \Bbt \Kt'$, and notes that $\Bbt \cap \Kt' = (\Tt \cap \Kt) \cdot \tilde{\alpha}(\Ubar \cap K)^*(\tilde{\alpha})^{-1}$ while $\sigrt^{\tilde{\alpha}}\big \vert_{(\Ubar^* \cap K^*)^{\tilde{\alpha}}}$ is the inflation of a representation of $\Ubark$. The proof of the $\Kt'$-analogue of Lemma \ref{lem:natlisom} (2) goes through with only the obvious adaptations coming from the changes made to the proof of (1). In particular, the formula of Lemma \ref{lem:natlisom} (2) is the same for the $\Kt'$-analogue. 
\end{proof}

\section{Genuine spherical Hecke algebras and Hecke bimodules}
\label{sec:sphericalHAstructure}
\subsection{Intertwining operators for compact inductions of $\Kt$-weights}
\label{subsec:sphericalHAdefK}
Let $\sigrt$ be a weight of $\Kt$ and let $\ind_{\Kt}^{\Gt}\sigrt$ denote the compact induction. The endomorphism algebra $\End_{\Gt}(\ind_{\Kt}^{\Gt}\sigrt)$ is called the \textit{genuine spherical Hecke algebra of $\Gt$ with respect to $\sigrt$} and is  denoted by $\SHA$.

More generally, let $\sigrt$ and $\sigst$ be two, possibly distinct, weights of $\Kt$. Then the \textit{genuine spherical Hecke bimodule of $\Gt$ with respect to $\sigrt$ and $\sigst$} is defined to be $\Hom_{\Gt}(\ind_{\Kt}^{\Gt}\sigrt, \ind_{\Kt}^{\Gt}\sigst)$ and is denoted by $\SHAs$. Then $\SHAs$ is a left $\mathcal{H}(\Gt, \Kt, \sigst)$-module and a right $\SHA$-module. Moreover, given three weights $\sigrt$, $\sigst$, and $\tilde{\sigma}_{\vec{t}}$ of $\Kt$, there is a product
\begin{equation} \label{eqn:HAcomplaw}
\mathcal{H}(\Gt, \Kt, \tilde{\sigma}_{\vec{s}}, \tilde{\sigma}_{\vec{t}}) \times \mathcal{H}(\Gt, \Kt, \tilde{\sigma}_{\vec{r}}, \tilde{\sigma}_{\vec{s}}) \rightarrow \mathcal{H}(\Gt, \Kt, \tilde{\sigma}_{\vec{r}}, \tilde{\sigma}_{\vec{t}})
\end{equation}
induced by composition.

Let $\mathbb{H}(\Gt, \Kt, \sigrt, \sigst)$ denote the $E$-vector space of compactly supported functions $f: \Gt \rightarrow \Hom_{E}(\sigrt, \sigst)$ such that
$$f(k_1gk_2) = \sigst(k_1) \circ f(g) \circ \sigrt(k_2)$$
for all $k_1$, $k_2 \in \Kt$ and $g \in \Gt$. In other words, each $f \in \mathbb{H}(\Gt, \Kt, \sigrt, \sigst)$ is compactly supported and satisfies
$$f(k_1 g(1, \zeta) k_2) = \zeta \cdot \sigs(Pr(k_1)) \circ f(g) \circ \sigr(Pr(k_2))$$
for all $k_1$, $k_2 \in \Ks$, $\zeta \in \mu_2$, and $g \in \Gt$. 

Frobenius reciprocity gives a bijection between $\SHAs$ and $\mathbb{H}(\Gt, \Kt, \sigrt, \sigst)$, compatible with the bimodule structure on $\mathbb{H}(\Gt, \Kt, \sigrt, \sigst)$ defined by the following convolution product: for $f_1 \in \mathbb{H}(\Gt, \Kt, \tilde{\sigma}_{\vec{s}}, \tilde{\sigma}_{\vec{t}})$ and $f_2 \in  \mathbb{H}(\Gt, \Kt, \tilde{\sigma}_{\vec{r}}, \tilde{\sigma}_{\vec{s}})$,
$$(f_1* f_2)(g) = \sum_{x \in \Kt \setminus \Gt} f_1(gx^{-1}) \circ f_2(x).$$
In particular, Frobenius reciprocity gives an $E$-algebra isomorphism between $\SHA$ and the convolution algebra $\mathbb{H}(\Gt, \Kt, \sigrt)$, for each weight $\sigrt$ of $\Kt$.  

\subsection{Intertwining operators for compact inductions of $\Kt'$-weights}
\label{subsec:sphericalHAdefK'}
We may define genuine spherical Hecke bimodules $\SHAps$ of $\Kt'$ in the same way as for $\Kt$, replacing $\Kt$ with $\Kt'$ and $\sigrt$, $\sigst$ with the conjugate weights $\sigrt^{\tilde{\alpha}}$, $\sigst^{\tilde{\alpha}}$. The following lemma, a special case of a general fact about conjugate representations (\cite{abdellatif:thesis}, Cor. 2.3.6), ensures that every genuine spherical Hecke bimodule of $\Kt'$ is isomorphic to a genuine spherical Hecke bimodule of $\Kt$. 

\begin{lem} \label{lem:kt'shaisom}
For any weights $\sigrt$ and $\sigst$ of $\Kt$, there is a $\Gt$-linear isomorphism 
$$\SHAs \rightarrow \SHAps.$$ 
\end{lem}

\begin{proof}[Proof of Lemma \ref{lem:kt'shaisom}] 
The identity map of $E$-vector spaces 
\begin{equation} \label{eqn:idmapHA} \Hom_{\Gt}\left(\ind_{\Kt}^{\Gt}\sigrt, \ind_{\Kt}^{\Gt}\sigst \right) \rightarrow \Hom_{\Gt}\left(\left(\ind_{\Kt}^{\Gt}\sigrt\right)^{\tilde{\alpha}}, \left( \ind_{\Kt}^{\Gt}\sigst \right)^{\tilde{\alpha}} \right)
\end{equation}
 is also an isomorphism of $\Gt$-modules. 
For any representation $\pi$ of $\Kt$, the map sending $f \in (\ind_{\Kt}^{\Gt}\pi)^{\tilde{\alpha}}$ to 
$$\Phi(f) = \left(g \mapsto f\left((\tilde{\alpha})^{-1}g \tilde{\alpha}\right) \right)$$
is a $\Gt$-linear isomorphism of  $(\ind_{\Kt}^{\Gt}\pi)^{\tilde{\alpha}}$ with $\ind_{\Kt'}^{\Gt}(\pi^{\tilde{\alpha}})$, inducing a $\Gt$-linear isomorphism 
\begin{equation} \label{eqn:conjmapHA}
\Hom_{\Gt}\left(\left(\ind_{\Kt}^{\Gt}\sigrt\right)^{\tilde{\alpha}}, \left( \ind_{\Kt}^{\Gt}\sigst \right)^{\tilde{\alpha}} \right) \rightarrow  \Hom_{\Gt}\left(\ind_{\Kt'}^{\Gt}\left(\sigrt^{\tilde{\alpha}}\right), \ind_{\Kt'}^{\Gt}\left(\sigst^{\tilde{\alpha}}\right)\right).
\end{equation}
The composition of (\ref{eqn:conjmapHA}) with (\ref{eqn:idmapHA}) is the desired $\Gt$-linear isomorphism.
\end{proof}

\subsection{Intertwining operators for compact inductions of $\Tt \cap \Kt$-representations}
\label{subsec:SHTdef}
Let $\pi_1$ and $\pi_2$ be two irreducible genuine representations of $\Tt \cap \Kt$. We define the \textit{genuine spherical Hecke bimodule of $\Tt$ with respect to $\Tt \cap \Kt$, $\pi_1$, and $\pi_2$} to be $\Hom_{\Tt}(\ind_{\Tt \cap \Kt}^{\Tt} \pi_1, \ind_{\Tt \cap \Kt}^{\Tt} \pi_2)$, and denote it by $\mathcal{H}(\Tt, \Tt \cap \Kt, \pi_1, \pi_2)$. It is only for formal reasons that we bother to define Hecke bimodules for pairs of nonisomorphic $\Tt \cap \Kt$-representations $\pi_1$, $\pi_2$: we show immediately (Lemma \ref{lem:SHTzero}) that if $\pi_1 \not \cong \pi_2$, then $\mathcal{H}(\Tt, \Tt \cap \Kt,  \pi_1, \pi_2) = 0$.

By Frobenius reciprocity, $\mathcal{H}(\Tt, \Tt \cap \Kt,  \pi_1, \pi_2)$ is isomorphic to the bimodule $\mathbb{H}(\Tt, \TKt,  \pi_1, \pi_2)$ of compactly supported functions $f: \Tt \rightarrow \Hom_{E}( \pi_1, \pi_2)$ such that 
$$f(k_1 t k_2) = \pi_1(k_1) \circ f(t) \circ \pi_2(k_2)$$
for all $k_1$, $k_2 \in \TKt$ and all $t \in \Tt$.
\begin{lem} \label{lem:SHTzero}
Let $\pi_1$, $\pi_2$ be two irreducible genuine representations of $\Tt \cap \Kt$. If $\pi_1 \not \cong \pi_2$ as $\TKt$-representations, then $\mathcal{H}(\Tt, \TKt, \pi_1, \pi_2) = 0$.
\end{lem}

\begin{proof}[Proof of Lemma \ref{lem:SHTzero}] Since $\pi_1$ and $\pi_2$ are irreducible representations of an abelian group and so are one-dimensional, we have $\Hom_E( \pi_1, \pi_2) \cong E$. Then each $f \in \mathbb{H}(\Tt, \TKt,  \pi_1, \pi_2)$ must satisfy
$$f(k t) = \pi_1(k) f(t) = \pi_2(k) f(t)$$
for all $k \in \TKt$ and all $t \in \Tt$, which is possible for $f \neq 0$ if and only if $\pi_1 \cong \pi_2$ as representations of $\Tt \cap \Kt$. Hence $\mathbb{H}(\Tt, \TKt,  \pi_1, \pi_2) = \mathcal{H}(\Tt, \TKt, \pi_1, \pi_2) = 0$ if $\pi_1\not \cong \pi_2$.
\end{proof}

Fix once and for all a $\Tt \cap \Kt$-linear isomorphism $\iota: (\tilde{\sigma}_{\vec{0}})_{\Ubark} \rightarrow (\tilde{\sigma}_{\overrightarrow{p-1}})_{\Ubark}$. For each pair $\vec{r}, \vec{s}$ of vectors in $\{ 0 , \dots, p-1\}^f$, define a $\Tt \cap \Kt$-linear map $\iota_{\vec{r}, \vec{s}}: (\tilde{\sigma}_{\vec{r}})_{\Ubark} \rightarrow (\tilde{\sigma}_{\vec{s}})_{\Ubark}$ as follows:
\begin{equation} \label{eqn:iotarsdef}
\iota_{\vec{r}, \vec{s}} = \begin{cases}
1 & \text{ if } \vec{r} = \vec{s}\\
\iota & \text{ if } \vec{r} = \vec{0}, \, \vec{s} = \overrightarrow{p-1},\\
\iota^{-1} & \text{ if } \vec{r} = \overrightarrow{p-1}, \, \vec{s} = \vec{0},\\
0 & \text{ otherwise}.
\end{cases}
\end{equation}

\begin{lem} \label{lem:SHTwtcoinv}
Let $\sigrt$, $\sigst$ be any weights of $\Kt$. Then there is a unique function $\psi_n^{\vec{r}, \vec{s}} \in \mathbb{H}(\Tt, \Tt \cap \Kt, (\sigrt)_{\Ubark}, (\sigst)_{\Ubark})$ satisfying 
$$\psi_n^{\vec{r}, \vec{s}}(\tilde{h}(\varpi)^m) = \begin{cases}
\iota_{\vec{r}, \vec{s}} & \text{ if } m = n,\\
0 & \text{ if } m \neq n,
\end{cases}$$
and the set $\{\psi_n^{\vec{r}, \vec{s}}: \, n \in \ZZ\}$ is a basis for $\mathbb{H}(\Tt, \TKt, (\sigrt)_{\Ubark}, (\sigst)_{\Ubark})$ as an $E$-vector space.
\end{lem}

\begin{proof}[Proof of Lemma \ref{lem:SHTwtcoinv}]
The statement that $\mathbb{H}(\Tt, \TKt, (\sigrt)_{\Ubark}, (\sigst)_{\Ubark}) = 0$ for $\sigrt$, $\sigst$ such that $\vec{r} \neq \vec{s}$ and $\{\vec{r}, \vec{s}\} \neq \{ \vec{0}, \overrightarrow{p-1}\}$ follows immediately from Lemma \ref{lem:SHTzero} and Corollary \ref{cor:ukisoms}. In the remaining cases, the target space of each $f \in \mathbb{H}(\Tt, \TKt, (\sigrt)_{\Ubark}, (\sigst)_{\Ubark})$ is one-dimensional, and $f$ is determined by its values on $\tilde{h}(\varpi)^m$, $m \in \ZZ$. The function $\psi^{\vec{r}, \vec{s}}_n$ is nonzero on $(\TKt)\tilde{h}(\varpi)^n$ and $\Tt = \amalg_{n \in \ZZ} \, \text{Supp}(\psi^{\vec{r}, \vec{s}}_n)$, so $\{ \psi_n^{\vec{r}, \vec{s}}\} $ is a basis for $\mathbb{H}(\Tt, \TKt, (\sigrt)_{\Ubark}, (\sigst)_{\Ubark})$ as an $E$-vector space. 
\end{proof}

Let $\tau^{\vec{r}, \vec{s}}_n$ denote the element of $\mathcal{H}(\Tt, \TKt, (\sigrt)_{\Ubark}, (\sigst)_{\Ubark})$ which corresponds to $\psi_n^{\vec{r}, \vec{s}}$ by Frobenius reciprocity. It follows from Lemma \ref{lem:SHTwtcoinv} that 
$\{ \tau^{\vec{r}, \vec{s}}_n: \, n \in \ZZ\}$ is a basis for $\mathcal{H}(\Tt, \TKt, (\sigrt)_{\Ubark}, (\sigst)_{\Ubark})$ as an $E$-vector space.

Following the convention of previous sections, we avoid duplicate notation when $\vec{r} = \vec{s}$ (so that $\tau^{\vec{r}}_n := \tau^{\vec{r}, \vec{r}}_n$, etc). In this case $\mathcal{H}(\Tt, \Tt \cap \Kt, (\sigrt)_{\Ubark})$ has an $E$-algebra structure given by composition, and is called a \textit{genuine spherical Hecke algebra of $\Tt$}. 

\begin{lem} \label{lem:SHTwtcoinvalg}
\begin{enumerate}
\item For any $\vec{r} \in \{ 0, \dots, p-1\}^{f}$, $(\tau^{\vec{r}}_1)^{n} = \tau^{\vec{r}}_{n}$ for all $n \in \ZZ$, and there is an $E$-algebra isomorphism $$\mathcal{H}(\Tt, \TKt, (\sigrt)_{\Ubark}) \rightarrow  E[(\tau^{\vec{r}}_1)^{\pm 1}].$$
\item Suppose that $\{ \vec{r}, \vec{s}\} = \{ \vec{0}, \overrightarrow{p-1}\}$. Then $\mathcal{H}(\Tt, \TKt, (\sigrt)_{\Ubark}, (\sigst)_{\Ubark})$ has the following Hecke bimodule structure: for $n$, $m \in \ZZ$, $\tau^{\vec{r}, \vec{s}}_n \circ \tau^{\vec{s}}_m = \tau^{\vec{r}}_m \circ \tau^{\vec{r}, \vec{s}}_n = \tau^{\vec{r}, \vec{s}}_{n+m}$, $\tau_n^{\vec{r}, \vec{s}} \circ \tau_m^{\vec{s}, \vec{r}} = \tau_{n+m}^{\vec{s}}$, and $\tau_m^{\vec{s}, \vec{r}} \circ \tau_n^{\vec{r}, \vec{s}} = \tau_{n+m}^{\vec{r}}.$
\end{enumerate}
\end{lem}

\begin{proof}[Proof of Lemma \ref{lem:SHTwtcoinvalg}]
\begin{enumerate}
\item We have $(\tau^{\vec{r}}_1)^{n} = \tau^{\vec{r}}_{n}$ if and only if $(\psi_1^{\vec{r}})^n = \psi_n^{\vec{r}}$ in $\mathbb{H}(\Tt, \TKt, (\sigrt)_{\Ubark})$. For any $n$, $m$, $k \in \ZZ$, 
\begin{align*}
(\psi^{\vec{r}}_n * \psi^{\vec{r}}_m)(\tilde{h}(\varpi)^k) & = \sum_{t \in (\TKt)\setminus \Tt} \psi^{\vec{r}}_n(\tilde{h}(\varpi)^kt^{-1})\circ \psi_m^{\vec{r}}(t)\\
& = \sum_{j \in \ZZ}\, \psi_n^{\vec{r}}(\tilde{h}(\varpi)^{k-j})\circ \psi_m^{\vec{r}}(\tilde{h}(\varpi)^j)
\end{align*}
The summand indexed by $j$ is nonzero only if both $j = m$ and $k - j = n$, so $(\psi^{\vec{r}}_n * \psi^{\vec{r}}_m)(\tilde{h}(\varpi)^k) = 0$ unless $k = n + m$. When $k = n + m$, we are left with 
$$(\psi_n^{\vec{r}} * \psi_m^{\vec{r}})(\tilde{h}(\varpi)^{n+m}) = \psi_n^{\vec{r}}(\tilde{h}(\varpi)^n)\circ \psi^{\vec{r}}_m(\tilde{h}(\varpi)^m) = 1,$$
so $\psi_n^{\vec{r}} * \psi_m^{\vec{r}} = \psi^{\vec{r}}_{n+m}.$ Passing back through Frobenius reciprocity, we have $\tau^{\vec{r}}_n \circ \tau^{\vec{r}}_m = \tau^{\vec{r}}_{n+m}$ for all $n$, $m \in \ZZ$. Hence $\tau^{\vec{r}}_1$ and $(\tau^{\vec{r}}_1)^{-1} = \tau^{\vec{r}}_{-1}$ generate $\SHT$ over $E$, and the map $\SHT \rightarrow E[(\tau^{\vec{r}}_1)^{\pm 1}]$ sending $\sum_{n \in \ZZ} a_n\tau_{n}^{\vec{r}}$ to $\sum_{n \in \ZZ} a_n(\tau^{\vec{r}}_1)^{n}$ is an isomorphism of $E$-algebras.
\item The product calculations are essentially the same as in the proof of (1).
\end{enumerate}
\end{proof}

Define $\mathbb{H}^{\leq 0}(\Tt, \Tt \cap \Kt, (\sigrt)_{\Ubark}, (\sigst)_{\Ubark})$ (resp., $\mathbb{H}^{< 0}(\Tt, \Tt \cap \Kt, (\sigrt)_{\Ubark}, (\sigst)_{\Ubark})$) to be the subset of $\mathbb{H}(\Tt, \Tt \cap \Kt, (\sigrt)_{\Ubark}, (\sigst)_{\Ubark})$ consisting of functions which are supported on $\amalg_{n \geq 0} (\Tt \cap \Kt)\tilde{h}(\varpi)^{-n}$ (resp., on $\amalg_{n > 0} (\Tt \cap \Kt) \tilde{h}(\varpi)^{-n}$), with the inherited product structure. Let $\mathcal{H}^{\leq 0}(\Tt, \TKt, (\sigrt)_{\Ubark}, (\sigst)_{\Ubark})$ (resp., $\mathcal{H}^{< 0 }(\Tt, \TKt, (\sigrt)_{\Ubark}, (\sigst)_{\Ubark})$) denote the corresponding subset of $\mathcal{H}(\Tt, \TKt, (\sigrt)_{\Ubark}, (\sigst)_{\Ubark})$.

From the definition of the $\tau_n^{\vec{r}}$ follows: 
\begin{lem} \label{lem:antidomisom}
The isomorphism of Lemma \ref{lem:SHTwtcoinv} (1) restricts to an $E$-algebra isomorphism
$$\mathcal{H}^{\leq 0}(\Tt, \TKt, (\sigrt)_{\Ubark}) \rightarrow E[(\tau_{1}^{\vec{r}})^{-1}] = E[\tau_{-1}^{\vec{r}}].$$
\end{lem}

\begin{rem}
\label{rem:genwtcase}
There is no harm in promoting the fixed $\Tt \cap \Kt$-linear isomorphism $\iota: (\tilde{\sigma}_{\vec{0}})_{\Ubark} \rightarrow (\tilde{\sigma}_{\overrightarrow{p-1}})_{\Ubark}$ to an identification of the underlying vector spaces. Doing so induces an identification of $\tau_n^{\vec{0}, \overrightarrow{p-1}}$ with $\tau_n^{\overrightarrow{p-1}, \vec{0}}$ for each $n$, as well as $E$-algebra isomorphisms $\mathcal{H}(\Tt, \TKt, (\tilde{\sigma}_{\vec{0}})_{\Ubark}, (\tilde{\sigma}_{\overrightarrow{p-1}})_{\Ubark} )\cong \mathcal{H}(\Tt, \TKt, (\tilde{\sigma}_{\vec{0}})_{\Ubark}, (\tilde{\sigma}_{\overrightarrow{p-1}})_{\Ubark}) \cong E[(\tau_1)^{\pm 1}]$ where $\tau_1$ stands for the now-identified operators $\tau_1^{\vec{0}, \overrightarrow{p-1}} = \tau_1^{\overrightarrow{p-1}, \vec{0}}$. However, for shallow reasons of consistency with the notation for spherical Hecke bimodules of $\Gt$, we have chosen not to make this identification. 
\end{rem}

\subsection{Product structure of  genuine spherical Hecke bimodules of $\Gt$}
\label{subsec:SHAvsbasis}
In this section we describe the structure of $\SHAs$ under composition, showing in particular that when $\vec{r} = \vec{s}$ the spherical Hecke algebra $\SHA$ is isomorphic to a polynomial algebra in one operator over $E$. The goal of the section is to prove the following statement:

\begin{thm} \label{thm:SHAsstructure} 
Let $\sigrt$, $\sigst$ be two weights of $\Kt$. 
\begin{enumerate}
\item If $\vec{r} \neq \vec{s}$ and $\{ \vec{r}, \vec{s}\} \neq \{ \vec{0}, \overrightarrow{p-1}\}$, then $\SHAs = 0$. 
\item The genuine spherical Hecke algebra $\SHA$ is isomorphic as an $E$-algebra to a polynomial algebra in one operator, denoted by $\mathcal{T}^{\vec{r}}_1$ and defined in (\ref{eqn:tprimedef}). 
\item If $\{ \vec{r}, \vec{s}\} = \{ \vec{0}, \overrightarrow{p-1}\}$, then $\SHAs$ is an infinite-dimensional $E$-vector space with an explicit basis $\{ \mathcal{T}^{\vec{r}, \vec{s}}_n: \, n > 0\}$ defined in (\ref{eqn:tprimedef}), and the following Hecke bimodule structure: for all $n > 0$, $m \geq 0$,
$$\mathcal{T}_n^{\vec{r}, \vec{s}} \circ \mathcal{T}_m^{\vec{r}} = \mathcal{T}_m^{\vec{s}} \circ \mathcal{T}_n^{\vec{r}, \vec{s}} = \mathcal{T}_{n+m}^{\vec{r}, \vec{s}},$$
and for all $n > 0$, $m > 0$,
$$\mathcal{T}_n^{\vec{r}, \vec{s}} \circ \mathcal{T}_m^{\vec{s}, \vec{r}} = \mathcal{T}_{n+m}^{\vec{s}}, \text{ and } \mathcal{T}_m^{\vec{s}, \vec{r}} \circ \mathcal{T}_n^{\vec{r}, \vec{s}} = \mathcal{T}_{n+m}^{\vec{r}}.$$
\end{enumerate}
\end{thm}

We will prove Theorem \ref{thm:SHAsstructure} by defining an injective Satake transform $\mathcal{S}_{\vec{r}, \vec{s}}: \SHAs  \rightarrow  \mathcal{H}(\Tt, \TKt, (\sigrt)_{\Ubark}, (\sigst)_{\Ubark})$ and then using the description of $\mathcal{H}(\Tt, \TKt, (\sigrt)_{\Ubark}, (\sigst)_{\Ubark})$ given in Lemma \ref{lem:SHTwtcoinv}.  

The statement and proof of the following proposition are closely based on the presentation of \cite{herzig:glnnotes} (especially Lemma 26 and Proposition 27), in which a mod $p$ Satake transform is given for $GL_n(F)$. The proof given there holds for $\Gt$ without significant changes, but we write it out for completeness. 

\begin{prop} \label{prop:gensataketransformi}
Let $\sigrt$, $\sigst$ be two weights of $\Kt$. We refer to Lemma \ref{lem:natlisom} for the definition of $(-)_{\Tt}$. 
\begin{enumerate}
\item There is a unique map 
$$\mathcal{S}_{\vec{r}, \vec{s}}: \SHAs  \rightarrow  \mathcal{H}(\Tt, \TKt, (\sigrt)_{\Ubark}, (\sigst)_{\Ubark})$$
such that $(f \circ \mathcal{T})_{\Tt} = f_{\Tt} \circ \mathcal{S}_{\vec{r}, \vec{s}}(\mathcal{T})$
for all $\mathcal{T} \in \SHAs,$ for all $f \in \Hom_{\Gt}(\ind_{\Kt}^{\Gt} \sigst, \, \Ind_{\Bbt}^{\Gt}(\pi))$, for every smooth genuine representation $\pi$ of $\Tt$. 

\item If $(\sigrt)_{\Ubark} \not \cong (\sigst)_{\Ubark}$ as $\TKt$-representations, then $\mathcal{S}_{\vec{r}, \vec{s}} = 0.$ Otherwise, 
$$\mathcal{S}_{\vec{r}, \vec{s}}(\mathcal{T})([1, p_{\Ubark}(v)]) = \sum_{t \in (\Tt \cap \Kt) \setminus \Tt} \left[t^{-1},
 \sum_{\bar{u} \in (\Ubar \cap K)^* \setminus \Ubar^*} p_{\Ubark} \left( \mathcal{T}'(v)(\bar{u}t)\right)\right]$$
for $\mathcal{T} \in \SHAs$, $t \in \Tt$, and $v \in V_{\vec{r}}$. 
Here $\mathcal{T}'$ denotes the element of $\Hom_{\Kt}(\sigrt, \ind_{\Kt}^{\Gt}\sigst \big \vert_{\Kt})$ which corresponds to $\mathcal{T} \in \SHAs$ by Frobenius reciprocity. The map $p_{\Ubark}$ is the projection $\tilde{\sigma} \rightarrow (\tilde{\sigma})_{\Ubark}$ for a weight $\tilde{\sigma}$ (with $\tilde{\sigma} = \sigrt$ on the left-hand side of the formula, and $\tilde{\sigma} = \sigst$ on the right-hand side). By $\Tt$-equivariance, the given values determine $\mathcal{S}_{\vec{r}, \vec{s}}(\mathcal{T})$. 
\item $\mathcal{S}_{\vec{r}, \vec{s}}$ is $E$-linear, and if $\tilde{\sigma}_{\vec{t}}$ is a third weight of $\Kt$, then 
$$\mathcal{S}_{\vec{r}, \vec{t}}(\underline{\mathcal{T}}\circ \mathcal{T}) = \mathcal{S}_{\vec{s}, \vec{t}}(\underline{\mathcal{T}}) \circ \mathcal{S}_{\vec{r}, \vec{s}}(\mathcal{T})$$
for all $\mathcal{T} \in \SHAs$, $\underline{\mathcal{T}}\in \mathcal{H}(\Gt, \Kt, \sigst, \tilde{\sigma}_{\vec{t}}).$ 
\end{enumerate}
\end{prop}

\begin{proof}[Proof of Proposition \ref{prop:gensataketransformi}] 
\begin{enumerate}
\item Let $\mathcal{T} \in \mathcal{H}(\Gt, \Kt, \sigrt, \sigst) = \Hom_{\Gt}(\ind_{\Kt}^{\Gt}\sigrt, \, \ind_{\Kt}^{\Gt}\sigst)$. Precomposition with $\mathcal{T}$ is a natural transformation 
$$\Hom_{\Gt}\left(\ind_{\Kt}^{\Gt}\sigst, \, \Ind_{\Bbt}^{\Gt}(-)\right) \longrightarrow \Hom_{\Gt}\left(\ind_{\Kt}^{\Gt}\sigrt, \, \Ind_{\Bbt}^{\Gt}(-)\right),$$
hence induces, via the natural isomorphism of Lemma \ref{lem:natlisom}, a natural transformation
$$\Hom_{\Tt}\left(\ind_{\TKt}^{\Tt}\left((\sigst)_{\Ubark}\right), -\right) \longrightarrow \Hom_{\Tt}\left(\ind_{\TKt}^{\Tt}\left((\sigrt)_{\Ubark}\right), -\right).$$
By the Yoneda Lemma, there is a unique map  $\mathcal{S}_{\vec{r}, \vec{s}}(\mathcal{T}) \in \Hom_{\Tt}\left(\ind_{\TKt}^{\Tt}\left((\sigrt)_{\Ubark} \right), \ind_{\TKt}^{\Tt}\left((\sigst)_{\Ubark} \right) \right) \cong \mathcal{H}\left(\Tt, \TKt, \left(\sigrt \right)_{\Ubark}, \left(\sigst \right)_{\Ubark}\right)$ such that 
\begin{equation} \label{eqn:satakelinearity}
(f \circ \mathcal{T})_{\Tt} = f_{\Tt} \circ \mathcal{S}_{\vec{r}, \vec{s}}(\mathcal{T})
\end{equation}
for all $f \in \Hom_{\Gt}\left(\ind_{\Kt}^{\Gt}\sigst, \, \Ind_{\Bbt}^{\Gt}(\pi)\right)$, for every smooth genuine representation $\pi$ of $\Tt.$ 

\item If $\left(\sigrt \right)_{\Ubark}$ and $\left(\sigst \right)_{\Ubark}$ are not isomorphic as $\TKt$-representations, then $\mathcal{S}_{\vec{r}, \vec{s}} = 0$ by Lemma \ref{lem:SHTzero}. Otherwise, let $f_0$ denote the unique element of  $\Hom_{\Gt}\left(\ind_{\Kt}^{\Gt}\sigst, \, \Ind_{\Bbt}^{\Gt}\left( \ind_{\TKt}^{\Tt}(\sigst)_{\Ubark} \right)\right)$ such that $(f_0)_{\Tt} = \ind_{\TKt}^{\Tt}(\iota_{\vec{s}, \vec{s}})$ (recall that we have chosen $\iota_{\vec{s}, \vec{s}}$ to be the identity map on $(\sigst)_{\Ubark}$), and let $f_0'$ denote the element of 
$\Hom_{\Kt}\left( \sigst, \, \Ind_{\Bbt}^{\Gt}\left(\ind_{\TKt}^{\Tt}(\sigst)_{\Ubark}\right)\big \vert_{\Kt}\right)$ which corresponds to $f_0$ by Frobenius reciprocity. 

The following equalities in $\Hom_{\Tt \cap \Kt}\left((\sigrt)_{\Ubark}, \, \ind_{\TKt}^{\Tt}(\sigst)_{\Ubark}\big \vert_{\Tt \cap \Kt}\right)$ are obtained by applying Frobenius reciprocity to both sides of (\ref{eqn:satakelinearity}): 
\begin{equation} \label{eqn:satderiveqn}
\left((f_0 \circ \mathcal{T})_{\Tt}\right)' = \left( \ind_{\TKt}^{\Tt}(\iota_{\vec{s}, \vec{s}}) \circ  \mathcal{S}_{\vec{r}, \vec{s}}(\mathcal{T})\right)' = \mathcal{S}_{\vec{r}, \vec{s}}(\mathcal{T})'.
\end{equation}
Let $v \in \Vsigt$ and let $(f_0 \circ \mathcal{T})'$ denote the element of $\Hom_{\Kt}\left(\sigrt, \, \Ind_{\Bbt}^{\Gt}(\ind_{\Tt \cap \Kt}^{\Tt}(\sigst)_{\Ubark})\right)$ which corresponds to $f_0 \circ \mathcal{T}$ by Frobenius reciprocity. By part (2) of Lemma \ref{lem:natlisom},
\begin{align*}
\left((f_0 \circ \mathcal{T})_{\Tt}\right)'(p_{\Ubark}(v)) & = \left( f_0 \circ \mathcal{T}\right)'(v)(1)\\
& = \left(f_0' * \mathcal{T}'\right)(v)(1)
\end{align*}
as elements of $\ind_{\Tt \cap \Kt}^{\Tt}(\sigst)_{\Ubark}$. Calculating the convolution product using Lemma 28 of \cite{herzig:glnnotes} for the second equality and using the fact that $\Ubar^* \cap \Kt = (\Ubar \cap K)^*$ for the third, we have
\begin{align*}
 \left(f_0' * \mathcal{T}'\right)(v)(1) & = \sum_{g \in \Kt \setminus \Gt} f_0'(\mathcal{T}'(v)(g))(g^{-1})\\
& = \sum_{t \in (\TKt) \setminus \Tt} \sum_{\bar{u} \in (\Ubar^* \cap \Kt) \setminus \Ubar^*} f_0'(\mathcal{T}'(\bar{u}t)v))(t^{-1}\bar{u}^{-1}) \\
& = \sum_{t \in (\TKt) \setminus \Tt} \sum_{\bar{u} \in (\Ubar \cap K)^* \setminus \Ubar^*} t^{-1} \cdot f_0'(\mathcal{T}'(\bar{u}t)v)(1)\\
& = \sum_{t \in (\TKt) \setminus \Tt} \sum_{\bar{u} \in (\Ubar \cap K)^* \setminus \Ubar^*} t^{-1}\bar{u}^{-1} \cdot \left((f_0)_{\Tt}\right)'\left(p_{\Ubark}(\mathcal{T}'(\bar{u}t)v)\right).\\
\end{align*}
The corresponding equality in $\Hom_{\Tt}(\ind_{\Tt \cap \Kt}^{\Tt}(\sigst)_{\Ubark}, \ind_{\Tt \cap \Kt}^{\Tt}(\sigrt)_{\Ubark})$ is 
\begin{align*}
\mathcal{S}_{\vec{r}, \vec{s}}(\mathcal{T})\left([1, p_{\Ubark}(v)] \right)& = \sum_{t \in (\TKt) \setminus \Tt} \left[ t^{-1}, \, \sum_{\bar{u} \in (\Ubar \cap K)^* \setminus \Ubar^*} \iota_{\vec{s}, \vec{s}} \circ p_{\Ubark} \circ\mathcal{T}'(\bar{u}t)(v) \right]\\
& = \sum_{t \in (\TKt) \setminus \Tt} \left[ t^{-1}, \, \sum_{\bar{u} \in (\Ubar \cap K)^* \setminus \Ubar^*} p_{\Ubark} \left(\mathcal{T}'(\bar{u}t)(v)\right) \right],
\end{align*}
which is the desired formula.

\item Both claims of (3) will follow from from (1). For the first claim, let $e \in E$. For every $\mathcal{T} \in \SHAs$, every smooth genuine representation $\pi$ of $\Tt$, and every $f \in \Hom_{\Gt}(\ind_{\Kt}^{\Gt}\sigst, \Ind_{\Bbt}^{\Gt}(\pi))$, we have
\begin{align*}
f_{\Tt} \circ \mathcal{S}_{\vec{r}, \vec{s}}(e\cdot\mathcal{T}) & = (f \circ e\mathcal{T})_{\Tt}\\
& = (ef \circ \mathcal{T})_{\Tt}\\
& = (ef)_{\Tt} \circ \mathcal{S}_{\vec{r}, \vec{s}}(\mathcal{T})\\
& = f_{\Tt} \circ e\cdot\mathcal{S}_{\vec{r}, \vec{s}}(\mathcal{T}).
\end{align*}
The uniqueness statement of (1) now implies that $\mathcal{S}_{\vec{r}, \vec{s}}(e \cdot -)$ and $e \cdot \mathcal{S}_{\vec{r}, \vec{s}}(-)$ are identical.

For the second claim, let $\sigrt$, $\sigst$, and $\tilde{\sigma}_{\vec{t}}$ be three weights of $\Kt$, and let $\pi$ be a smooth genuine representation of $\Tt$. Let $\mathcal{T} \in \SHAs$, $\underline{\mathcal{T}}\in \mathcal{H}(\Gt, \Kt, \sigst, \tilde{\sigma}_{\vec{t}})$, and $f \in \Hom_{\Gt}\left(\ind_{\Kt}^{\Gt} \tilde{\sigma}_{\vec{t}}, \Ind_{\Bbt}^{\Gt}(\pi)\right).$ 
Then 
\begin{align*}
f_{\Tt} \circ \mathcal{S}_{\vec{r}, \vec{t}}(\underline{\mathcal{T}}\circ \mathcal{T}) & = (f \circ \underline{\mathcal{T}}\circ \mathcal{T})_{\Tt}\\
& = (f \circ \underline{\mathcal{T}})_{\Tt} \circ \mathcal{S}_{\vec{r}, \vec{s}}(\mathcal{T})\\
& = f_{\Tt} \circ \mathcal{S}_{\vec{s}, \vec{t}}(\underline{\mathcal{T}}) \circ \mathcal{S}_{\vec{r}, \vec{s}}(\mathcal{T}).
\end{align*}
So for fixed $\underline{\mathcal{T}}$ the maps $\mathcal{S}_{\vec{r}, \vec{t}}(\underline{\mathcal{T}}, -)$ and $\mathcal{S}_{\vec{s}, \vec{t}}(\underline{\mathcal{T}}) \circ \mathcal{S}_{\vec{r}, \vec{s}}(-)$ agree on $\mathcal{H}(\Gt, \Kt, \sigrt, \sigst)$, and therefore are identical by the uniqueness statement of (1). Allowing $\underline{\mathcal{T}}$ to vary over $\mathcal{H}(\Gt, \Kt, \sigst, \tilde{\sigma}_{\vec{t}})$, we get the desired compatibility. 
\end{enumerate}
\end{proof}

Next, in Lemma \ref{lem:gshabasis} and Corollary \ref{cor:shabasis} we determine an explicit basis for $\mathbb{H}(\Gt, \Kt, \sigrt, \sigst)$ as an $E$-vector space, getting by proxy an $E$-basis for $\mathcal{H}(\Gt, \Kt, \sigrt, \sigst)$. The basis is normalized so as to be compatible with the system $\{ \iota_{\vec{r}, \vec{s}}: \vec{r}, \vec{s} \in \{0, \dots, p-1\}^f\}$ of $\Tt \cap \Kt$-linear maps chosen in (\ref{eqn:iotarsdef}).

\begin{lem} \label{lem:gshabasis}
Let $\sigrt$, $\sigst$ be two weights of $\Kt$, and let $\rho_{\vec{r}, \vec{s}}$ denote the following composition: 
\begin{displaymath}
\xymatrix{ \sigrt  \ar[rr]^{\rho_{\vec{r}, \vec{s}}} \ar[d]^{p_{\Ubark}} & & \sigst \\ (\sigrt)_{\Ubark} \ar[r]_{\iota_{\vec{r}, \vec{s}}} & (\sigst)_{\Ubark} \ar[r]_{j_{\vec{s}}^{-1}} & (\sigst)^{\Uk} \ar[u]^{i_{\vec{s}}^{\Uk}}}
\end{displaymath}

The space of functions in $\mathbb{H}(\Gt, \Kt, \sigrt, \sigst)$ with support in a double coset of the form $\Kt \tilde{h}(\varpi)^{-n} \Kt$, $n \geq 0$, is at most one-dimensional and is spanned by the function $\varphi^{\vec{r}, \vec{s}}_n$ defined as follows:
\begin{align*}
\varphi^{\vec{r}, \vec{s}}_0(\tilde{h}(\varpi)^m) &= \begin{cases}
1 & \text{ if } m = 0 \text{ and } \vec{r} = \vec{s},\\
0 & \text{ otherwise},
\end{cases}\\
\varphi^{\vec{r}, \vec{s}}_n(\tilde{h}(\varpi)^m) & = \begin{cases}
\rho_{\vec{r}, \vec{s}} & \text{ if } m = -n \text{ and } (\sigrt)_{\Ubark} \cong (\sigst)_{\Ubark} \text{ as } \TKt \text{-representations},\\
0 & \text{ if } |m| \neq n \text{ or }  (\sigrt)_{\Ubark}\not \cong (\sigst)_{\Ubark}  \text{ as } \TKt \text{-representations}.
\end{cases}
\end{align*}
\end{lem}

\begin{proof}[Proof of Lemma \ref{lem:gshabasis}]
The proof goes along the same lines as that of the similar statement for $G$ in \cite{abdellatif:thesis} Lemme 3.5.5. Suppose that a function $\varphi \in \mathbb{H}(\Gt, \Kt, \sigrt, \sigst)$ has support contained in $\Kt \tilde{h}(\varpi)^{-n} \Kt$. By definition of $\mathbb{H}(\Gt, \Kt, \sigrt, \sigst)$, 
\begin{equation} \label{eqn:ktequivariance}
\sigst(k_1) \circ \varphi(\tilde{h}(\varpi)^n) = \varphi(\tilde{h}(\varpi)^n) \circ \sigrt(k_2)
\end{equation}
whenever $k_1$, $k_2 \in \Kt$ satisfy 
\begin{equation} \label{eqn:ktcommreln}
k_1 \tilde{h}(\varpi)^n = \tilde{h}(\varpi)^n k_2.
\end{equation}

In the case $n = 0$, we have $\sigst(k) \circ \varphi((1, 1)) = \varphi((1, 1)) \circ \sigrt(k)$ for all $k \in \Kt$. Since $\sigst$ is an irreducible $\Kt$-representation, either $\varphi((1, 1))$ is an isomorphism or is zero. In the former case, i.e., if $\sigrt = \sigst$ and $\varphi((1, 1)) \neq 0$, Schur's Lemma implies that $\varphi((1, 1)) \in E^\times.$ 

In the case $n > 0$, two elements $k_1 \in \Kt$ and $k_2 \in \Ks$ satisfy (\ref{eqn:ktcommreln}) if and only if $Pr(k_2) = 
\left( \begin{array}{cc}
a    & b \\
c    & d 
  \end{array} \right) \in K$ with $v_F(b) \geq 2n$. A calculation shows that then $Pr(k_1) = 
\left( \begin{array}{cc}
a    & \varpi^{-2n}b \\
\varpi^{2n}c    & d 
  \end{array} \right)$ and $k_1 \in \Ks$. By (\ref{eqn:ktequivariance}) and the definition of $\sigrt$, $\sigst$,
\begin{equation} \label{eqn:kteqconsequence}
\sigma_{\vec{s}}
\left( \begin{array}{cc}
    a    & \varpi^{-2n}b \\
0   &  d
  \end{array} \right) \circ \varphi(\tilde{h}(\varpi)^n) = \varphi(\tilde{h}(\varpi)^n) \circ \sigma_{\vec{r}} 
\left( \begin{array}{cc}
a    & 0 \\
c    & d 
  \end{array} \right)
\end{equation}

for all such $a$, $b$, $c$, $d$. For the same reasons as in the proof of \cite{barthellivne:irredmodp} Lemma 7, the equality (\ref{eqn:kteqconsequence}) is equivalent to $\varphi(\tilde{h}(\varpi)^{-n})$ having the following three properties: (1) the image of $\varphi(\tilde{h}(\varpi)^{-n})$ is contained in $(\sigst)^{\Uk}$, (2) $\varphi(\tilde{h}(\varpi)^{-n})$ factors through the projection $p_{\Ubark}: \sigrt \rightarrow (\sigrt)_{\Ubark}$, and (3) $\sigst(t) \circ \varphi(\tilde{h}(\varpi)^{-n}) = \varphi(\tilde{h}(\varpi)^{-n}) \circ \sigrt(t)$ for all $t \in \Tt \cap \Kt.$ Due to properties (1) and (2), $\varphi(\tilde{h}(\varpi)^{-n})$ is a composition of the form given in the statement of the lemma, for some map $\iota: (\sigrt)_{\Ubark} \rightarrow (\sigst)_{\Ubark}$. By property (3) $\iota$ must be $\Tt \cap \Kt$-linear, and since $(\sigrt)_{\Ubark}$ and $(\sigst)_{\Ubark}$ are one-dimensional, $\iota_{\vec{r}, \vec{s}}$ is either 0 or a $\Tt \cap \Kt$-isomorphism. Such an isomorphism, if it exists, is unique up to a scalar and thus the choice does not affect the $E$-span of $\varphi$. Thus we may take $\iota = \iota_{\vec{r}, \vec{s}}$, and the resulting function $\varphi^{\vec{r}, \vec{s}}_n := \varphi$ spans the space of functions in $\mathbb{H}(\Gt, \Kt, \sigrt, \sigst)$ with support in $\Kt \tilde{h}(\varpi)^{-n}\Kt$.  
\end{proof}


As a corollary of Lemma \ref{lem:gshabasis}, we have:
\begin{cor}
\label{cor:shabasis}
$\mathbb{H}(\Gt, \Kt, \sigrt, \sigst) = 0$ if neither $\vec{r} = \vec{s}$ nor $\{\vec{r}, \vec{s}\} = \{ \vec{0}, \overrightarrow{p-1}\}$, and otherwise a basis for $\mathbb{H}(\Gt, \Kt, \sigrt, \sigst)$ as an $E$-vector space is given by
$$\begin{cases}
\{\varphi^{\vec{r}, \vec{s}}_n\}_{n \geq 0} & \text{ if } \vec{r} = \vec{s},\\
\{ \varphi^{\vec{r}, \vec{s}}_n\}_{n > 0} & \text{ if } \{\vec{r}, \vec{s} \} = \{ \vec{0}, \overrightarrow{p-1}\}.
\end{cases}$$
\end{cor}
\begin{proof}[Proof of Corollary \ref{cor:shabasis}]
The definition of $\iota_{\vec{r}, \vec{s}}$ (cf. (\ref{eqn:iotarsdef})) implies that the functions $\tilde{\phi}^{\vec{r}, \vec{s}}_n$ are all identically zero if $\vec{r} \neq \vec{s}$ and $\{ \vec{r}, \vec{s}\} \neq \{ \vec{0}, \overrightarrow{p-1}\}$. Otherwise, there exists a vector $v \in V_{\vec{r}}$ such that $\varphi^{\vec{r}, \vec{s}}_n(\tilde{h}(\varpi)^n)(v) \neq 0$ (for example, any $v \in V^{\Uk}$). Hence $\emptyset \neq \text{Supp}(\varphi^{\vec{r}, \vec{s}}_n)$, and by Lemma \ref{lem:gshabasis}, $\varphi^{\vec{r}, \vec{s}}_n$ spans the $E$-vector space of functions in $\mathbb{H}(\Gt, \Kt, \sigrt, \sigst)$ with support in $\Kt \tilde{h}(\varpi)^{-n}\Kt$. The Cartan decomposition $\Gt = \coprod_{n \leq 0} \Kt \tilde{h}(\varpi)^{-n}\Kt$ implies that the set $\{ \varphi^{\vec{r}, \vec{s}}_n: \, n \in\ZZ\}$ is linearly independent. 
\end{proof}

Let $\mathcal{T}_n^{\vec{r},\vec{s}}$ and $(\mathcal{T}_n^{\vec{r},\vec{s}})'$ denote, respectively, the elements of $\Hom_{\Gt}(\ind_{\Kt}^{\Gt}\sigrt, \, \ind_{\Kt}^{\Gt}\sigst)$ and of $\Hom_{\Kt}(\sigrt, \, \ind_{\Kt}^{\Gt}\sigst \big \vert_{\Kt})$ which correspond to $\varphi^{\vec{r}, \vec{s}}_n$ by Frobenius reciprocity. Explicitly, for $v \in V_{\vec{r}},$ $g \in \Gt$, and $f \in \ind_{\Kt}^{\Gt} \sigrt$, 
\begin{equation}
\label{eqn:tprimedef}
(\mathcal{T}^{\vec{r}, \vec{s}}_n)'(v)(g) = \varphi^{\vec{r}, \vec{s}}_n(g)(v),
\end{equation}
$$ \mathcal{T}^{\vec{r}, \vec{s}}_n(f)(g) = \sum_{\Kt x \in \Kt \setminus \Gt} \varphi^{\vec{r}, \vec{s}}_n(gx^{-1})(f(x)) = \sum_{\Kt x \in \Kt\setminus\Kt  \tilde{h}(\varpi)^{n}\Kt g} \varphi^{\vec{r}, \vec{s}}_n(gx^{-1})(f(x)).$$
If $\vec{r} = \vec{s}$, we will write $\mathcal{T}^{\vec{r}}_n$ instead of $\mathcal{T}^{\vec{r}, \vec{r}}_n$. We next explicitly determine the image of $\mathcal{T}^{\vec{r}, \vec{s}}_n$ under the Satake transform $\mathcal{S}_{\vec{r}, \vec{s}}$.

\begin{prop} \label{prop:satakeproperties}
If $\vec{r} = \vec{s}$ or if $\{\vec{r}, \vec{s} \} = \{ \vec{0}, \overrightarrow{p-1}\}$, then for $n > 0$,
$$\mathcal{S}_{\vec{r}, \vec{s}}(\mathcal{T}^{\vec{r}, \vec{s}}_n) = \tau^{\vec{r},\vec{s}}_{-n}.$$
If $\vec{r} = \vec{s}$, then $\mathcal{S}_{\vec{r}}(\mathcal{T}^{\vec{r}}_0) = 1$. 
\end{prop}

\begin{proof}[Proof of Proposition \ref{prop:satakeproperties}]
In order to de-clutter the notation, set $\mathcal{T}_n:= \mathcal{T}^{\vec{r}, \vec{s}}_n$ and $\varphi_n:= \varphi^{\vec{r}, \vec{s}}_{n}$ for the duration of the proof. We will pass to $\Hom_{\Tt \cap \Kt}((\sigrt)_{\Ubark}, \, \ind_{\Tt \cap \Kt}^{\Tt}(\sigst)_{\Ubark})$ and show that $\mathcal{S}_{\vec{r}, \vec{s}}(\mathcal{T}_n)' = \tau_{-n}'$. Fix $n \geq 0$, $m \in \ZZ$, and $v \in V_{\vec{r}}$. Using the formula for $\mathcal{S}_{\vec{r}, \vec{s}}(\mathcal{T}_n)'$ found in the proof of Proposition \ref{prop:gensataketransformi} (2) and the definition (\ref{eqn:tprimedef}) of $\mathcal{T}'_n$, 
{\tiny
\begin{align*}
\mathcal{S}_{\vec{r}, \vec{s}}(\mathcal{T}_n)'(p_{\Ubark}v)(\tilde{h}(\varpi)^m) & = \sum_{\bar{u} \in  (\Ubar \cap K)^* \setminus \Ubar^*} p_{\Ubark}\left( \mathcal{T}_n'(v)(\bar{u}\tilde{h}(\varpi)^m)\right)\\
& =  \sum_{\substack{\bar{u} \in  (\Ubar \cap K)^* \setminus \Ubar^*:\\ (\Ubar \cap K)^*\bar{u}\tilde{h}(\varpi)^m \subset \Kt \tilde{h}(\varpi)^{-n} \Kt}} p_{\Ubark} \circ \phit_n(\bar{u}\tilde{h}(\varpi)^m)v \\
& =\left(\sum_{\substack{\bar{u} \in  (\Ubar \cap K)^* \setminus \Ubar^*:\\ (\Ubar \cap K)^*\bar{u}\tilde{h}(\varpi)^{m} \subset \Ks \tilde{h}(\varpi)^{-n} \Ks}} \hspace{-1cm} p_{\Ubark} \circ \phit_n(\bar{u}\tilde{h}(\varpi)^m)v\, +\hspace{-1cm} \sum_{\substack{\bar{u} \in  (\Ubar \cap K)^* \setminus \Ubar^*:\\ (\Ubar \cap K)^*\bar{u}\tilde{h}(\varpi)^m \subset \Ks \tilde{h}(\varpi)^{-n}(1, -1)\Ks}} \hspace{-1cm} p_{\Ubark} \circ \phit_n(\bar{u}\tilde{h}(\varpi)^m)v \right).
\end{align*}
}

Lemma \ref{lem:satnontrivwts} specifies the inner summands in the above formula when $\vec{s} \neq 0$, and Lemma \ref{lem:wtr1dim} does the same for $\vec{s} = \vec{0}$. 
\begin{lem} \label{lem:satnontrivwts}
Suppose that $\vec{s} \neq \vec{0}$ and suppose that $(\sigrt)_{\Ubark} \cong (\sigst)_{\Ubark}$ as $\TKt$-representations. If $v \in V_{\vec{r}}$ and if the triple ($n \geq 0$, $m \in \ZZ$, $\tilde{\bar{u}} \in (\Ubar \cap K)^* \setminus \Ubar^*$) satisfies
$$K^*\tilde{\bar{u}}\tilde{h}(\varpi)^m \subset  K^* \tilde{h}(\varpi)^{-n}(1, \zeta)K^*,$$
then
$$p_{\Ubark}\left( \varphi^{\vec{r}, \vec{s}}_n(\tilde{\bar{u}}\tilde{h}(\varpi)^m)(v) \right) = \begin{cases}
p_{\Ubark}v & \text{ if } m = n = 0 \text{ and } \vec{r} = \vec{s},\\
p_{\Ubark}\circ \rho_{\vec{r}, \vec{s}}v & \text{ if } m = -n < 0,\\
0 & \text{ otherwise}. 
\end{cases}$$
\end{lem}

\begin{proof}[Proof of Lemma \ref{lem:satnontrivwts}]
We continue to write $\varphi_n$ for $\varphi^{\vec{r}, \vec{s}}_n$, $\bar{u}$ for $Pr(\tilde{\bar{u}})$, and we refer to Lemma \ref{lem:gshabasis} for the definitions of $\varphi_n$ and $\rho_{\vec{r}, \vec{s}}$. The proof breaks up into the following cases:
\begin{enumerate}
\item \underline{$m = n = 0$.}  Then $\tilde{\bar{u}}\tilde{h}(\varpi)^m = \tilde{\bar{u}} \in (\Ubar \cap K)^*$, so $\zeta$ must equal 1 and 
\begin{align*}
p_{\Ubark} \left( \varphi_n(\tilde{\bar{u}}\tilde{h}(\varpi)^m)(v)  \right) & = p_{\Ubark} \left(\varphi_0(\tilde{\bar{u}})(v) \right) = p_{\Ubark} \circ \sigma_{\vec{s}}(\bar{u}) \circ \varphi_0((1, 1))\,( v).
\end{align*}
If $\vec{r} = \vec{s}$ then $\varphi_0((1, 1)) = 1$, and since $\bar{u} \in \Ubar \cap K$ we have
\begin{align*}
p_{\Ubark} \circ  \sigma_{\vec{s}}(\bar{u}) \circ \varphi_0((1, 1))\, v & = p_{\Ubark} \circ \sigma_{\vec{r}}(\bar{u})(v) = p_{\Ubark}v.
\end{align*} 
If $\vec{r} \neq \vec{s}$, then $\varphi_0((1, 1)) = 0$, so $p_{\Ubark} \left( \varphi_n(\tilde{\bar{u}}\tilde{h}(\varpi)^m)(v)  \right) = 0.$

\item \underline{$m = -n < 0$.} We have
$$\tilde{\bar{u}}\tilde{h}(\varpi)^{-n} \in \Kt \tilde{h}(\varpi)^{-n} \Kt$$
if and only if $\tilde{\bar{u}} \in (\Ubar \cap K)^*$, so again $\zeta = 1$, and 
\begin{align*}
p_{\Ubark} \left( \varphi_n(\tilde{\bar{u}}\tilde{h}(\varpi)^m)(v)  \right) & = p_{\Ubark} \circ \sigma_{\vec{s}}(\bar{u})\circ  \varphi_n(\tilde{h}(\varpi)^{-n})\, (v) = p_{\Ubark} \circ \sigma_{\vec{s}}(\bar{u})\circ  \rho_{\vec{r}, \vec{s}}\, (v) = p_{\Ubark}\circ \rho_{\vec{r}, \vec{s}}v.
\end{align*}

\item \underline{$-n < m < n$.}  Applying the transpose operation to the decomposition of $K^*\tilde{h}(\varpi)^{-n}K^*$ given in Lemma \ref{lem:leftksreps}, we have in this case (up to multiplication by $(\Ubar \cap K)^*$ on the left) that $\tilde{\bar{u}} = \tilde{\bar{u}}(\lambda\varpi^{-2n+1})$ for some $\lambda \in \cI_{2n-1}$ such that $v_F(\lambda) = n - m - 1$. Then, applying the explicit Bruhat decomposition and calculating, we have
\begin{align*}
\tilde{\bar{u}}(\lambda\varpi^{-2n+1}) \tilde{h}(\varpi)^m & = \tilde{u}(\lambda^{-1}\varpi^{2n-1})\tilde{h}(-\lambda^{-1}\varpi^{-n+m+1})\tilde{w}(1)\tilde{h}(\varpi)^{-n}\tilde{u}(\lambda^{-1}\varpi^{2n-2m-1})(1, (-1, \varpi^m)_F \eta_n(\lambda)).
\end{align*}
The conditions $-n < m < n$ and $v_F(\lambda) = n - m - 1$ imply that $\lambda^{-1}\varpi^{2n-1} \in \varpi\OF$, $-\lambda^{-1}\varpi^{-n + m + 1} \in \OF^\times$, and $\lambda^{-1}\varpi^{2n-2m-1} \in \varpi\OF$. Using the fact that $\sigma_{\vec{s}}$ and $\sigma_{\vec{r}}$ factor through reduction $\pmod{\varpi}$, we have
\begin{align*}
p_{\Ubark}\left(\varphi_n\left( \tilde{\bar{u}}\tilde{h}(\varpi)^m\right)(v) \right) & =(-1, \varpi^m)_F \eta_n(\lambda) \cdot p_{\Ubark}\circ \sigma_{\vec{s}}\left(h(-\lambda^{-1}\varpi^{-n+m+1})\tilde{w}(1)\right) \circ \varphi_n(\tilde{h}(\varpi)^{-n})(v)\\
& = (-1, \varpi^m)_F \eta_n(\lambda) \cdot \delta_{\vec{s}}\left(-\lambda^{-1}\varpi^{-n+m+1}\right) \cdot p_{\Ubark}\circ \sigma_{\vec{s}}(w(1))\circ\rho_{\vec{r}, \vec{s}}v,
\end{align*}
which is equal to 0 since, as $\vec{s} \neq 0$, the image of $\sigma_{\vec{s}}(w(1))\circ \rho_{\vec{r}, \vec{s}}$ lies in the kernel of $p_{\Ubark}$. 

\item \underline{$0 < m = n$.} 
Again applying the transpose operation to the result of Lemma \ref{lem:leftksreps}, we have in this case that $\tilde{\bar{u}} = \tilde{\bar{u}}(\lambda\varpi^{-2n})$ for some $\lambda \in \cI_{2n}$. For such $\lambda$,
\begin{align*}
\tilde{\bar{u}}(\lambda\varpi^{-2n})\tilde{h}(\varpi)^{-n}  = \tilde{w}(1)\tilde{h}(\varpi)^{-n}\tilde{u}(-\lambda)\tilde{w}(-1)(-1, \varpi^n)_F,
\end{align*}
so $\zeta = (-1, \varpi)_F$ and 
\begin{align*}
p_{\Ubark}\left(\varphi_n(\tilde{\bar{u}}\tilde{h}(\varpi)^m)(v)\right) & = (-1, \varpi^n)_F \cdot p_{\Ubark} \circ \sigma_{\vec{s}}(w(1)) \circ \rho_{\vec{r}, \vec{s}} \circ \sigma_{\vec{r}}(u(-\lambda)w(-1))(v)
\end{align*}
which again is equal to 0 since the image of $\sigma_{\vec{s}}(w(1)) \circ \rho_{\vec{r}, \vec{s}}$ lies in the kernel of $p_{\Ubark}$. 
\end{enumerate}
\end{proof}

If instead $\vec{s} = \vec{0}$ is taken in the situation of Lemma \ref{lem:satnontrivwts}, then $\sigma_{\vec{s}}$ is the trivial representation of $K$, $p_{\Ubark} = 1$, and we obtain the following statement: 

\begin{lem} \label{lem:wtr1dim}
Suppose that $(\sigrt)_{\Ubark} \cong (\tilde{\sigma}_{\vec{0}})_{\Ubark}$ as $\Tt \cap \Kt$-representations, i.e., suppose that $\vec{r} \in \{ \vec{0}, \overrightarrow{p-1}\}$. If $v \in V_{\vec{r}}$ and if the triple ($n \geq 0$, $m \in \ZZ$, and $\tilde{\bar{u}} \in (\Ubar \cap K)^* \setminus \Ubar^*)$ satisfies 
$$K^* \tilde{\bar{u}}\tilde{h}(\varpi)^m \subset K^* \tilde{h}(\varpi)^{-n}(1, \zeta)K^*,$$
then 
\begin{align*}
p_{\Ubark}\left(\phit_n(\tilde{\bar{u}}\tilde{h}(\varpi)^m)(v) \right) & = \phit_n(\tilde{\bar{u}}\tilde{h}(\varpi)^m)(v)  = \begin{cases}
v & \text{ if } m = n = 0 \text{ and } \vec{r} = \vec{s},\\
\zeta \cdot \rho_{\vec{r}, \vec{0}}v & \text{ if } 0 < n \text{ and } - n \leq  m < n,\\
\zeta \cdot \rho_{\vec{r}, \vec{0}} \circ \sigma_{\vec{r}} 
\left(w(1) \right)v & \text{ if } 0 < m = n.
\end{cases}
\end{align*}
In particular, if $\vec{r} = \vec{s} = \vec{0}$, then $p_{\Ubark}\left(\phit_n(\tilde{\bar{u}}\tilde{h}(\varpi)^m(v)) \right) = \zeta \cdot v$ for all $n$, $m$, and $\tilde{\bar{u}}$ satisfying the conditions of the lemma. 
\end{lem}

Lemmas \ref{lem:satnontrivwts} and \ref{lem:wtr1dim} imply that for any weights $\sigrt$, $\sigst$ of $\Kt$, and for a fixed triple ($n\geq 0$, $m \in \ZZ$, $\zeta\in \mu_2$), the value of $p_{\Ubar}\circ \varphi_n(\bar{u}\tilde{h}(\varpi)^m)v$ is independent of the choice of a representative $\tilde{\bar{u}}$ from the index set $ \{ \tilde{\bar{u}} \in (\Ubar \cap K)^*\setminus \Ubar^*: \, (\Ubar \cap K)^* \tilde{\bar{u}}\tilde{h}(\varpi)^m \subset K^* \tilde{h}(\varpi)^{-n}(1, \zeta)K^*\}.$ To finish the evaluation of $\mathcal{S}_{\vec{r}, \vec{s}}(\mathcal{T}_n)'(p_{\Ubark}v)(\tilde{h}(\varpi)^m)$, it only remains to count the order (mod $q$) of each such index set. These orders were determined over $\ZZ$ in Lemma \ref{lem:cartaniwasawa} (via Lemma \ref{lem:citranspose}). Reducing modulo $q$ in the formulae of Lemma \ref{lem:cartaniwasawa}, 
\begin{equation}
\label{eqn:ordercounts}
\# \{ \bar{u} \in (\Ubar \cap K)^*\setminus \Ubar^*: \, (\Ubar \cap K)^* \bar{u}\tilde{h}(\varpi)^m \subset K^* \tilde{h}(\varpi)^{-n}(1, \zeta)K^*\} \equiv
\begin{cases}
1\pmod{q} & \text{ if } m = -n \text{ and } \zeta = 1,\\
\frac{q-1}{2}\pmod{q} & \text{ if } m = -n + 1,\\
0\pmod{q} & \text{ otherwise. }
\end{cases}
\end{equation}

Hence if $\vec{s} \neq \vec{0}$, we deduce from Lemma \ref{lem:satnontrivwts} and (\ref{eqn:ordercounts}) that 
\begin{align*}
\mathcal{S}_{\vec{r}, \vec{s}}(\mathcal{T}_n)'(p_{\Ubark}v)(\tilde{h}(\varpi)^m) & = \begin{cases}
p_{\Ubark}v & \text{ if } m = n = 0 \text{ and } \vec{r} = \vec{s},\\
p_{\Ubark} \circ \rho_{\vec{r}, \vec{s}}v & \text{ if } m = -n < 0,\\
0 & \text{ otherwise}.
\end{cases}
\end{align*}

From Lemma \ref{lem:wtr1dim} and (\ref{eqn:ordercounts}) we get the following formula for $\vec{s} = \vec{0}$:  
\begin{align*}
\mathcal{S}_{\vec{r}, \vec{0}}(\mathcal{T}_n)'(p_{\Ubark}v)(\tilde{h}(\varpi)^m) & = \begin{cases}
v  & \text{ if } m = n = 0 \text{ and } \vec{r} = \vec{0},\\
\rho_{\vec{r}, \vec{0}}v & \text{ if } m = -n < 0,\\
0 & \text{ otherwise},
\end{cases}
\end{align*}
where the vanishing when $m = -n+1$ is due to the fact that 
\begin{align*}
\mathcal{S}_{\vec{r}, \vec{0}}(\mathcal{T}_n)'(p_{\Ubark}v)(\tilde{h}(\varpi)^m) & = 
|S_{n, m, 1}|\cdot \rho_{\vec{r}, \vec{0}}v - |S_{n, m, -1}| \cdot \rho_{\vec{r}, \vec{0}}v = \left(\frac{q-1}{2} \right)\rho_{\vec{r}, \vec{0}}v - \left( \frac{q-1}{2}\right)\rho_{\vec{r}, \vec{0}}v =  0. 
\end{align*}

On the other hand, $$\tau_{-n}'(p_{\Ubark}v)(\tilde{h}(\varpi)^m) = \psi^{\vec{r}, \vec{s}}_{-n}(\tilde{h}(\varpi)^m)(p_{\Ubark}v) = \begin{cases}
p_{\Ubark}v & \text{ if } n = m = 0 \text{ and } \vec{r} = \vec{s},\\
\iota_{\vec{r}, \vec{s}}(p_{\Ubark}v) & \text{ if } m = -n \text{ and } \vec{r} = \vec{s} \text{ or } \{ \vec{r}, \vec{s}\} = \{ \vec{0}, \overrightarrow{p-1}\},\\
0 & \text{ otherwise.}
\end{cases}$$

It follows from the definition of $\rho_{\vec{r}, \vec{s}}$ that $\iota_{\vec{r}, \vec{s}} \circ p_{\Ubark} = p_{\Ubark} \circ \rho_{\vec{r}, \vec{s}}$, and $p_{\Ubark} \circ \rho_{\vec{r}, \vec{0}} = \rho_{\vec{r}, \vec{0}}$. Thus the formulae for $\mathcal{S}_{\vec{r}, \vec{s}}(\mathcal{T}_n)'(p_{\Ubark}v)(\tilde{h}(\varpi)^m)$ and $\tau_{-n}'(p_{\Ubark}v)(\tilde{h}(\varpi)^m)$ agree, so  $\mathcal{S}_{\vec{r}, \vec{s}}(\mathcal{T}_n)' = \tau_{-n}'$ if $n > 0$ and $\mathcal{S}_{\vec{r}}(\mathcal{T}_{0})' = \tau_{0}'$ if $\vec{r} = \vec{s}$. Passing back to $\mathcal{H}(\Tt, \Tt\cap \Kt, (\sigrt)_{\Ubark}, (\sigst)_{\Ubark})$ through the equivalence of Frobenius reciprocity, we get the statement of the proposition. 
\end{proof}

In particular, $\mathcal{S}_{\vec{r}, \vec{s}}$ is injective. Proposition \ref{prop:satakeproperties}, together with the description of $\mathcal{H}^{\leq 0}(\Tt, \TKt, (\sigrt)_{\Ubark}, (\sigst)_{\Ubark})$ from Lemma \ref{lem:antidomisom}, gives the following corollary. 

\begin{cor}
\label{cor:injimag}
\begin{enumerate}
\item If $\vec{r} = \vec{s}$, then $\mathcal{S}_{\vec{r}, \vec{s}}$ is an $E$-algebra isomorphism $$\SHA \rightarrow E[\tau_{-1}^{\vec{r}}] \cong \mathcal{H}^{\leq 0}(\Tt, \TKt, (\sigrt)_{\Ubark}).$$ Thus $\SHA$ is a polynomial algebra over $E$ in the single operator $\mathcal{S}_{\vec{r}}^{-1}(\tau^{\vec{r}}_{-1}) = \mathcal{T}^{\vec{r}}_1.$ 
\item For each pair $\vec{r} \neq \vec{s}$, the map $\mathcal{S}_{\vec{r}, \vec{s}}$ is an $E$-linear bijection $$\SHAs \rightarrow \mathcal{H}^{<0}(\Tt, \TKt, (\sigrt)_{\Ubark}, (\sigst)_{\Ubark}),$$
and the family of maps $\{\mathcal{S}_{-, -}\}$ respects the Hecke bimodule structure on each side.
\end{enumerate}
\end{cor}

From Corollary \ref{cor:injimag} and Lemma \ref{lem:kt'weights} we get an analogous description of the genuine spherical Hecke bimodules of $\Gt$ with respect to $\Kt'$:

\begin{cor}\label{cor:injimagkt'}
Let $\mathcal{T}^{\tilde{\alpha}, \vec{r}, \vec{s}}_{1}$ (resp., $\mathcal{T}_1^{\tilde{\alpha}, \vec{r}}$) denote the image of $\mathcal{T}^{\vec{r}, \vec{s}}_1$ (resp. $\mathcal{T}_1^{\vec{r}}$) under the isomorphism of Lemma \ref{lem:kt'weights}. 
\begin{enumerate}
\item The composition of the inverse of the isomorphism of Lemma \ref{lem:kt'weights} (taking $\vec{r} = \vec{s}$) with $\mathcal{S}_{\vec{r}}$ is an $E$-algebra isomorphism 
$$\mathcal{H}(\Gt, \Kt', \sigrt^{\tilde{\alpha}}) \rightarrow \mathcal{H}^{\leq 0}(\Tt, \Tt \cap \Kt, (\sigrt)_{\Ubark}).$$
Thus $\mathcal{H}(\Gt, \Kt', \sigrt^{\tilde{\alpha}})$ is isomorphic to a polynomial algebra over $E$ in the single operator $\mathcal{T}^{\tilde{\alpha}, \vec{r}}_1$. 
\item For each pair $\vec{r} \neq \vec{s}$, there is an $E$-linear bijection
$$\mathcal{H}(\Gt, \Kt', \sigrt^{\tilde{\alpha}}, \sigst^{\tilde{\alpha}}) \rightarrow \mathcal{H}^{< 0 }(\Tt, \Tt \cap \Kt, (\sigrt)_{\Ubark}, (\sigst)_{\Ubark})$$
which is compatible with the Hecke bimodule structure on each side.
\end{enumerate}
\end{cor}

We conclude this section by using the Satake transform to calculate compositions of elements of compatible genuine spherical Hecke bimodules. 

\begin{lem} \label{lem:cowmap}
\begin{enumerate}
\item For each $n \geq 0$ and $\vec{r} \in \{0, \dots, p-1\}^f$, the following equality holds in $\SHA$: 
$$(\mathcal{T}^{\vec{r}}_1)^n = \mathcal{T}^{\vec{r}}_n.$$
\item If $\{\vec{r}, \vec{s}\} = \{ \vec{0}, \overrightarrow{p-1}\}$, then for each $n \geq 0$ and $m \geq 0$,
$$\mathcal{T}^{\vec{s}, \vec{r}}_n\circ \mathcal{T}^{\vec{r}, \vec{s}}_m = (\mathcal{T}^{\vec{r}}_1)^{n+m} \text{ and } \mathcal{T}^{\vec{r}, \vec{s}}_m \circ \mathcal{T}^{\vec{s}, \vec{r}}_n= (\mathcal{T}^{\vec{s}}_1)^{n+m}.$$
\end{enumerate}
\end{lem}

\begin{proof}[Proof of Lemma \ref{lem:cowmap}]
\begin{enumerate}
\item Since $\mathcal{S}_{\vec{r}}$ is a homomorphism of $E$-algebras (Proposition \ref{prop:gensataketransformi} (3)), 
$$\mathcal{S}_{\vec{r}}\left((\mathcal{T}^{\vec{r}}_1)^n\right) =\left( \mathcal{S}^{\vec{r}}( \mathcal{T}^{\vec{r}}_1)\right)^n.$$
By Proposition \ref{prop:satakeproperties}, $\mathcal{S}_{\vec{r}}\left( \mathcal{T}^{\vec{r}}_1\right)^n = (\tau^{\vec{r}}_{-1})^n,$ which is equal to $\tau_{-n}^{\vec{r}}$ by Lemma \ref{lem:SHTwtcoinvalg} (1).\\ Then $(\mathcal{T}^{\vec{r}}_1)^n = \mathcal{S}_{\vec{r}}^{-1}(\tau^{\vec{r}}_{-n}) = \mathcal{T}^{\vec{r}}_n$. 

\item Suppose that $\{\vec{r}, \vec{s}\} = \{ \vec{0}, \overrightarrow{p-1} \}$, and let $n \geq 0$ and $m \geq 0$. By Proposition \ref{prop:gensataketransformi} (3) again, 
$$\mathcal{S}_{\vec{r}}(\mathcal{T}^{\vec{s}, \vec{r}}_n\circ \mathcal{T}^{\vec{r}, \vec{s}}_m) = \mathcal{S}_{\vec{s}, \vec{r}}(\mathcal{T}^{\vec{s}, \vec{r}}_n) \circ \mathcal{S}_{\vec{r}, \vec{s}}(\mathcal{T}^{\vec{r}, \vec{s}}_m).$$
By Proposition \ref{prop:satakeproperties}, $$\mathcal{S}_{\vec{s}, \vec{r}}(\mathcal{T}^{\vec{s}, \vec{r}}_n) \circ \mathcal{S}_{\vec{r}, \vec{s}}(\mathcal{T}^{\vec{r}, \vec{s}}_m) = \tau^{\vec{s}, \vec{r}}_{-n} \circ \tau^{\vec{r}, \vec{s}}_{-m},$$
and by Lemma \ref{lem:SHTwtcoinvalg} (2), $ \tau^{\vec{s}, \vec{r}}_{-n} \circ \tau^{\vec{r}, \vec{s}}_{-m} = \tau^{\vec{r}}_{-(n+m)}.$
Then $\mathcal{T}^{\vec{s}, \vec{r}}_n \circ \mathcal{T}^{\vec{r}, \vec{s}}_m = \mathcal{S}_{\vec{r}}^{-1}(\tau_{-(n+m)}^{\vec{r}}) = \mathcal{T}^{\vec{r}}_{n+m}$, which is equal to $(\mathcal{T}^{\vec{r}}_{1})^{n+m}$ by (1) of the present lemma. The second equality of (2) is proved in the same way after exchanging $\vec{s}$ and $\vec{r}$.  
\end{enumerate}
\end{proof}

\section{Explicit description of the Hecke action}
\label{sec:heckedescription}

In this section we give an explicit formula for the action of $\mathcal{H}(\Gt, \Kt, \tilde{\sigma}_{\vec{r}})$ on the compact induction $\ind_{\Kt}^{\Gt}(\tilde{\sigma}_{\vec{r}})$. It is enough to specify the action of the generator $\mathcal{T}_1^{\vec{r}}$ on the basic functions $[g, v]$ for $g \in S$ and $v \in V_{\vec{r}}$. After the proof of Proposition \ref{prop:heckeaction}, we give a second description of this action in terms of the Bruhat-Tits tree of $SL_2(F)$. In the following section we derive some consequences, notably the freeness of $\ind_{\Kt}^{\Gt}(\tilde{\sigma}_{\vec{r}})$ as an $\mathcal{H}(\Gt, \Kt, \tilde{\sigma}_{\vec{r}}$-module. 

\subsection{Formula for the Hecke action}
\label{subsec:heckeformula}

\begin{prop}
\label{prop:heckeaction}
Let $v \in V_{\vec{r}}$. Let $\mathcal{T}^{\vec{r}}$ denote the operator $\mathcal{T}^{\vec{r}}_1$ defined in (\ref{eqn:tprimedef}), and let $\rho_{\vec{r}}$ denote the endomorphism of $\tilde{\sigma}_{\vec{r}}$ defined in Lemma \ref{lem:gshabasis}. Unless otherwise noted, let $n \geq 1$. 
\begin{enumerate}
\item
\begin{align*}
(-1, \varpi)_F \cdot \mathcal{T}^{\vec{r}}\left([(1, 1), v] \right) = & \sum_{\lambda \in I_2} [\tilde{h}^0_{2, \lambda}, \, \rho_{\vec{r}}\circ \sigma_{\vec{r}}\left(u(-\lambda) \right)v] + \sum_{0 \neq \lambda \in \cI_1} [\tilde{h}^1_{1, \lambda}, \, \rho_{\vec{r}} \circ \sigma_{\vec{r}}\left(w(-\lambda)\right)v]\\
& + [\tilde{h}^1_{1, 0}, \, \sigma_{\vec{r}}(w(1)) \circ \rho_{\vec{r}} \circ \sigma_{\vec{r}}(w(-1))v].
\end{align*}

\item Let $\kappa \in \cI_{2n}$. Then
\begin{align*}
(-1, \varpi)_F \cdot \mathcal{T}^{\vec{r}}\left([\tilde{h}^0_{2n, \kappa},\, v]\right) & = \sum_{\lambda \in \cI_2} [\tilde{h}^0_{2n + 2, \kappa + \lambda\varpi^{2n}}, \,  \rho_{\vec{r}} \circ \sigma_{\vec{r}}\left(u(-\lambda) \right)v] + \sum_{0 \neq \lambda \in \cI} \eta_1(\lambda) \cdot [\tilde{h}^0_{2n, \kappa_{+\lambda}}, \rho_{\vec{r}} \circ \sigma_{\vec{r}}(w(-\lambda))v]\\
& + (-1, \varpi)_F \cdot [\tilde{h}^0_{2n-2, [\kappa]_{2n-2}}, \, \sigma_{\vec{r}}\left(u \left(\frac{\kappa-[\kappa]_{2n-2}}{\varpi^{2n-2}}\right)w(1) \right) \circ \rho_{\vec{r}}\circ \sigma_{\vec{r}}\left(w(-1) \right)v ].
\end{align*}

\item Let $n \geq 2$ and let $\kappa \in \cI_{2n-1}$ such that $0 \leq v_F(\kappa) \leq 2n-3$. Then
\begin{align*}
(-1, \varpi)_F \cdot \mathcal{T}^{\vec{r}}\left([\tilde{h}^1_{2n-1, \kappa}, \, v] \right)= &  \sum_{\lambda \in \cI_2} [\tilde{h}^1_{2n+1, \kappa + \varpi^{2n-1}\lambda}, \, \rho_{\vec{r}} \circ \sigma_{\vec{r}}(u(-\lambda)v] + \sum_{0 \neq \lambda \in \cI} \eta_1(\lambda) \cdot [\tilde{h}^1_{2n-1, \kappa_{+\lambda}},  \, \rho_{\vec{r}}\circ \sigma_{\vec{r}}(w(-\lambda))v]\\
& + (-1, \varpi)_F \cdot[\tilde{h}^1_{2n-3, [\kappa]_{2n-3}}, \,\sigma_{\vec{r}}\left(u \left(\frac{\kappa-[\kappa]_{2n-3}}{\varpi^{2n-3}}\right)w(1) \right) \circ \rho_{\vec{r}}\circ \sigma_{\vec{r}}\left(w(-1) \right)v ].
\end{align*}

\item Let $\kappa \in \cI_{2n-1}$ such that $v_F(\kappa)= 2n-2$, and write $\kappa_{2n-2}$ for the unique element of $\mathfrak{k}$ such that $\kappa\varpi^{-2n+2} \equiv [\kappa_{2n-2}] \pmod{\varpi}$.
\begin{align*}
(-1, \varpi)_F\cdot \mathcal{T}^{\vec{r}}\left([\tilde{h}^1_{2n-1, \kappa},\, v] \right) =& \sum_{\lambda \in \cI_2} [\tilde{h}^1_{2n+1, \kappa + \varpi^{2n-1}\lambda}, \, \rho_{\vec{r}}\circ \sigma_{\vec{r}}(u(-\lambda))v]\\
& +\sum_{\substack{\lambda \in \cI_1\\ \lambda \notin \{0, -[\kappa_{2n-2}]\}}} \left((-\lambda^{-1} - [\kappa_{2n-2}]^{-1}), \varpi) \right)_F [\tilde{h}^1_{2n-1, \kappa_{+\lambda}},  \rho_{\vec{r}} \circ \sigma_{\vec{r}}(w(-\lambda))v]\\
& + \eta_n(\kappa) \cdot [\tilde{h}^1_{2n-1, 0}, \, \sigma_{\vec{r}}\left(u\left( \frac{\kappa - [\kappa]_{2n-3}}{\varpi^{2n-3}}\right) w(1)\right)\circ  \rho_{\vec{r}} \circ \sigma_{\vec{r}}(w(-1))v]\\
&  +(-1, \varpi)_F \cdot  [\tilde{h}^1_{2n-3, 0}, \, \rho_{\vec{r}} \circ \sigma_{\vec{r}}(w([\kappa_{2n-2}]))v]. 
\end{align*}

\item
\begin{align*}
(-1, \varpi)_F \cdot \mathcal{T}^{\vec{r}}\left([\tilde{h}^1_{2n-1, 0}, \, v] \right) = & \sum_{\substack{\lambda \in \cI_2\\ v_F(\lambda) = 0}} [\tilde{h}^1_{2n+1, \lambda\varpi^{2n-1}}, \, \rho_{\vec{r}} \circ \sigma_{\vec{r}}\left( u\left(-\lambda\right)\right)v ] + \sum_{0 \neq \lambda \in \cI_1} [\tilde{h}^1_{2n+1, \lambda\varpi^{2n}}, \, \rho_{\vec{r}} \circ \sigma_{\vec{r}}(w(-\lambda))v ]\\
& +[\tilde{h}^1_{2n+1, 0}, \, \sigma_{\vec{r}} \left(w(1) \right) \circ \rho_{\vec{r}} \circ \sigma_{\vec{r}}(w(-1))v] + \sum_{\substack{\lambda \in \cI_2\\ v_F(\lambda) = 1}} (\lambda, \varpi)_F \cdot [\tilde{h}^1_{2n-1, \lambda \varpi^{2n-3}}, \,  \rho_{\vec{r}}v]\\
& + (-1, \varpi)_F \cdot [\tilde{h}^1_{2n-3, 0}, \, \rho_{\vec{r}}v]. 
\end{align*}

\end{enumerate}
\end{prop}

\begin{proof}[Proof of Proposition \ref{prop:heckeaction}]
In this proof we will write $\varphi^{\vec{r}}$ for the function $\varphi_{1}^{\vec{r}}$ defined in Lemma \ref{lem:gshabasis}. Since $\mathcal{T}^{\vec{r}}$ corresponds to $\varphi^{\vec{r}}$ by Frobenius reciprocity, 
\begin{equation}
\label{eqn:tphifrobrec}
\mathcal{T}^{\vec{r}}([g, v]) = \sum_{g'\Kt \in \Gt/\Kt}[gg', \varphi^{\vec{r}}(g'^{-1})(v)] = \sum_{g' \in S}[gg', \varphi^{\vec{r}}(g'^{-1})(v)]
\end{equation}
for each $g \in \Gt$ and $v \in V_{\vec{r}}$. 

\begin{lem}
\label{lem:phiinvertreps}
\begin{enumerate}
\item Let $\lambda \in \mathcal{I}_2$. Then
$$\varphi^{\vec{r}}\left((\tilde{h}^0_{2, \lambda})^{-1} \right) = (-1, \varpi)_F \cdot \rho_{\vec{r}} \circ \sigma_{\vec{r}} \left( u(-\lambda)\right).$$

\item Let $0 \neq \lambda \in \mathcal{I}_1$. Then
$$\varphi^{\vec{r}}\left( (\tilde{h}^1_{1,\lambda})^{-1}\right) = (-1, \varpi)_F \cdot \rho_{\vec{r}} \circ \sigma_{\vec{r}}(h(-\lambda)w(1)).$$

\item $$\varphi^{\vec{r}} \left( (\tilde{h}^1_{1, 0})^{-1}\right) = (-1, \varpi)_F \cdot \sigma_{\vec{r}}\left(w(1) \right) \circ \rho_{\vec{r}} \circ \sigma_{\vec{r}}\left(w(-1) \right). $$
\end{enumerate}
\end{lem}

\begin{proof}[Proof of Lemma \ref{lem:phiinvertreps}]
\begin{enumerate}
\item For $\lambda \in \cI_2$,
$$\tilde{h}^0_{2,\lambda} := \tilde{h}(\varpi)\tilde{u}(\lambda\varpi^{-2})(1, (-1, \varpi)_F) = \tilde{u}(\lambda)\tilde{h}(\varpi)(1, (-1, \varpi)_F),
$$
so 
$$\varphi^{\vec{r}}\left((\tilde{h}^0_{2, \lambda})^{-1} \right) = (-1, \varpi)_F \cdot \varphi^{\vec{r}}\left(\tilde{h}(\varpi)^{-1} \right) \circ \sigma_{\vec{r}} \left( u(-\lambda)\right) = (-1, \varpi)_F \cdot \rho_{\vec{r}} \circ \sigma_{\vec{r}} \left( u(-\lambda)\right) $$

\item For $0 \neq \lambda \in \cI_1$,
$$\tilde{h}^1_{1, \lambda} : = \tilde{u} \left(\lambda\varpi^{-1} \right) \cdot (1, \eta_1(\lambda)) = 
\tilde{\bar{u}}(\lambda^{-1}\varpi) \tilde{h}(\varpi)^{-1}\tilde{u}(\lambda\varpi)\tilde{w}(-1)\tilde{h}(-\lambda^{-1}),$$
so
\begin{align*}
\varphi^{\vec{r}} \left((\tilde{h}^1_{1, \lambda})^{-1} \right) & = \varphi^{\vec{r}} \left(\tilde{h}(-\lambda)\tilde{w}(1)\tilde{u}(-\lambda\varpi)\tilde{h}(\varpi)\tilde{\bar{u}}(-\lambda^{-1}\varpi) \right)\\
& = (-1, \varpi)_F \cdot \sigma_{\vec{r}}\left(h(-\lambda)\bar{u}(\lambda\varpi)\right) \circ \varphi^{\vec{r}}(\tilde{h}(\varpi)^{-1})\circ \sigma_{\vec{r}}\left(u(\lambda^{-1}\varpi)w(1)\right).
\end{align*}

Using the fact that $\sigma_{\vec{r}}$ factors through reduction modulo $\varpi$, we get
$$\varphi^{\vec{r}}\left( (\tilde{h}^1_{1,\lambda})^{-1}\right) = (-1, \varpi)_F \cdot \varphi^{\vec{r}}\left(\tilde{h}(\varpi)^{-1} \right) \circ \sigma_{\vec{r}}(w(-\lambda)) = (-1, \varpi)_F \cdot \rho_{\vec{r}} \circ \sigma_{\vec{r}}(w(-\lambda)).$$

\item Finally $\tilde{h}^1_{1, 0} = \tilde{h}(\varpi)^{-1} = \tilde{w}(1)\tilde{h}(\varpi)\tilde{w}(-1)\cdot (1, (-1, \varpi)_F),$ so 
$$ \varphi^{\vec{r}} \left( (\tilde{h}^1_{1, 0})^{-1}\right) = (-1, \varpi)_F \cdot \sigma_{\vec{r}}\left(w(1) \right) \circ \varphi^{\vec{r}}(\tilde{h}(\varpi)^{-1}) \circ \sigma_{\vec{r}}\left(w(-1) \right) = (-1, \varpi)_F \cdot \sigma_{\vec{r}}\left( w(1)\right) \circ \rho_{\vec{r}} \circ \sigma_{\vec{r}}\left(w(-1)\right). $$
\end{enumerate}
\end{proof}

For each pair $g$, $g' \in S\times S$, the following lemma gives an expression for the product $gg'$ of the form $hk$, where $h \in S$ and $k \in \Kt$. 

\begin{lem}
\label{lem:sproducts}
Let $n \geq 1$. 
\begin{enumerate}
\item Let $\kappa \in \cI_{2n}$. Then:\\
\begin{enumerate}
\item for $\lambda \in \cI_2$,
$$\tilde{h}^0_{2n, \kappa}\tilde{h}^0_{2, \lambda} = \tilde{h}^0_{2n+2, \kappa + \varpi^{2n}\lambda};$$
\item for $0 \neq \lambda \in \cI_1$, 
\begin{align*}
\tilde{h}^0_{2n, \kappa}\tilde{h}^1_{1, \lambda} &= \tilde{h}^0_{2n, \kappa_{+\lambda}}\tilde{u}(x_{\lambda}) \cdot (1, \eta_1(\lambda))
\end{align*}
for some $x_{\lambda} \in \OF$;
\item 
$$\tilde{h}^0_{2n, \kappa}\tilde{h}^1_{1, 0} = \tilde{h}^0_{2n-2, [\kappa]_{2n-2}}\tilde{u}\left(\frac{\kappa - [\kappa]_{2n-2}}{\varpi^{2n-2}} \right)\cdot (1, (-1, \varpi)_F).$$\\
\end{enumerate}

\item Let $0 \neq \kappa \in \cI_{2n-1}$. Then:\\
\begin{enumerate}
\item for $\lambda \in \cI_2$,
$$ \tilde{h}^1_{2n-1, \kappa}\tilde{h}^0_{2, \lambda} = \tilde{h}^1_{2n+1, \,\kappa + \lambda\varpi^{2n-1}};$$
\item for $0 \neq \lambda \in \cI_1,$
\begin{enumerate}
\item if $0 \leq v_F(\kappa) \leq  2n-3$, 
\begin{align*} \tilde{h}^1_{2n-1, \kappa}\tilde{h}^1_{1, \lambda} & = \tilde{h}^1_{2n-1,\kappa_{+\lambda}}\tilde{u}(x_{\lambda})\cdot(1, \eta_1(\lambda))
\end{align*}
for some $x_{\lambda} \in \OF$;
\item if $v_F(\kappa) = 2n-2$ and $\kappa \varpi^{-2n+2} \not \equiv - \lambda \pmod{\varpi}$, 
\begin{align*}
\tilde{h}^1_{2n-1, \kappa}\tilde{h}^1_{1, \lambda}& =  \tilde{h}^1_{2n-1,\kappa_{+\lambda}} \tilde{u}(x_{\lambda})\cdot (1, \eta_n(\kappa)\eta_1(\lambda)\eta_n(\kappa_{+\lambda}))\\
& = \tilde{h}^1_{2n-1,\kappa_{+\lambda}} \tilde{u}(x_{\lambda})\left(1, \left(-\lambda^{-1}- [\kappa_{2n-2}]^{-1}, \varpi\right)_F\right)
\end{align*}
for some $x_{\lambda} \in \OF$;
\item if $v_F(\kappa) = 2n-2$ and $\kappa \varpi^{-2n+2} \equiv -\lambda \pmod{\varpi}$,
$$ \tilde{h}^1_{2n-1, \kappa}\tilde{h}^1_{1, \lambda} = \tilde{h}^1_{2n-3, 0} \tilde{u}(x_{\lambda})\cdot (1, (-1, \varpi)_F)$$
for some $x_{\lambda} \in \OF$.\\
\end{enumerate}

\item 
\begin{enumerate}
\item If $0 \leq v_F(\kappa) \leq 2n-3$, then 
$$ \tilde{h}^1_{2n-1, \kappa}\tilde{h}^1_{1, 0} = \tilde{h}^1_{2n-3, [\kappa]_{2n-3}}\tilde{u}\left(\frac{\kappa - [\kappa]_{2n-3}}{\varpi^{2n-3}} \right) \cdot (1, (-1, \varpi)_F);$$

\item if $v_F(\kappa) = 2n-2$, then
\begin{align*}
\tilde{h}^1_{2n-1, \kappa} \tilde{h}^1_{1, 0}& = \tilde{h}^1_{2n-1, 0} \tilde{u}\left(\frac{\kappa - [\kappa]_{2n-3}}{\varpi^{2n-3}}\right) \cdot (1, \eta_n(\kappa)). 
\end{align*}
\end{enumerate}
\end{enumerate}

\item
\begin{enumerate}
\item for $\lambda \in \cI_2$ such that $v_F(\lambda) = 0$,
$$\tilde{h}^1_{2n-1, 0}\tilde{h}^0_{2, \lambda} = \tilde{h}^1_{2n+1,\lambda \varpi^{2n-1}};$$
\item for $\lambda \in \cI_2$ such that $v_F(\lambda) = 1$,
\begin{align*}
\tilde{h}^1_{2n-1, 0}\tilde{h}^0_{2, \lambda} & = \tilde{h}^1_{2n-1, \lambda\varpi^{2n-3}}(1, (\lambda, \varpi)_F);
\end{align*}
\item
$$\tilde{h}^1_{2n-1, 0}\tilde{h}^0_{2, 0} = \tilde{h}^1_{2n-3, 0}\cdot (1, (-1, \varpi)_F);$$
\item for $\lambda \in \cI_1$ (either 0 or not),
$$\tilde{h}^1_{2n-1,0} \tilde{h}^1_{1, \lambda} = \tilde{h}^1_{2n+1, \lambda\varpi^{2n}}.$$
\end{enumerate}

\end{enumerate}
\end{lem}

Proposition \ref{prop:heckeaction} now follows from (\ref{eqn:tphifrobrec}) by applying Lemmas \ref{lem:phiinvertreps} and \ref{lem:sproducts} together with the relation 
$$[gk, v] = [g, \tilde{\sigma}_{\vec{r}}(k)v] \text{ for all } g \in \Gt, \, k \in \Kt, \, v \in V_{\vec{r}}.$$
In particular, the terms $x_{\lambda} \in \OF$ which appear in Lemma \ref{lem:sproducts} disappear because the image of $\rho_{\vec{r}}$ is contained in the $(U \cap K)^*$-invariant subspace of $V_{\vec{r}}$. 
\end{proof}

\begin{cor}
\label{cor:supportradius}
Let $0 \neq f \in \ind_{\Kt}^{\Gt}(\tilde{\sigma}_{\vec{r}})$ and let $n := \text{max}\left(\{i \in \NN: \, \text{Supp}(f) \cap \Kt \tilde{h}(\varpi)^{i}\Kt \neq \emptyset\}\right)$. Then
$$\text{max}\left(\{i \in \NN : \, \text{Supp}(\mathcal{T}^{\vec{r}}(f)) \cap \Kt \tilde{h}(\varpi)^{i}\Kt \neq \emptyset\} \right) = n + 1.$$ 
\end{cor}

\begin{proof}[Proof of Corollary \ref{cor:supportradius}]
It is clear from Proposition \ref{prop:heckeaction} that $\text{Supp}(\mathcal{T}^{\vec{r}}(f)) \subset \coprod_{0 \leq i \leq n+1} \Kt \tilde{h}(\varpi)^{i}\Kt$, so we only have to prove that the intersection of $\text{Supp}(\mathcal{T}^{\vec{r}}(f))$ with $\Kt \tilde{h}(\varpi)^{n + 1}\Kt$ is nonempty.

We first define a certain set $C_g \subset \Kt \tilde{h}(\varpi)^{i + 1}\Kt$ for each $g \in S_i$, $i \geq 0$:
\begin{align*}
C_g & := 
\begin{cases}
 \{(\tilde{h}^0_{2i+2, \kappa + \varpi^{2i}\lambda})^{-1}: \, \lambda \in \cI_2& \text{ if } i \geq 1 \text{ and } g = \tilde{h}^0_{2i, \kappa} \in S^0_{i},\\
 \{(\tilde{h}^1_{2i+1, \kappa +\varpi^{2i-1}\lambda})^{-1}: \, \lambda \in \cI_2 & \text{ if } i \geq 1 \text{ and } g = \tilde{h}^1_{2i-1, \kappa} \in S^1_i,\\
 \{h^{-1}: \, h \in S_1\} & \text{ if } g = (1, 1).
\end{cases}
\end{align*}
It follows from Lemma \ref{lem:leftksreps} that the union $\coprod_{g \in S_i}\coprod_{h \in C_g}\Kt h$ is disjoint and equal to $\Kt \tilde{h}(\varpi)^{i + 1}\Kt$, and it follows from Proposition \ref{prop:heckeaction} that for each $i \geq 0$ and $g \in S_i$ we have $\text{Supp}(\mathcal{T}^{\vec{r}}[g, v]) \cap \Kt \tilde{h}(\varpi)^{i + 1}\Kt \subset \coprod_{h \in C_g}\Kt h$. 

Since $\text{Supp}(f) \cap \Kt \tilde{h}(\varpi)^{n}\Kt \neq \emptyset$, there exists some $g \in S_n$ such that $f(g^{-1}) \neq 0$. Let $v = f(g^{-1})$; then $f(g^{-1}) = [g, v](g^{-1})$. Now let $h \in C_g$. Since $\text{Supp}(f) \cap \Kt \tilde{h}(\varpi)^{i}\Kt = \emptyset$ for $i > n$, it follows from Proposition \ref{prop:heckeaction} that 
\begin{align*}
\mathcal{T}^{\vec{r}}(f)(h) = \sum_{g' \in S_n} \mathcal{T}^{\vec{r}}([g', f(g'^{-1})])(h) = \sum_{\substack{g' \in S_n \\ h \in C_{g'}}}\mathcal{T}^{\vec{r}}([g', f(g'^{-1})])(h) = \mathcal{T}^{\vec{r}}([g, v])(h).
\end{align*}

We claim that $\sum_{h \in C_g}\mathcal{T}^{\vec{r}}([g, v])(h) \neq 0$. If $\sum_{h \in C_g} \mathcal{T}^{\vec{r}}([g, v])(h) = 0$, then by Proposition \ref{prop:heckeaction} we would have
$$v \in \bigcap_{\substack{\lambda \in \cI_2\\ v_F(\lambda)= 0}} \rho_{\vec{r}} \circ \sigma_{\vec{r}}(u(-\lambda)) v.$$ 
Since $\sigma_{\vec{r}}$ factors through reduction modulo $\varpi$, this is equivalent to having
$$v \in \bigcap_{a \in \mathfrak{k}}\rho_{\vec{r}} \circ \sigma_{\vec{r}}(u([a])).$$
But $\rho_{\vec{r}} \circ \sigma_{\vec{r}}\left(u([a])\right)$ is exactly the functional $L_{a}$ defined in \cite{barthellivne:irredmodp} $\S$5, and by Lemma 21 of \textit{loc.~cit.} the intersection $\bigcap_{a \in \mathfrak{k}} \text{ker}(L_{a})$ is trivial. By assumption $v = f(g^{-1}) \neq 0$, and hence there exists an $h \in C_g$ such that $\mathcal{T}^{\vec{r}}([g, v])(h) \neq 0$. Therefore $h \in \text{Supp}(\mathcal{T}^{\vec{r}}(f)) \cap \Kt \tilde{h}(\varpi)^{n+1}\Kt,$ so the intersection is nonempty. 
\end{proof}

\subsection{Description in terms of the tree of $SL_2(F)$}
\label{subsec:tree}

Let $\mathcal{X}$ denote the Bruhat-Tits tree of $G = SL_2(F)$ (see \cite{serre:trees} for a reference). Recall that the vertices of $\mathcal{X}$ are identified with the stabilizers of homothety classes of lattices in $F \times F$, and that two vertices $v$, $v'$ are connected by an edge if and only if there is a lattice $L$ belonging to the class stabilized by $v$, and a lattice $L'$ belonging to the class stabilized by $v'$, such that $\varpi L \subset L' \subset L$. Let $d(v_1, v_2)$ denote the integer-valued distance between two vertices $v_1$, $v_2$, defined as the number of edges traversed in the shortest path from $v_2$ to $v_2$; the action of $G$ on $\text{Vert}(\mathcal{X})$ preserves this distance. Let $v_0$ denote the vertex corresponding to $K$, the stabilizer of the homothety class of $\OF \times \OF$; then the set $\{v \in \text{Vert}(\mathcal{X}): \, 2 \big \vert d(v_0, v)\}$ is identified with $G/K$ as a $G$-module, while the set $\{v \in \text{Vert}(\mathcal{X}): \, 2 \not \big \vert d(v_0, v)\}$ is identified with $G/K'$ as a $G$-module. 

We identify the quotient $\Gt/\Kt$ with $G/K$ by identifying $\tilde{h}\Kt$ with $Pr(\tilde{h})K$ for each $\tilde{h} \in S$, and hence at the level of sets we can identify $\Gt/\Kt$ with the vertices of $\mathcal{X}$ lying at even distance from $v_0$. In particular the circle $C_{2n}$ of vertices lying at radius $2n$ from $v_0$ is identified with $S_n$, the distance between $\tilde{h}^i_{m, \kappa}$ and $\tilde{h}^i_{m, \kappa'}$ is given by $2\left(m - v_F(\kappa - \kappa')\right)$ for $\kappa \neq \kappa' \in \cI_{m}$, and $d(\tilde{h}^i_{m, \kappa}, \tilde{h}^i_{m-2, \kappa'}) = 2$ if and only if $\kappa' = [\kappa]_{m-2}$.  

For any weight $\tilde{\sigma}_{\vec{r}}$ of $\Kt$, we can now view the compact induction $\ind_{\Kt}^{\Gt}(\tilde{\sigma}_{\vec{r}})$ as the space of compactly supported sections of a local system on $\mathcal{X}$, as follows. Put an equivalence relation $\sim$ on $\Gt \times V_{\vec{r}}$ by setting $(gk, w) \sim (g, \tilde{\sigma}_{\vec{r}}(k)w)$ for each $g \in \Gt$, $k \in \Kt$, and $w \in V_{\vec{r}}$; then the fiber of our desired local system over a vertex $v$ of $\mathcal{X}$ is empty if $d(v_0, v)$ is odd, while if $d(v_0, v)$ is even and $\tilde{h} \in S$ is identified with $v$, then the fiber over $v$ is equal to the set of equivalence classes $[(\tilde{h}k, w)]$ in $\Gt \times V_{\vec{r}} / \sim$ as $k$ runs over $\Kt$ and $w$ runs over $V_{\vec{r}}$. 

Now the action of $\mathcal{T}^{\vec{r}}_1$ has the following description in terms of $\mathcal{X}$. If $[g, w]$ is a basic function in $\ind_{\Kt}^{\Gt}(\sigrt)$ and $v \in \text{Vert}(\mathcal{X})$ is the vertex of $\mathcal{X}$ identified with $g\Kt$, then Proposition \ref{prop:heckeaction} says that the support of $\mathcal{T}^{\vec{r}}_1([g, w])$ is contained in the set $\{v' \in \text{Vert}(\mathcal{X}): \, d(v', v) = 2\}$. If $v \in C_{2n}$ for some $n > 0$, then the latter set consists of $q^2 + q$ vertices: namely, $q^2$ vertices of $C_{2n+2}$, $q-1$ vertices of $C_{2n} \setminus \{v\}$, and one vertex in $C_{2n-2}$, while if $v = v_0$ then the set is just $C_{2n + 2}= C_2$. The Corollary \ref{cor:supportradius} implies that the support of $\mathcal{T}^{\vec{r}}_1([g, w])$ contains at least one of the vertices in $C_{2n+2}$, and hence that $\mathcal{T}^{\vec{r}}_1$ expands, by exactly two increments, the radius of support of any compactly supported section of the local system on $\mathcal{X}$.

\section{Universal modules for genuine spherical Hecke algebras}

\subsection{Universal modules for spherical Hecke algebras of $\Tt$} The universal module for the genuine spherical Hecke algebra $\SHT$ is the compact induction $\ind_{\Tt \cap \Kt}^{\Tt}\left((\sigrt)_{\Ubark}\right)$. Its structure is very simple and is presented here mainly for easy reference in the proofs of Proposition \ref{prop:SHAaction} and Theorem \ref{thm:pswtisom}. 

\begin{lem}
\label{lem:SHTfree}
Let $\sigrt$ be any weight of $\Kt$. Then $\ind_{\Tt \cap \Kt}^{\Tt}\left((\sigrt)_{\Ubark}\right)$ is a free $\SHT$-module. A free basis is given by the single element $[1, p_{\Ubark}v]$, where $v$ is any nonzero vector in $V_{\vec{r}}^{\Ubark}.$ 
\end{lem}

\begin{proof}[Proof of Lemma \ref{lem:SHTfree}]
Let $0 \neq v \in V_{\vec{r}}^{\Ubark}$; then $p_{\Ubark}v \neq 0$. An $E$-vector space basis for $\ind_{\Tt \cap \Kt}(\sigrt)_{\Ubark}$ is given by $\{ [\tilde{h}(\varpi)^n, p_{\Ubark}v]: n \in \ZZ\}$. For $n \in \ZZ$, 
\begin{align*}
(\tau^{\vec{r}}_{-1})^n([1, p_{\Ubark}v]) = [\tilde{h}(\varpi)^n, p_{\Ubark}v],
\end{align*}
so $[1, p_{\Ubark}v]$ is a free $\SHT$-basis for $\ind_{\Tt \cap \Kt}^{\Tt}(\sigrt)_{\Ubark}.$ 
\end{proof}

Given a genuine representation $\pi$ of $\Tt$, the Hecke algebra $\SHT$ acts on the weight space $\Hom_{\Tt \cap \Kt}((\sigrt)_{\Ubark}, \pi \big \vert_{\Tt \cap \Kt})$ from the right. Likewise, $\mathcal{H}(\Tt, \Tt \cap \Kt', (\sigrt^{\tilde{\alpha}})_{\Ubark})$ acts on $\Hom_{\Tt \cap \Kt'}((\sigrt^{\tilde{\alpha}})_{\Ubark}, \pi \big \vert_{\Tt \cap \Kt'})$. These actions are clearly scalar; we next determine the eigenvalue of the respective generators. 

\begin{lem}
\label{lem:shtheckeaction}
Let $\pi$ be any smooth genuine representation of $\Tt$ and let $\sigrt$ be a weight of $\Kt$. Then
\begin{enumerate}
\item For each $f' \in \Hom_{\Tt \cap \Kt}((\sigrt)_{\Ubark}, \pi\big \vert_{\Tt \cap \Kt})$, 
$$f' \cdot \tau_{-1}^{\vec{r}} = \pi(\tilde{h}(\varpi))\cdot f'.$$
\item For each $f' \in \Hom_{\Tt \cap \Kt'}((\sigrt^{\tilde{\alpha}})_{\Ubark}, \pi\big \vert_{\Tt \cap \Kt'})$, 
$$f' \cdot \tau_{-1}^{\tilde{\alpha}, \vec{r}} = \pi(\tilde{h}(\varpi)) \cdot f'.$$
\end{enumerate}
\end{lem}

\begin{proof}[Proof of Lemma \ref{lem:shtheckeaction}]
\begin{enumerate}
\item Let $f'$ be any element of $\Hom_{\Tt \cap \Kt}((\sigrt)_{\Ubark}, \pi \big \vert_{\Tt \cap \Kt})$ and let $v \in V_{\vec{r}}$.
\begin{align*}  (f' \cdot \tau^{\vec{r}}_{-1})(p_{\Ubark}v) &= (f' * \psi_{-1}^{\vec{r}})(p_{\Ubark}v)\\
& = \sum_{t \in (\Tt \cap \Kt) \setminus \Tt}\pi (t^{-1}) \cdot f'\left(\psi_{-1}^{\vec{r}}(t)(p_{\Ubark}v)\right)\\
 & = \sum_{n \in \ZZ} \pi (\tilde{h}(\varpi)^{-n}) \cdot f'\left(\psi_{-1}^{\vec{r}}(\tilde{h}(\varpi)^n)(p_{\Ubark}v) \right)\\
& =\pi (\tilde{h}(\varpi)) \cdot f'\left(\psi_{-1}^{\vec{r}}(\tilde{h}(\varpi)^{-1})( p_{\Ubark}v) \right)\\
& =\pi (\tilde{h}(\varpi)) \cdot f'(p_{\Ubark}v). 
\end{align*}

\item The proof is essentially identical to that of part (1). 

\end{enumerate}
\end{proof}

Finally we check that each genuine character of $\Tt$ can be constructed as the tensor product of a universal module for some $\SHT$ with a certain character of a spherical Hecke algebra. The following lemma is a prototype for the parametrization of genuine representations of $\Gt$ defined in the next section. 

\begin{lem}
\label{lem:shtquotients}
If $\pi$ is a smooth genuine character of $\Tt$, $\sigrt$ is a weight of $\Kt$ such that $\Hom_{\Tt \cap \Kt}((\sigrt)_{\Ubark}, \pi \big \vert_{\Tt \cap \Kt}) \neq 0$, $\theta$ is the character of $\SHT$ defined by $\theta(\tau^{\vec{r}}_{-1})= \pi(\tilde{h}(\varpi))$, and $\theta^{\tilde{\alpha}}$ is the character of $\mathcal{H}(\Tt, \Tt \cap \Kt', (\tilde{\sigma}^{\tilde{\alpha}}_{\vec{r}'})_{\Ubark})$ defined by $\theta^{\tilde{\alpha}}(\tau^{\tilde{\alpha},\vec{r}}_{-1}) = \pi(\tilde{h}(\varpi))$, then $$\pi \cong \ind_{\Tt \cap \Kt}^{\Tt}(\sigrt)_{\Ubark} \otimes_{\SHT}\theta \cong \ind_{\Tt \cap \Kt'}^{\Tt}(\tilde{\sigma}^{\tilde{\alpha}}_{\vec{r}'})_{\Ubark}\otimes_{\mathcal{H}(\Tt, \Tt \cap \Kt', (\tilde{\sigma}^{\tilde{\alpha}}_{\vec{r}'})_{\Ubark})}\theta^{\tilde{\alpha}}$$
where $\vec{r'} \in \{0, \dots, p-1\}^f$ satisfies $\sum_{i = 0}^{f-1}r_i'p^i \equiv \sum_{i = 0}^{f-1}(r_i + \frac{p-1}{2})p^i \pmod{q-1}.$ 
\end{lem}

\begin{proof}[Proof of Lemma \ref{lem:shtquotients}]
The genuine character $\pi$ is determined by the data $\pi \big \vert_{\Tt \cap \Kt}$ and $\pi(\tilde{h}(\varpi))$. If and only if $\Hom_{\Tt \cap \Kt}((\sigrt)_{\Ubark}, \pi \big \vert_{\Tt \cap \Kt}) \neq 0$, we have $\pi \big \vert_{\Tt \cap \Kt} \cong (\sigrt)_{\Ubark}$ as $\Tt \cap \Kt$-representations. In addition, by Lemma \ref{lem:kt'coinvs} (2) we have $\pi \big \vert_{\Tt \cap \Kt'} \cong (\tilde{\sigma}_{\vec{r'}}^{\tilde{\alpha}})_{\Ubark}$ as $\Tt \cap \Kt = \Tt \cap \Kt'$-representations if and only if $\vec{r'}$ satisfies the condition given in the lemma. By definition of $\theta$ and $\theta^{\tilde{\alpha}}$, the value of each tensor product on $\tilde{h}(\varpi)$ matches that of $\pi$. The result follows from the universal property of the tensor product. 
\end{proof}

\subsection{Universal modules for spherical Hecke algebras of  $\Gt$} 
\label{subsec:univmodsha}
The universal module for the genuine spherical Hecke algebra $\SHA$ is the compact induction $\ind_{\Kt}^{\Gt}(\sigrt)$. We continue to write $\mathcal{T}^{\vec{r}}$ for the operator $\mathcal{T}^{\vec{r}}_1$. 

\begin{prop}
\label{prop:SHAfree}
Let $\sigrt$ be any weight of $\Kt$. Then $\ind_{\Kt}^{\Gt}(\sigrt)$ is a free $\SHA$-module. 
\end{prop}

\begin{proof}[Proof of Proposition \ref{prop:SHAfree}]
It follows from Corollary \ref{cor:supportradius} that if $f \in \ind_{\Kt}^{\Gt}(\sigrt)$ and if $\mathcal{T}^{\vec{r}}(f)$ is supported in $\coprod_{0 \leq i \leq n+1} \Kt \tilde{h}(\varpi)^{i}\Kt$, then $f$ is supported in $\coprod_{0 \leq i \leq n} \Kt \tilde{h}(\varpi)^{i}\Kt$. (In other words, if $\mathcal{T}^{\vec{r}}(f)$ is supported in the ball of radius $2n+2$ around the unit vertex of the tree $\mathcal{X}$, then $f$ is supported in the ball of radius $2n$ around the unit vertex.) Now a free $\mathcal{H}(\Gt, \Kt, \sigrt)$-basis for $\ind_{\Kt}^{\Gt}(\sigrt)$ can be (noncanonically) constructed using the method used in the proof of \cite{barthellivne:irredmodp} Theorem 19.

For completeness, here is the construction. Fix a basis $\mathcal{B}$ of $V_{\vec{r}}$ and put $A_0 = \{ [1, b]\}_{b \in \mathcal{B}}$: then $A_0$ is a basis for the space of functions in $\ind_{\Kt}^{\Gt}(\sigrt)$ which are supported in $\Kt$. Suppose that $n \geq 0$ and that we have, for each $i$ such that $0 \leq i \leq n$, a set $A_{i}$ of functions in $\ind_{\Kt}^{\Gt}(\sigrt)$ such that
\begin{enumerate}
\item each element of $A_i$ is supported in $\Kt\tilde{h}(\varpi)^{i}\Kt$, and
\item $B_n := \coprod_{\substack{i, k \geq 0\\ i + k \leq n}} \{ \left(\mathcal{T}^{\vec{r}}\right)^k(f): \, f \in A_i\}$ is a basis for the space of functions in $\ind_{\Kt}^{\Gt}(\sigrt)$ which are supported in $\coprod_{i = 0}^{n} \Kt \tilde{h}(\varpi)^{i}\Kt.$ 
\end{enumerate}
If it is true that the set $$B_{n+1}:= \coprod_{\substack{0 \leq i, k; \, i \leq  n\\ i + k \leq n+1}} \{\left(\mathcal{T}^{\vec{r}}\right)^k(f): \, f \in A_i\}$$ is linearly independent, then there exists a set $A_{n+1}$ of functions supported in $\Kt \tilde{h}(\varpi)^{n+1}\Kt$ such that $A_{n+1} \amalg B_{n+1}$ is a basis for the space of functions in $\ind_{\Kt}^{\Gt}(\sigrt)$ with support in $\coprod_{i = 0}^{n+1}\Kt \tilde{h}(\varpi)^i\Kt$. If $B_{n+1}$ is not linearly independent, then there exist $f,$ $f'$ with support in $\coprod_{i = 0}^{n}\Kt \tilde{h}(\varpi)^{i}\Kt$ such that $\mathcal{T}^{\vec{r}}(f) + f' = 0$. Then by Corollary \ref{cor:supportradius}, the support of $f$ lies in $\coprod_{i = 0}^{n-1}\Kt \tilde{h}(\varpi)^{i}\Kt.$ But now both $\mathcal{T}^{\vec{r}}(f)$ and $f'$ belong to the linearly independent set $B_n$, so $f = f' = 0$. Therefore $B_{n+1}$ is linearly independent. 
\end{proof}

\begin{lem}
\label{lem:cokernelnonzero}
Let $\sigrt$ be any weight of $\Kt$, let $\lambda \in E$, and let $\Theta_{\lambda}$ be the character of $\SHA$ defined by $\Theta_{\lambda}(\mathcal{T}^{\vec{r}}) = \lambda.$ Set $\pi(\vec{r}, \lambda):= \ind_{\Kt}^{\Gt}\sigrt \otimes_{\SHA} \Theta_{\lambda}.$ Then 
\begin{enumerate}
\item $\pi(\vec{r}, \lambda)$ is an infinite-dimensional genuine representation of $\Gt$. 
\item If $\rho$ is a nonzero quotient of $\pi(\vec{r}, \lambda)$, then $\rho$ contains $\sigrt$ and $\mathcal{T}^{\vec{r}}$ acts by $\lambda$ on $\Hom_{\Kt}(\sigrt, \, \rho\big \vert_{\Kt})$. 
\end{enumerate}
\end{lem}

\begin{proof}[Proof of Lemma \ref{lem:cokernelnonzero}]
\begin{enumerate}
\item The representation $\pi(\vec{r}, \lambda)$ is isomorphic to the quotient $(\ind_{\Kt}^{\Gt}\sigrt)/(\mathcal{T}^{\vec{r}}-\lambda).$ The image of $(\mathcal{T}^{\vec{r}} - \lambda)$ in $\ind_{\Kt}^{\Gt}\sigrt$ consists of genuine functions on $\Gt$, so the quotient $\pi(\vec{r}, \lambda)$ is genuine as well. The infinite-dimensionality of $\pi(\vec{r}, \lambda)$ follows from Proposition \ref{prop:SHAfree}; a vector space basis is given by the image in $\pi(\vec{r}, \lambda)$ of the free $\mathcal{H}(\Gt, \Kt, \tilde{\sigma}_{\vec{r}})$-basis constructed in the proof of that proposition. 

\item The composition of the quotients $\ind_{\Kt}^{\Gt}\sigrt \rightarrow \pi(\vec{r}, \lambda)$ and $\pi(\vec{r}, \lambda)\rightarrow \rho$ is nonzero, so by Frobenius reciprocity there exists a corresponding injection $\sigrt \hookrightarrow \rho$. The Hecke algebra $\SHA$ acts on $\Hom_{\Gt}(\ind_{\Kt}^{\Gt}\sigrt, \, \rho)$ by precomposition, and the image of the identity map of $\ind_{\Kt}^{\Gt}\sigrt$ in $\Hom_{\Gt}(\ind_{\Kt}^{\Gt}\sigrt, \, \rho)$ is an eigenvector for $\mathcal{T}^{\vec{r}}$ with eigenvalue $\lambda$. Passing to $\Hom_{\Kt}(\sigrt, \, \rho\big \vert_{\Kt})$ by Frobenius reciprocity, we get the second statement of the lemma. 
\end{enumerate}
\end{proof} 

Let $\rho$ be any smooth genuine representation of $\Gt$ and let $\sigrt$ be a $\Kt$-weight. Suppose that $\sigrt$ is contained in $\rho$, i.e., that $\Hom_{\Kt}(\sigrt, \rho \big \vert_{\Kt}) \neq 0$. Then $\SHA$ acts on $\Hom_{\Kt}(\sigrt, \rho \big \vert_{\Kt})$ from the right via Frobenius reciprocity. If the action of $\SHA$ admits an eigenvector with associated eigenvalue $\lambda$, then there exists a $\Gt$-linear map $\pi(\vec{r}, \lambda) \rightarrow \rho$; if $\rho$ is also irreducible, then it follows that $\rho$ is a quotient of $\pi(\vec{r}, \lambda)$. 

Let $\rho$ be a smooth genuine irreducible representation of $\Gt$. A pair $(\vec{r}, \lambda)$, with $\vec{r} \in \{0, \dots, p-1 \}^f$ and $\lambda \in E$, is a \textit{parameter for $\rho$ with respect to $\Kt$} if $\rho$ is a quotient of $\pi(\vec{r}, \lambda)$. 

Lemma \ref{lem:cokernelnonzero} has an obvious analogue for $\Kt'$: 

\begin{lem} \label{lem:kt'cokernelnonzero}
Let $\sigrt^{\tilde{\alpha}}$ be a weight of $\Kt'$, let $\lambda \in E$, and let $\Theta_{\lambda}^{\tilde{\alpha}}$ be the character of $\mathcal{H}_{E}(\Gt, \Kt', \sigrt^{\tilde{\alpha}})$ defined by $\Theta_{\lambda}^{\tilde{\alpha}}(\mathcal{T}^{\tilde{\alpha}}_1) = \lambda$. Set $\pi'(\vec{r}, \lambda) := \ind_{\Kt'}^{\Gt}\sigrt^{\tilde{\alpha}}\otimes_{\mathcal{H}(\Gt, \Kt', \sigrt^{\tilde{\alpha}})}\Theta_{\lambda}^{\tilde{\alpha}}.$ Then
\begin{enumerate}
\item $\pi'(\vec{r}, \lambda)$ is an infinite-dimensional genuine representation of $\Gt$.
\item If $\rho$ is a quotient of $\pi'(\vec{r}, \lambda)$, then $\rho$ contains $\sigrt^{\tilde{\alpha}}$ and $\mathcal{T}^{\tilde{\alpha}, \vec{r}}_1$ acts by $\lambda$ on $\Hom_{\Kt'}(\sigrt^{\tilde{\alpha}}, \rho \big \vert_{\Kt'}).$ 
\end{enumerate}
\end{lem}

Lemma \ref{lem:kt'cokernelnonzero} follows from:

\begin{lem} \label{lem:pipi'isom}
Let $\vec{r} \in \{0, \dots, p-1\}^f$ and $\lambda \in E$, and define $\pi'(\vec{r}, \lambda)$ as in the statement of Lemma \ref{lem:kt'cokernelnonzero}.  Then $\left(\pi(\vec{r}, \lambda)\right)^{\tilde{\alpha}} \cong \pi'(\vec{r}, \lambda).$
\end{lem}

\begin{proof}[Proof of Lemma \ref{lem:pipi'isom}]
Recall that in the proof of Lemma \ref{lem:kt'shaisom} was defined, for a weight $\sigrt$ of $\Kt$, an isomorphism $\Phi: \left(\ind_{\Kt}^{\Gt}\sigrt\right)^{\tilde{\alpha}} \rightarrow \ind_{\Kt'}^{\Gt}(\sigrt^{\tilde{\alpha}})$ such that $\Phi(h)(g) = h((\tilde{\alpha})^{-1}g\tilde{\alpha})$ for all $h \in \left(\ind_{\Kt}^{\Gt}\sigrt\right)^{\tilde{\alpha}}$ and $g \in \Gt$. Via the identity map $\SHA \rightarrow \End_{\Gt}\left((\ind_{\Kt'}^{\Gt}\sigrt)^{\tilde{\alpha}} \right)$, we may view $\mathcal{T}^{\vec{r}}$ as an endomorphism of $\left(\ind_{\Kt}^{\Gt}\sigrt \right)^{\tilde{\alpha}}$. Then the map $\End_{\Gt}\left((\ind_{\Kt}^{\Gt}\sigrt)^{\tilde{\alpha}} \right) \rightarrow \mathcal{H}(\Gt, \Kt', \sigrt^{\tilde{\alpha}})$ induced by $\Phi$ sends $\mathcal{T}^{\vec{r}}$ to $\mathcal{T}^{\tilde{\alpha}}_{1}$. Thus $\Phi$ induces an isomorphism $\left(\pi(\vec{r}, \lambda) \right)^{\tilde{\alpha}} \rightarrow \pi'(\vec{r}, \lambda).$ 
\end{proof}

If $\rho$ is a smooth genuine irreducible representation of $\Gt$, we will say that $(\vec{r}, \lambda)$ is a \textit{parameter for $\rho$ with respect to $\Kt'$} if $\rho$ is a quotient of $\pi'(\vec{r}, \lambda)$.

\begin{prop}\label{prop:possparamexists}
If $\rho$ is a smooth genuine irreducible admissible representation of $\Gt$, then $\rho$ has a parameter with respect to each of $\Kt$ and $\Kt'$. 
\end{prop}

\begin{proof}[Proof of Proposition \ref{prop:possparamexists}]
If $\rho$ is smooth and genuine, then $\rho$ contains a $\Kt$-weight $\sigrt$ (by Proposition \ref{prop:ktweights} (2)) and a $\Kt'$-weight $\sigst^{\tilde{\alpha}}$ (by Proposition \ref{lem:kt'weights} (2)), and for these weights we have $\Hom_{\Kt}(\sigrt, \rho\big \vert_{\Kt}) \neq 0$ and $\Hom_{\Kt'}(\sigst^{\tilde{\alpha}}, \rho\big \vert_{\Kt'}) \neq 0$. If $\rho$ is admissible, then the weight spaces are finite-dimensional, hence admit eigenvectors for the actions of $\SHA$ and $\mathcal{H}(\Gt, \Kt', \sigst^{\tilde{\alpha}})$ respectively. Let $\lambda$ denote an eigenvalue for the action of $\mathcal{T}^{\vec{r}}$ on $\Hom_{\Kt}(\sigrt, \rho\big \vert_{\Kt})$ and $\lambda'$ denote an eigenvalue for the action of $\mathcal{T}^{\tilde{\alpha}\vec{r}}_1$ on $\Hom_{\Kt'}(\sigst^{\tilde{\alpha}}, \rho\big \vert_{\Kt'})$. Then there exist nonzero $\Gt$-linear maps $\pi(\vec{r}, \lambda) \rightarrow \rho$ and $\pi'(\vec{s}, \lambda') \rightarrow \rho$, which are surjective if $\rho$ is irreducible. 
\end{proof}

A parameter $(\vec{r}, \lambda)$ (with respect to either $\Kt$ or $\Kt'$) will be called \textit{supersingular} if $\lambda = 0$, and \textit{nonsupersingular} otherwise. We will call $\pi(\vec{r}, \lambda)$ and $\pi'(\vec{r}, \lambda)$ \textit{cokernel modules} due to their construction as the cokernel of the Hecke algebra on its universal module.   

The following theorem is the main classification result. The proof is given in $\S$\ref{sec:classnproof}.  

\begin{thm} \label{thm:irredadmissclassn}
The smooth, genuine, irreducible, admissible $E$-representations of $\Gt$ fall into two disjoint classes:
\begin{enumerate}
\item those which have only nonsupersingular parameters,
\item those which have only supersingular parameters.
\end{enumerate}
The representations in the first class are exactly the genuine principal series representations of $\Gt$. All representations in the second class are supercuspidal.
\end{thm}

\begin{rem}
\label{rem:admisshyp}
We comment here on two questions raised by Theorem \ref{thm:irredadmissclassn}.
\begin{enumerate}
\item \textit{Necessity of the admissibility hypothesis:} The dichotomy of Theorem \ref{thm:irredadmissclassn} applies to the class of smooth, genuine, irreducible representations $\pi$ such that $\Hom_{\Kt}(\tilde{\sigma}_{\vec{r}}, \pi \big \vert_{\Kt})$ contains an eigenvector for $\mathcal{H}(\Gt, \Kt, \tilde{\sigma}_{\vec{r}})$ whenever $\tilde{\sigma}_{\vec{r}}$ is a weight of $\pi$. In the cases of $GL_2(F)$ (\cite{barthellivne:irredmodp} Proposition 32) and $SL_2(F)$ (\cite{abdellatif:thesis} Th\'eor\`eme 3.5.36, using $\text{char}(F) \neq 2$), in fact every smooth irreducible $E$-representation satisfies this eigenvector condition, so in those cases the dichotomy between nonsupersingular and supersingular parameters applies to all smooth irreducible representations. It is not clear whether the admissibility hypothesis can be similarly removed here.
\item \textit{Existence of supersingular and supercuspidal representations:} Using the argument of \cite{barthellivne:ordunram} Theorem 10 and our Proposition \ref{prop:SHAfree}, we can show that there exist smooth irreducible genuine representations of $\Gt$ which have supersingular parameters; if such a representation is also admissible, then it is supercuspidal, but we cannot confirm that this occurs.
\end{enumerate}
\end{rem}

\section{Nonsupercuspidal representations of $\Gt$}
\label{sec:psreps}
In this section we study the genuine representations of $\Gt$ which arise by parabolic induction; in particular, we consider smooth inductions to $\Gt$ of smooth genuine characters of $\Bbt$. We show that every such representation is irreducible (Theorem \ref{thm:pswtisom}) and that there are no intertwiners between $\Gt$-representations induced from distinct characters of $\Bbt$. We also give dictionaries (Theorem \ref{thm:dictionary}, Corollary \ref{cor:dictionary2}, and Corollary \ref{cor:dictionary3}) between parametrizations of the nonsupercuspidal representations: by parameters with respect to $\Kt$, by parameters with respect to $\Kt'$, and by characters of $F^\times$. 

\subsection{Hecke action on weight spaces of nonsupercuspidal representations}
\label{subsec:heckeactionps}

By Lemma \ref{lem:commutators} (2), the abelianization of $\Bbt$ is the quotient $\Bbt/\Ubar^*= \Tt$. Therefore any genuine character of $\Bbt$ arises by inflation of a genuine character of $\Tt$. 

\begin{prop} \label{prop:SHAaction}
Let $\pi$ be a smooth genuine character of $\Tt$, let $\sigrt$ be a weight of $\Kt$, let $\tau^{\vec{r}}_{-1}$ denote the generator of $\SHT$ defined in $\S$\ref{subsec:SHTdef}, and let $\mathcal{T}^{\vec{r}}$ denote the generator of $\SHA$ defined in (\ref{eqn:tprimedef}).
\begin{enumerate}
\item For each $f' \in \text{Hom}_{\Kt}(\sigrt, \text{Ind}_{\Bbt}^{\Gt}(\pi)\big \vert_{\Kt}),$
$$f'\cdot \mathcal{T}_1^{\vec{r}} = \pi(\tilde{h}(\varpi)) \cdot f'.$$
\item For each $f' \in \text{Hom}_{\Kt'}(\sigrt^{\tilde{\alpha}}, \text{Ind}_{\Bbt}^{\Gt}(\pi)\big \vert_{\Kt'}),$
$$f' \cdot \mathcal{T}_1^{\tilde{\alpha}, \vec{r}} = \pi(\tilde{h}(\varpi))\cdot f'.$$
\end{enumerate}
\end{prop}

\begin{proof}[Proof of Proposition \ref{prop:SHAaction}]
 To ease notation, put $\mathcal{T}_1:= \mathcal{T}^{\vec{r}}$, $\mathcal{T}^{\tilde{\alpha}}_1 := \mathcal{T}^{\tilde{\alpha}, \vec{r}}_1$, $\tau_{-1}:= \tau_{-1}^{\vec{r}}$, and $\psi_{-1}:= \psi_{-1}^{\vec{r}}$ for the duration of the proof.
\begin{enumerate}
\item Let $f'$ be any element of $\Hom_{\Kt}(\sigrt,\text{Ind}_{\Bbt}^{\Gt}\pi\big \vert_{\Kt})$, and let $f$ denote the element of $\Hom_{\Gt}(\ind_{\Kt}^{\Gt}\sigrt, \, \text{Ind}_{\Bbt}^{\Gt}\pi)$ which corresponds to $f'$ by Frobenius reciprocity. Let $\mathcal{F}_0$ denote the unique element of 
$$ \Hom_{\Gt}\left(\ind_{\Kt}^{\Gt}\sigrt, \, \Ind_{\Bbt}^{\Gt}\left(\ind_{\TKt}^{\Tt}(\sigrt)_{\Ubark} \right)\right)$$ satisfying 
$$(\mathcal{F}_0)_{\Tt} = 1 \in \End_{\Tt}\left(\ind_{\TKt}^{\Tt}(\sigrt)_{\Ubark}\right).$$ 
Then $f_{\Tt} = f_{\Tt} \circ (\mathcal{F}_0)_{\Tt},$ so $f = \Ind_{\Bbt}^{\Gt}(f_{\Tt}) \circ \mathcal{F}_0.$

By Proposition \ref{prop:gensataketransformi} (1),
$$(\mathcal{F}_0 \circ \mathcal{T}_1)_{\Tt} = \mathcal{S}_{\vec{r}}(\mathcal{T}_1) = \mathcal{S}_{\vec{r}}(\mathcal{T}_1) \circ (\mathcal{F}_0)_{\Tt},$$
so
$$\mathcal{F}_0 \circ \mathcal{T}_1 = \Ind_{\Bbt}^{\Gt}(\mathcal{S}_{\vec{r}}(\mathcal{T}_1)) \circ \mathcal{F}_0.$$

Thus $f \circ \mathcal{T}_1 = \Ind_{\Bbt}^{\Gt}(f_{\Tt}) \circ \mathcal{F}_0 \circ \mathcal{T}_1 = \Ind_{\Bbt}^{\Gt}(f_{\Tt}) \circ \Ind_{\Bbt}^{\Gt}(\mathcal{S}_{\vec{r}}(\mathcal{T}_1)) \circ \mathcal{F}_0 = \Ind_{\Bbt}^{\Gt}(f_{\Tt} \circ \tau_{-1}) \circ \mathcal{F}_0,$
using Proposition \ref{prop:satakeproperties} for the last equality. 

By Frobenius reciprocity $f_{\Tt} \circ \tau_{-1}$ corresponds to $(f_{\Tt})' \cdot \tau_{-1} \in \Hom_{\Tt \cap \Kt}((\sigrt)_{\Ubark}, \, \pi \big \vert_{\Tt \cap \Kt})$. By (1) of this proposition, $(f_{\Tt})' \cdot \tau_{-1} = \pi(\tilde{h}(\varpi)) \cdot (f_{\Tt})',$ so $f_{\Tt} \circ \tau_{-1} = \pi(\tilde{h}(\varpi)) \cdot f_{\Tt}$. Then $\Ind_{\Bbt}^{\Gt}(f_{\Tt} \circ \tau_{-1}) = \pi(\tilde{h}(\varpi)) \cdot \Ind_{\Bbt}^{\Gt}(f_{\Tt})$, so 
$$f \circ \mathcal{T}_1  = \pi(\tilde{h}(\varpi)) \cdot \Ind_{\Bbt}^{\Gt}(f_{\Tt}) \circ \mathcal{F}_0 = \pi(\tilde{h}(\varpi))\cdot f.$$
Passing back to $\Hom_{\Kt}(\sigrt, \Ind_{\Bbt}^{\Gt}\pi \big \vert_{\Kt})$, we have $f'\cdot \mathcal{T} = \pi(\tilde{h}(\varpi))\cdot f'.$

\item We will show that $\mathcal{T}^{\tilde{\alpha}}_1$ acts on $\Hom_{\Gt}\left(\ind_{\Kt'}^{\Gt}\sigrt^{\tilde{\alpha}}, \Ind_{\Bbt}^{\Gt}(\pi) \right)$ by the scalar $\pi(\tilde{h}(\varpi))$; then (3) follows by Frobenius reciprocity. We again use the map $\Phi$ defined in the proof of Lemma \ref{lem:kt'shaisom}, precomposing with the identity map of vector spaces to obtain an isomorphism $\SHA \rightarrow \mathcal{H}(\Gt, \Kt', \sigrt^{\tilde{\alpha}})$ which we again denote by $\Phi$. This isomorphism satisfies $\mathcal{T}^{\tilde{\alpha}}_1(\Phi(h)) = \Phi(\mathcal{T}_1(h))$ for all $h \in \ind_{\Kt}^{\Gt}\sigrt$. 

Given $f \in \Hom_{\Gt}\left(\ind_{\Kt'}^{\Gt}\sigrt^{\tilde{\alpha}}, \Ind_{\Bbt}^{\Gt}(\pi) \right)$, let $\mathfrak{f} = \Phi^{-1}(f) \in  \Hom_{\Gt}\left(\ind_{\Kt}^{\Gt}\sigrt, \Ind_{\Bbt}^{\Gt}(\pi) \right)$; i.e., 
$$f(\Phi(h)) = \mathfrak{f}(h)$$
for all $h \in \ind_{\Kt}^{\Gt}\sigrt$. We have to show that 
\begin{equation} \label{eqn:kt'eval}
\left(f \circ \mathcal{T}^{\tilde{\alpha}}_1\right)(\Phi(h)) = \pi(\tilde{h}(\varpi))\cdot f(\Phi(h))
\end{equation}
for all $h \in \ind_{\Kt}^{\Gt}\sigrt$. But the left-hand side of (\ref{eqn:kt'eval}) is equal to
\begin{align*}f \circ \Phi(\mathcal{T}_1)(\Phi(h)) & = f \left(\Phi(\mathcal{T}(h))\right)\\
& = (\mathfrak{f} \circ \mathcal{T}_1)(h),
\end{align*}
while the right-hand side of (\ref{eqn:kt'eval}) is equal to $\pi(\tilde{h}(\varpi))\cdot \mathfrak{f}(\Phi(h))$. By (2) we have $(\mathfrak{f} \circ \mathcal{T}_1)(h) = \pi(\tilde{h}(\varpi))\cdot \mathfrak{f}(\Phi(h))$, so (\ref{eqn:kt'eval}) holds. 
\end{enumerate}
\end{proof}

\subsection{Genuine characters of $\Tt$} \label{subsec:gencharsTt} 
We begin by recalling the parametrization of genuine complex characters of $\Tt$. 
\subsubsection{Genuine $\CC$-valued characters of $\Tt$} The genuine $\CC$-valued characters of $\Tt$ have been classified in terms of a certain map $\gamma: F^\times \times \hat{F} \rightarrow \mu_4(\CC)$, the \textit{Weil index} defined in \cite{weil:unitaires}. Here $\hat{F}$ denotes the group of $\CC$-valued continuous additive characters of $F$. Fixing a nontrivial character $\psi \in \hat{F}$, one obtains a function $\gamma_F(-, \psi): F^\times \rightarrow \mu_4(\CC)$. Define
\begin{align*}
\chi_{\psi}: \Tt &\rightarrow \mu_4(\CC)\\
(h(x), \zeta)& \mapsto  \zeta \cdot \gamma_F(x, \psi)^{-1}.
\end{align*}
Then $\chi_{\psi}$ is a smooth genuine $\CC$-character of $\Tt$. The following fact is mentioned in \cite{gansavin:epsilon}:
\begin{lem} \label{lem:ttgenchars}
Let $\psi$ be a nontrivial $\CC$-valued additive character of $F$. There is a bijection, depending on $\psi$, of 
$$\{ \text{smooth $E$-valued characters of $F^\times$}\} \rightarrow \{ \text{smooth genuine $\CC$-valued characters of $\Tt$} \}, $$
given by $\mu \mapsto \mu \cdot \chi_{\psi}.$
\end{lem}

\subsubsection{Genuine $E$-valued characters of $\Tt$} 
\label{subsubsec:genechartt}
We would like to have a similar parametrization of the genuine $E$-valued characters of $\Tt$. It is well-known that there is no nontrivial $E$-valued continuous additive character of $F$. However, given a nontrivial $\CC$-valued additive character $\psi$ of $F$, we may postcompose the resulting $\mu_4(\CC)$-valued function $\gamma_F(-, \psi)^{-1}$ with an isomorphism of $\mu_4(\CC)$ with $\mu_4(E)$ to obtain a $\mu_4(E)$-valued function on $F^\times$. We denote this function again by $\gamma_F(-, \psi)^{-1}$, and define an $E$-valued character of $\Tt$ by 
\begin{equation} \label{eqn:genchar} 
\chi_{\psi}\left(h(a), \zeta \right) := \zeta \gamma_F(a, \psi)^{-1}.
\end{equation}

Thus we directly transport Lemma \ref{lem:ttgenchars} to obtain a parametrization of genuine $E$-valued characters of $\Tt$:
\begin{lem}
\label{lem:gencharparam}
Let $\psi$ be a $\CC$-valued additive character of $F$, and let $\chi_{\psi}$ be the genuine $E$-valued character of $\Tt$ defined in (\ref{eqn:genchar}). Then there is a bijection
$$\{\text{smooth $E$-valued characters of $F^\times$} \} \rightarrow \{ \text{smooth genuine $E$-valued characters of $\Tt$}\}$$
given by $\mu \mapsto \mu \cdot \chi_{\psi}$, where $\mu$ is viewed as a (non-genuine) character of $\Tt$ by inflation. 
\end{lem}

\subsubsection{The Weil index} 
Here we recall some properties of the Weil index and develop explicit expressions for certain values, which will be needed in later sections. 

Let $L$ be either a local field (e.g., $F$) or a finite field (e.g., $\mathfrak{k}$), and let $\psi$ be a nontrivial additive $\CC$-valued character of $L$. Let $\gamma_{L, \psi}$ denote the Weil index defined in \cite{weil:unitaires} Theorem 2. This is a $\mu_4(\CC)$-valued function, depending on $\psi$, defined on the Witt group of quadratic forms over $L$. For $a \in L^\times$, let $\gamma_{L, \psi}(a)$ denote the value of $\gamma_{L,\psi}$ on the class of the quadratic form $\mathfrak{q}(x) = ax^2$. The following theorem, first proven by Weil \cite{weil:unitaires}, is stated here as in Ranga Rao \cite{rao:weilrepn}.

\begin{thm}[\cite{rao:weilrepn} Theorem A.4] \label{thm:raoweilindex} 
Let $\gamma_L(a, \psi) = \gamma_{L, \psi}(a)/\gamma_{L, \psi}(1).$ Then
\begin{enumerate}
\item $\gamma_L(ac^2, \psi) = \gamma_L(a, \psi)$ for any $a$, $c \in L^\times$, 
\item $\gamma_L(a, \psi)\gamma_L(b, \psi) = \gamma_L(ab, \psi)(a, b)_F.$  
\end{enumerate}
\end{thm}

Another theorem stated by Rao \cite{rao:weilrepn} gives a more detailed description of $\gamma_F(a, \psi)$ when $a \in \OF$:
\begin{thm}[\cite{rao:weilrepn} Theorem A.11 with correction]
\label{thm:raoweilindexof}
Suppose $F$ is a nonarchimedean local field with finite residue field $\mathfrak{k}$ of odd characteristic, and let $m$ denote the conductor of $\psi$. Given any integer $n$, let $\delta(n) \in \{0, 1\}$ denote its parity. For $y \in \OF$, define 
$$\bar{\psi}(x + \varpi y) = \psi(\varpi^{-m-1}x).$$
Then 
\begin{enumerate}
\item $\bar{\psi}$ is a nontrivial character of $\mathfrak{k}$ and 
$$\gamma_{F, \psi}(1) = \left( \gamma_{\mathfrak{k}, \bar{\psi}}(1) \right)^{\delta(m)},$$

\item If $a \in \OF$ and $a = u \cdot \varpi^{v_F(a)}$ with $u \in \OF^\times$,  
$$\gamma_{F}(a, \psi) = \begin{cases} 
1 & \text{ if } \delta(m) = \delta(v_F(a)) = 0,\\
(u, \varpi)_F\cdot \gamma_{\mathfrak{k}}(\bar{\psi}) & \text{ if } \delta(m) = 0 \text{ and } \delta(v_F(a)) = 1,\\
(u, \varpi)_F & \text{ if } \delta(m) = 1\text{ and } \delta(v_F(a)) = 0,\\
\gamma_{\mathfrak{k}}(\bar{\psi})^{-1} & \text{ if } \delta(m) = \delta(v_F(a)) = 1.
\end{cases}$$
In particular, if $a \in \OF^\times,$ 
$$\gamma_F(a, \psi) = (a, \varpi)_F^{m}.$$
\end{enumerate}
\end{thm}

\begin{proof}[Proof of Theorem \ref{thm:raoweilindex}]
The proof of (1) is given in \cite{rao:weilrepn}. We give a proof of (2) since our formula differs from the one given in \cite{rao:weilrepn} Thm. A.11 (and appears to disagree with it: for example, we claim that $\gamma_F(-, \psi)$ is nontrivial on units when the conductor of $\psi$ is odd; our formula is consistent with Gan-Savin's remark in $\S$2.6 of \cite{gansavin:epsilon}, while the apparent statement in Rao is not). If $a = u \cdot \varpi^{v_F(a)} \in \OF^*$ and $\psi$ is a nontrivial additive $\CC$-character of $F$ with conductor $m$, then the conductor of $\psi_a = (x \mapsto \psi(ax))$ is $m + v_F(a)$. By part (1) of the theorem, we have
\begin{equation} \label{eqn:gammakformula}
\gamma_{F}(a, \psi) = \frac{\gamma_{F}(\psi_a)}{\gamma_F(\psi)} = \frac{(\gamma_{\mathfrak{k}}(\bar{\psi_a}))^{\delta(m + v_F(a))}}{(\gamma_\mathfrak{k}(\bar{\psi}))^{\delta(m)}}.
\end{equation}
The character $\bar{\psi_a}$ sends $x + \varpi\OF$ to $\psi(\varpi^{-m - v_F(a)-1}ax) = \psi(\varpi^{-m -1}ux),$ so 
$$\bar{\psi_a} = \bar{\psi_u} = \bar{\psi}_{\bar{u}}.$$
We have
$$\gamma_\mathfrak{k}(\bar{u}, \bar{\psi}) = \frac{\gamma_{\mathfrak{k}}(\bar{\psi}_{\bar{u}})}{\gamma_\mathfrak{k}(\bar{\psi})} = \left( \frac{\bar{u}}{\mathfrak{k}}\right),$$
where the first equality is the definition of $\gamma_\mathfrak{k}(\bar{u}, \bar{\psi})$ and the second is \cite{rao:weilrepn} Thm. A.9 (i). 
Under our assumptions on $F$, the Legendre symbol $\left(\frac{\bar{u}}{\mathfrak{k}} \right)$ is equal to the Hilbert symbol $(u, \varpi)_F$, so
$$\gamma_\mathfrak{k}(\bar{\psi_a})= \gamma_\mathfrak{k}(\bar{u}, \bar{\psi}) \cdot \gamma_\mathfrak{k}(\bar{\psi}) = (u, \varpi)_F \cdot \gamma_\mathfrak{k}(\bar{\psi}).$$
Substituting $(u, \varpi)_F \cdot \gamma_\mathfrak{k}(\bar{\psi})$ in (\ref{eqn:gammakformula}) gives
\begin{equation} \label{eqn:gammaFformula}
\gamma_{F}(a, \psi) = \frac{\left( (u, \varpi)_F \cdot \gamma_\mathfrak{k}(\bar{\psi}) \right)^{\delta(m + v_F(a))}}{\left( \gamma_\mathfrak{k}(\bar{\psi})\right)^{\delta(m)}},
\end{equation}
which depends on both the parity of $m$ and also of $v_F(a)$. Considering (\ref{eqn:gammaFformula}) in each of the four cases, we get (2).
\end{proof}

\subsection{Weight spaces of principal series representations} 
\label{subsec:psisotypes}
By Lemma \ref{lem:commutators} (2), any genuine character of $\Bbt$ is the inflation of a genuine character of $\Tt$. Fix a nontrivial additive $\CC$-valued character $\psi$ and for any smooth character $\mu$ of $F^\times$, let $\mu \cdot \chi_\psi$ also denote the inflation to $\Bbt$ of the $E$-valued character of $\Tt$ considered in $\S$\ref{subsubsec:genechartt}. By Lemma \ref{lem:gencharparam}, every genuine $E$-valued character of $\Tt$ appears as $\mu \cdot \chi_{\psi}$ for some character $\mu$ of $F^\times$. Hence any genuine principal series representation of $\Gt$ is isomorphic to $\text{Ind}_{\Bbt}^{\Gt}(\mu\cdot \chi_{\psi})$ for some smooth character $\mu$ of $F^\times.$ 

\begin{prop} \label{prop:psweights}
Fix a nontrivial additive character $\psi: F \rightarrow \CC$ and denote the conductor of $\psi$ by $m$. Let $\mu$ be a smooth $E$-valued character of $F^\times$ and view $\mu \cdot \chi_{\psi}$ (defined in (\ref{eqn:genchar})) as a genuine character of $\Bbt$ by inflation from $\Tt$.
\begin{enumerate}
\item Suppose $\mu \big \vert_{\OF^\times} \neq (-, \varpi)_F^m$. Then there is a unique weight $\sigrt$ of $\Kt$ such that if $\sigst$ is any weight of $\Kt$,
$$\text{dim}_{E}\,\text{Hom}_{\Kt}(\sigst, \text{Ind}_{\Bbt}^{\Gt}(\mu \cdot \chi_{\psi}) \big \vert_{\Kt}) = \begin{cases}
1 & \text{ if } \vec{s} = \vec{r},\\
0 & \text{ otherwise}.
\end{cases}$$ The weight $\sigrt$ satisfies $1 < \text{dim}_{E}\,\sigrt < q$, i.e., $\vec{r} \notin \{ \vec{0}, \overrightarrow{p-1}\}$. 

\item Suppose $\mu \big \vert_{\OF^\times} = (-, \varpi)_F^m$. Then 
$$\text{dim}_{E}\,\text{Hom}_{\Kt}(\sigst, \text{Ind}_{\Bbt}^{\Gt}(\mu \cdot \chi_{\psi}) \big \vert_{\Kt}) = \begin{cases}
1 & \text{ if } \vec{s} = \vec{0} \text{ or } \vec{s} = \overrightarrow{p-1},\\
0 & \text{otherwise}.
\end{cases}$$

\item Suppose $\mu \big \vert_{\OF^\times} \neq (-, \varpi)_F^{m+1}$. Then there is a unique weight $\sigrt^{\tilde{\alpha}}$ of $\Kt'$ such that if $\sigst^{\tilde{\alpha}}$ is any weight of $\Kt'$, 
$$\text{dim}_{E}\,\text{Hom}_{\Kt'}(\sigst^{\tilde{\alpha}}, \text{Ind}_{\Bbt}^{\Gt}(\mu \cdot \chi_{\psi}) \big \vert_{\Kt'}) = \begin{cases}
1 & \text{ if } \vec{s} = \vec{r},\\
0 & \text{ otherwise}.
\end{cases}$$ The weight $\sigrt^{\tilde{\alpha}}$ satisfies $1 < \text{dim}_{E}\,\sigrt^{\tilde{\alpha}} < q$, i.e., $\vec{r} \notin \{ \vec{0}, \overrightarrow{p-1}\}$. 

\item Suppose $\mu \big \vert_{\OF^\times} = (-, \varpi)_F^{m+1}$. Then 
$$\text{dim}_{E}\,\text{Hom}_{\Kt'}(\sigst^{\tilde{\alpha}}, \text{Ind}_{\Bbt}^{\Gt}(\mu \cdot \chi_{\psi}) \big \vert_{\Kt'}) = \begin{cases}
1 & \text{ if } \vec{s} = \vec{0} \text{ or } \vec{s} = \overrightarrow{p-1},\\
0 & \text{otherwise}.
\end{cases}$$
\end{enumerate}
\end{prop}

\begin{proof}[Proof of Proposition \ref{prop:psweights}]
Let $\mu$ be any smooth character of $F^\times$ and let $\sigrt$ be any weight of $\Kt$. Each of $\mu\cdot \chi_{\psi}$ and $(\sigrt)_{\Ubark}$ is a one-dimensional representation of $\Tt \cap \Kt$, so the space $\Hom_{\TKt}( (\sigrt)_{\Ubark}, \, \mu \cdot \chi_{\psi} \big \vert_{\TKt})$ is one-dimensional if $(\sigrt)_{\Ubark} \cong \mu\cdot \chi_{\psi}\big \vert_{\TKt}$ as $\TKt$-representations, and zero otherwise.

By Frobenius reciprocity, $$\text{dim}_E \, \Hom_{\TKt}\left( (\sigrt)_{\Ubark}, \, \mu \cdot \chi_{\psi} \big \vert_{\TKt}\right) = \text{dim}_E\, \Hom_{\Tt}\left(\ind_{\TKt}^{\Tt}\left((\sigrt)_{\Ubark} \right),\, \mu \cdot \chi_{\psi} \right).$$ By Lemma \ref{lem:natlisom} (1) followed by Frobenius reciprocity again, 
\begin{align*}
\text{dim}_E\, \Hom_{\Tt}\left(\ind_{\TKt}^{\Tt}\left((\sigrt)_{\Ubark} \right),\, \mu \cdot \chi_{\psi} \right) &= \text{dim}_E\, \Hom_{\Gt}\left(\ind_{\Kt}^{\Gt} \sigrt, \, \Ind_{\Bbt}^{\Gt}(\mu \cdot \chi_{\psi})\right) \\ & = \text{dim}_E\, \Hom_{\Kt}\left(\sigrt, \, \Ind_{\Bbt}^{\Gt}(\mu \cdot \chi_{\psi})\big \vert_{\Kt}\right).
\end{align*}
Thus we have 
$$\text{dim}_E\, \Hom_{\Kt}(\sigrt, \, \Ind_{\Bbt}^{\Gt}(\mu \cdot \chi_{\psi})\big \vert_{\Kt}) = \begin{cases} 1 & \text{ if } (\sigrt)_{\Ubark} \cong \mu \cdot \chi_{\psi}\big \vert_{\Tt \cap \Kt} \text{ as } \TKt\text{-representations},\\
0 & \text{ otherwise. }\end{cases}$$
  
Recall that every smooth $E$-character of $\OF^\times$ is isomorphic to $\delta_{\vec{s}}$ for exactly one $\vec{s} \in \{ 0, \dots, p-1\}^f\setminus\{\overrightarrow{p-1}\}$ (see $\S$\ref{subsubsec:coefffield} for the definition of $\delta_{\vec{s}}$), and that $\delta_{\vec{0}} = \delta_{\overrightarrow{p-1}}$. Furthermore, each smooth genuine representation of $\Tt \cap \Kt$ is isomorphic to $\tilde{\delta}_{\vec{s}}:= \delta_{\vec{s}} \otimes \epsilon$, where $\delta_{\vec{s}}$ is viewed as a character of $\Tt \cap \Ks$ by inflation. In particular, the $\Tt \cap \Kt$-representation on $(\sigrt)_{\Ubark}$ is $\tilde{\delta}_{\vec{r}}$.

By Theorem \ref{thm:raoweilindexof}, $\gamma_F(a, \psi)^{-1} = (a, \varpi)^m_F$ for all $a \in \OF^\times$. Under the condition of case (1) of the proposition, the restriction of $\mu\cdot \chi_{\psi}$ to $\Tt \cap \Ks$ factors through a nontrivial smooth character of $\OF^\times$, i.e., through $\delta_{\vec{s}}$ for some $\vec{s} \notin \{ \vec{0}, \overrightarrow{p-1}\}$, so we have a $\TKt$-isomorphism $(\sigrt)_{\Ubark} \cong \mu \cdot \chi_{\psi}$ if and only if $\vec{s} = \vec{r}$. Since $\vec{r}$ is neither $\vec{0}$ nor $\overrightarrow{p-1}$, the dimension of $\sigrt$ is strictly between 1 and $q$. 

In case (2) of the proposition, the restriction of $\mu \cdot \chi_{\psi}$ to $\Tt \cap \Ks$ factors through the trivial character of $\OF^\times$. Thus we have a $\TKt$-isomorphism $(\sigrt)_{\Ubark} \cong \mu \cdot \chi_{\psi}$ exactly when $\vec{r} = \vec{0}$ or $\vec{r} = \overrightarrow{p-1}$, i.e., exactly when $\text{dim}_E\, \sigrt = 1$ or $q$. 

The proofs of (3) and (4) are similar. The difference between the statements (1) and (3) (resp,. between (2) and (4)) is a consequence of the fact (Lemma \ref{lem:kt'coinvs}) that for a weight $\sigrt^{\tilde{\alpha}}$ of $\Kt'$, the $\Tt \cap \Kt$-representation on $(\sigrt^{\tilde{\alpha}})_{\Ubark}$ is $\tilde{\delta}_{\vec{r}} \cdot (-, \varpi)_F$ rather than $\tilde{\delta}_{\vec{r}}$.
\end{proof} 

We maintain the notations $\mathcal{T}_n:= \mathcal{T}_n^{\vec{r}}$, $\tau_{-n}:= \tau_{-n}^{\vec{r}}$, $\mathcal{S} : = \mathcal{S}_{\vec{r}}$ for the remainder of this section, and continue to refer to the map $\mathcal{F}_0$ defined in the proof of Proposition  \ref{prop:SHAaction}. Define another map

\begin{equation} \label{eqn:wtisomcand}
\mathcal{F}: \text{ind}_{\Kt}^{\Gt}(\sigrt) \otimes_{\SHA, \mathcal{S}_{\vec{r}}}\SHT \rightarrow \text{Ind}_{\Bbt}^{\Gt}\left(\text{ind}_{\Tt \cap \Kt}^{\Tt}\left((\sigrt)_{\Ubark}\right)\big \vert_{\Tt \cap \Kt}\right).
\end{equation}

by setting
$\mathcal{F}(f\cdot \tau_{-n}) : = \left( \Ind_{\Bbt}^{\Gt}(\tau_{-n})\circ \mathcal{F}_0 \right)(f)$ 
for $f \in \ind_{\Kt}^{\Gt}(\sigrt)$, $n \in \ZZ$. Then $\mathcal{F}$ is clearly $\Gt$-linear, and has the following $\SHA$-linearity:
\begin{align*}
\left(\mathcal{F}\circ\mathcal{T}_1\right)(f \cdot \tau_{-n}) &= 
\left(\Ind_{\Bbt}^{\Gt}(\tau_{-n})\circ \mathcal{F}_0 \circ \mathcal{T}_1\right)(f)\\
& = \left(\Ind_{\Bbt}^{\Gt}(\tau_{-n} \circ \mathcal{S}(\mathcal{T}_1) \circ \mathcal{F}_0 \right)(f)\\
& = \left(\Ind_{\Bbt}^{\Gt}(\tau_{-(n+1)}  \circ \mathcal{F}_0\right)(f).
\end{align*}

\begin{prop} \label{prop:wtisomuniv}
The map $\mathcal{F}:  \text{ind}_{\Kt}^{\Gt}(\sigrt) \otimes_{\SHA} \SHT \rightarrow \text{Ind}_{\Bbt}^{\Gt}\left(\text{ind}_{\Tt \cap \Kt}^{\Tt}\left((\sigrt)_{\Ubark}\right)\right)$ defined in (\ref{eqn:wtisomcand}) is an injective map of $(\Gt, \, \SHT)$-modules. If $\vec{r} \neq \vec{0}$, then $\mathcal{F}$ is an isomorphism. 
\end{prop}

\begin{proof}[Proof of Proposition \ref{prop:wtisomuniv}]
The argument of \cite{herzig:glnnotes} Theorem 32 goes through with only the obvious adaptations. We repeat the argument here for completeness. For injectivity of $\mathcal{F}$, it suffices to show that $$\mathcal{F}_0: \ind_{\Kt}^{\Gt}\sigrt \rightarrow \Ind_{\Bbt}^{\Gt}\left(\ind_{\Tt \cap \Kt}^{\Tt}\left( (\sigrt)_{\Ubark} \right) \right)$$ is injective. Suppose that $\mathcal{F}_0$ has a nonzero kernel. Then $\text{ker}(\mathcal{F}_0)$ is a genuine subrepresentation of $\ind_{\Kt}^{\Gt}\sigrt$, so $\text{ker}(\mathcal{F}_0)$ contains some weight $\sigst$ of $\Kt$. The $\Kt$-linear injection $\sigst \hookrightarrow \text{ker}(\mathcal{F}_0)\big \vert_{\Kt}$ corresponds by Frobenius reciprocity to a nonzero $\Gt$-linear homomorphism $\ind_{\Kt}^{\Gt}\sigst \rightarrow \text{ker}(\mathcal{F}_0)$, which we may compose with the $\Gt$-linear inclusion $\text{ker}(\mathcal{F}_0) \hookrightarrow \ind_{\Kt}^{\Gt}\sigrt$ to get a nonzero element $\Phi \in \mathcal{H}_{E}(\Gt, \Kt, \sigst, \sigrt)$ such that $\mathcal{F}_0 \circ \Phi = 0$. Then
\begin{align*}
0  = (\mathcal{F}_0 \circ \Phi)_{\Tt} = 1_{\ind_{\Tt \cap \Kt}^{\Tt}((\sigrt)_{\Ubark})} \circ \mathcal{S}_{\vec{s}, \vec{r}}(\Phi),
\end{align*}
so $\mathcal{S}_{\vec{s}, \vec{r}}(\Phi) = 0.$ But $\mathcal{S}_{\vec{s}, \vec{r}}$ is injective, so $\Phi = 0$, contradicting the assumption that $\text{ker}(\mathcal{F}_0) \neq 0$. 

Next we show surjectivity under the assumption that $\vec{r} \neq 0.$ Let $0 \neq v \in \Vsigt^{\Uk}$ and write 
$$f_0 = \mathcal{F}([1, v]).$$
Then, unwinding the chain of isomorphisms in Lemma \ref{lem:natlisom}, we have the following formula for $f_0$:  
$$f_0(g) = [t,\, p_{\Ubark}\circ \sigrt(k)v ],$$
where $g = tuk$ for $t \in \Tt$, $u \in \Ubar^*$, and $k \in \Kt$, and where $p_{\Ubark}$ denotes the projection $\sigrt \rightarrow (\sigrt)_{\Ubark}$.

Let $e_{\vec{0}}$ (resp., $e_{\vec{r}}$) denote a basis for the one-dimensional highest-weight (resp., lowest-weight) space of $\sigma_{\vec{r}}$. Let $k \in \Kt$ and consider $Pr(k) \in K$. By the Bruhat decomposition of $G(\mathfrak{k})$, the image of $Pr(k)$ in $G(\mathfrak{k})$ lies in either $w(-1)B(\mathfrak{k})$ or in $\Bbt(\mathfrak{k})U(\mathfrak{k}).$ In the former case, then $\sigrt(k)e_{\vec{0}} =  \zeta \sigma_{\vec{r}}\left( w(-1) \right) \circ \sigma_{\vec{r}}(b) e_{\vec{0}}$ 
for some $\zeta \in \mu_2$ and $b \in B \cap K$. Since $\sigma_{\vec{r}}(b) e_{\vec{0}} \subset E\cdot e_{\vec{0}}$, we then have $\sigrt(k)e_{\vec{0}} \subset E\cdot e_{\vec{r}},$ hence (since $\vec{r} \neq \vec{0}$), $p_{\Ubark}\circ\sigrt(k)v = 0$.

Suppose $k \in \Kt$ is in the latter case, i.e., that the image of $Pr(k)$ in $G(\mathfrak{k})$ lies in $\bar{B}(\mathfrak{k})U(\mathfrak{k}).$ Then $Pr(k) \in (\overline{B} \cap K) \cdot (U \cap K)$ and $\sigrt(k)e_{\vec{0}} \subset E \cdot e_{\vec{0}}$, so $p_{\Ubark} \circ \sigrt(k)v \neq 0.$ Hence the support of $f_0$ is equal to $\Bbt\cdot (U \cap K)^*.$ 
 In particular, if $a \in \OF^\times$, then
\begin{equation} \label{eqn:f0form}
f_0 \left(
  \left( \begin{array}{cc}
1      & \varpi^{r}a \\
0      & 1 
    \end{array} \right), 1 \right)  = \begin{cases} [1,\, p_{\Ubark}(v)] & \text{ if } r \geq 0 \\
\left[ \left(
    \left( \begin{array}{cc}
\varpi^r a        & 0 \\
0        & \varpi^{-r}a^{-1} 
      \end{array} \right), (-1, \varpi)_F^r \right),\, p_{\Ubark} \circ \sigma_{\vec{r}}
    \left(w(1) \right)v \, \right] = 0 & \text{ if } r < 0.
\end{cases}
\end{equation}

We proceed in two steps: first, we will show that $f_0$ generates the subspace of $\text{Ind}_{\Bbt}^{\Gt}\left(\text{ind}_{\Tt \cap \Kt}^{\Tt}\left((\sigrt)_{\Ubark} \right) \right)$ consisting of functions with support in $\Bbt \cdot (U \cap K)^*$. Then we will show that functions in this subspace generate all of $\text{Ind}_{\Bbt}^{\Gt}\left(\text{ind}_{\Tt \cap \Kt}^{\Tt}\left((\sigrt)_{\Ubark} \right) \right)$ under the $\Gt$- and $\SHT$-actions. 

Define a map 
$$\Big \{ f \in \text{Ind}_{\Bbt}^{\Gt}\left(\text{ind}_{\Tt \cap \Kt}^{\Tt}\left((\sigrt)_{\Ubark} \right) \right) : \text{ supp}(f)\subset \Bbt \cdot (U \cap K)^* \Big \} \longrightarrow \mathcal{C}_c^{\infty}\left(F, \,\text{ind}_{\Tt \cap \Kt}^{\Tt}\left((\sigrt)_{\Ubark} \right)\right)$$
by 
$$ f \mapsto \left(\underline{f}: c \mapsto f \left(u(c), 1\right) \right),$$
and note that extension by zero is an inverse. Hence $f \mapsto \underline{f}$ defines an isomorphism of $E$-vector spaces, and is also a map of $\SHA$-modules. The subspace of $\text{Ind}_{\Bbt}^{\Gt}\left(\text{ind}_{\Tt \cap \Kt}^{\Tt}\left((\sigrt)_{\Ubark} \right) \right)$ with support in $\Bbt \cdot (U \cap K)^*$ is $\Bt$-stable, and the induced $\Bt$-module structure on $\mathcal{C}_c^{\infty}(F,\, \text{ind}_{\Tt \cap \Kt}^{\Tt}\left((\sigrt)_{\Ubark} \right)$ is given by 
\begin{align*} 
\left( 
 u(c), 1\right)f'(d) & = \underline{f}(c + d),\\
\left(h(x), \zeta \right)f'(d) & = \underline{f}(x^{-2}d).
\end{align*}

By (\ref{eqn:f0form}), for $c \in F$ we have 
$$\underline{f_0}(c)  = \begin{cases}
[1,\, p_{\Ubark}(v)] & \text{ if } c \in \OF,\\
0 & \text{ otherwise}.\end{cases} $$

Then for $n \in \ZZ$, 
$$(\underline{f_0} \cdot \mathcal{T}^n)(c) =\left(\tilde{h}(\varpi)^{n}\cdot \underline{f_0} \right)(c) =  \begin{cases}
\left[\tilde{h}(\varpi)^n, p_{\Ubark}(v)\right] & \text{ if } c \in \OF,\\
0 & \text{ otherwise},
\end{cases} $$
so the $\SHT$-module generated by $\underline{f_0}$ contains all constant functions in $\mathcal{C}^\infty_c\left(F, \, \text{ind}_{\Tt \cap \Kt}^{\Tt} \left((\sigrt)_{\Ubark}\right)\right)$ which are supported on $\OF$. Acting by $\Tt$ and by $U^*$ gives all constant functions on scalings and translations of $\OF$, and the set of these functions spans $\mathcal{C}^\infty_c\left(F, \, \text{ind}_{\Tt \cap \Kt}^{\Tt} \left((\sigrt)_{\Ubark}\right)\right)$. Hence $f_0$ generates all functions in $\text{Ind}_{\Bbt}^{\Gt}\left(\text{ind}_{\Tt \cap \Kt}^{\Tt}\left((\sigrt)_{\Ubark}\right)\right)$ which have support in $\Bbt \cdot (U \cap K)^*.$ The set of such functions generates all of $\text{Ind}_{\Bbt}^{\Gt}\left(\text{ind}_{\Tt \cap \Kt}^{\Tt}\left((\sigrt)_{\Ubark}\right)\right)$ under the action of $\Gt$, so $\mathcal{F}$ is surjective. 
\end{proof}

By Proposition \ref{prop:SHAaction}, the action of the Hecke operator $\mathcal{T}^{\vec{r}}_1$ on any nontrivial weight space of a principal series representation $\Ind_{\Bbt}^{\Gt}(\mu\cdot \chi_{\psi})$ has a unique eigenvalue, namely $\mu\cdot \chi_{\psi}(\tilde{h}(\varpi))$. From now on, write 
\begin{equation}
\label{eqn:lambdamupsi}
\lambda_{\mu, \psi}:= \mu\cdot \chi_{\psi}(\tilde{h}(\varpi)).
\end{equation}

\begin{cor} \label{cor:shtcharwtinj}
Let $\sigrt$ be a weight of $\Kt$ with $\vec{r} \neq \vec{0}$, and let $\mu$ be a smooth character of $F^\times$ such that $\text{dim}_E\, \Hom_{\Kt}(\sigrt, \Ind_{\Bbt}^{\Gt}(\mu \cdot \chi_{\psi})) = 1.$ Define a character $\theta_{\mu, \psi}$ of $\SHT$ by setting
$$\theta_{\mu, \psi}(\tau_{-1}^{\vec{r}}) = \lambda_{\mu, \psi}.$$

Then there is a  $\Gt$-linear isomorphism
$$\pi(\vec{r}, \lambda_{\mu, \psi}) \longrightarrow \Ind_{\Bbt}^{\Gt}(\ind_{\Tt \cap \Kt}^{\Tt}(\sigrt)_{\Ubark}) \otimes_{\SHA, \, \mathcal{S}_{\vec{r}}}\theta_{\mu , \psi} $$
\end{cor}

\begin{proof}[Proof of Corollary \ref{cor:shtcharwtinj}]
Since $\vec{r} \neq 0$, Proposition \ref{prop:wtisomuniv} demonstrates that the map $\mathcal{F}: \ind_{\Kt}^{\Gt}(\sigrt) \otimes_{\SHA, \mathcal{S}_{\vec{r}}}\SHT \rightarrow \Ind_{\Bbt}^{\Gt}\left(\ind_{\Tt \cap \Kt}^{\Tt}\left(\sigrt \right)_{\Ubark} \right)$ defined in (\ref{eqn:wtisomcand}) is an isomorphism of $\Gt$- and $\SHT$-modules. Specializing the action of $\SHT$ to $\theta_{\mu \cdot \psi}$ on each side, we obtain an isomorphism of $\Gt$-modules
$$\ind_{\Kt}^{\Gt}(\sigrt)\otimes_{\SHA, \, \mathcal{S}_{\vec{r}}}\theta_{\mu, \psi} \rightarrow \Ind_{\Bbt}^{\Gt}(\mu \cdot \chi_{\psi}) \otimes_{\SHT}\theta_{\mu,\psi},$$
and domain of this isomorphism is equal to $\pi(\vec{r}, \lambda_{\mu, \psi})$ since $\theta_{\mu,\psi} \circ \mathcal{S}^{-1}_{\vec{r}}(\mathcal{T}^{\vec{r}}_1) = \lambda_{\mu, \psi}.$
\end{proof}

\begin{lem} \label{lem:pswtauxlem}
Under the hypotheses of Corollary \ref{cor:shtcharwtinj}, there is a $\Gt$-linear isomorphism 
$$\Ind_{\Bbt}^{\Gt}\left(\mu \cdot \chi_{\psi}\right) \rightarrow \Ind_{\Bbt}^{\Gt}\left(\ind_{\Tt \cap \Kt}^{\Tt}(\sigrt)_{\Ubark}\right) \otimes_{\SHT}\theta_{\mu, \psi}.$$
\end{lem}

\begin{proof}[Proof of Lemma \ref{lem:pswtauxlem}]
Since $\SHT$ is Noetherian and $\Ind_{\Bbt}^{\Gt}(-)$ is an exact functor, and also (by Lemma \ref{lem:SHTfree}) the universal module $\ind_{\Tt \cap \Kt}^{\Tt}(\sigrt)_{\Ubark}$ is a flat $\SHT$-module, the proof of the Corollary to Theorem 32 of \cite{herzig:glnnotes} goes through to give a $\Gt$-linear isomorphism $$\Ind_{\Bbt}^{\Gt}\left(\ind_{\Tt \cap \Kt}^{\Tt}((\sigrt)_{\Ubark})\otimes_{\SHT} \theta_{\mu, \psi}\right) \rightarrow \Ind_{\Bbt}^{\Gt}\left(\ind_{\Tt \cap \Kt}^{\Tt}((\sigrt)_{\Ubark})\right) \otimes_{\SHT}\theta_{\mu, \psi}.$$ Under the hypotheses of Corollary \ref{cor:shtcharwtinj}, Lemma \ref{lem:shtquotients} implies that we have also have an $\Gt$-linear isomorphism $\Ind_{\Bbt}^{\Gt}(\mu \cdot \chi_{\psi}) \rightarrow \Ind_{\Bbt}^{\Gt}\left(\ind_{\Tt \cap \Kt}^{\Tt}((\sigrt)_{\Ubark})\otimes_{\SHT} \theta_{\mu, \psi}\right).$
\end{proof}

\begin{thm} \label{thm:pswtisom}
Let $\psi: F \rightarrow \CC$ be a nontrivial additive character, let $\mu: F^\times \rightarrow E^\times$ be a smooth character, and let $\chi_{\psi}$ be as defined in (\ref{eqn:genchar}).
\begin{enumerate}
\item $\Ind_{\Bbt}^{\Gt}(\mu \cdot \chi_{\psi})$ is an irreducible $\Gt$-representation. 
\item Let $\sigrt$ be any weight of $\Kt$ such that $\Hom_{\Kt}(\sigrt, \Ind_{\Bbt}^{\Gt}(\mu \cdot \chi_{\psi})) \neq 0$, and let $\lambda_{\mu, \psi} \in E$ be as defined in (\ref{eqn:lambdamupsi}). Then there is a $\Gt$-linear isomorphism
$$\pi(\vec{r}, \lambda_{\mu, \psi}) \rightarrow \Ind_{\Bbt}^{\Gt}(\mu \cdot \chi_{\psi}).$$
\end{enumerate}
\end{thm}

\begin{proof}[Proof of Theorem \ref{thm:pswtisom}]
First suppose that $\Hom_{\Kt}(\sigrt, \Ind_{\Bbt}^{\Gt}(\mu \cdot \chi_{\psi})\big \vert_{\Kt}) \neq 0$ for some weight $\sigrt$ of $\Kt$ with $\vec{r} \notin \{\vec{0}, \overrightarrow{p-1}\}$. Then Proposition \ref{prop:psweights} (1) and (2) imply that $\sigrt$ is the unique weight of $\Kt$ contained in $\Ind_{\Bbt}^{\Gt}(\mu \cdot \chi_{\psi})$ and that $\text{dim}_E \, \Hom_{\Kt}(\sigrt, \Ind_{\Bbt}^{\Gt}(\mu \cdot \chi_{\psi})\big \vert_{\Kt}) = 1$. Composing the map of Corollary \ref{cor:shtcharwtinj} with the inverse of the map of Lemma \ref{lem:pswtauxlem}, we have a $\Gt$-linear isomorphism 
$$\pi(\vec{r}, \lambda_{\mu, \psi}) \rightarrow \Ind_{\Bbt}^{\Gt}\left(\mu \cdot \chi_{\psi} \right).$$ Hence $\pi(\vec{r}, \lambda_{\mu, \psi})$ likewise contains a unique weight of $\Kt$, namely $\sigrt$, and has a 1-dimensional $\sigrt$-weight space. Thus by Proposition \ref{prop:ktwtsignif}, the weight $\sigrt$ generates an irreducible $\Gt$-submodule of $\pi(\vec{r}, \lambda_{\mu, \psi})$. The image of $\sigrt$ under the composition $\sigrt \hookrightarrow \ind_{\Kt}^{\Gt}\sigrt \rightarrow \pi(\vec{r}, \lambda_{\mu, \psi})$ generates all of $\pi(\vec{r}, \lambda_{\mu, \psi})$ as a $\Gt$-module, so both $\pi(\vec{r}, \lambda_{\mu, \psi})$ and $\Ind_{\Bbt}^{\Gt}(\mu \cdot \chi_{\psi})$ are irreducible. We have proved both parts of Theorem \ref{thm:pswtisom} in the case where $\Ind_{\Bbt}^{\Gt}(\mu \cdot \chi_{\psi})$ contains a $\Kt$-weight different from $\tilde{\sigma}_{\vec{0}}$ and $\tilde{\sigma}_{\overrightarrow{p-1}}.$ 

Suppose that $\Ind_{\Bbt}^{\Gt}(\mu \cdot \chi_{\psi})$ does not contain any $\Kt$-weight $\sigrt$ such that $\vec{r} \notin \{\vec{0}, \overrightarrow{p-1}\}$. Then by Proposition \ref{prop:psweights} (1) and (2), the $\Kt$-weights contained in $\Ind_{\Bbt}^{\Gt}(\mu \cdot \chi_{\psi})$ are exactly $\tilde{\sigma}_{\vec{0}}$ and $\tilde{\sigma}_{\overrightarrow{p-1}}$, and both weight spaces are 1-dimensional. The map of Corollary \ref{cor:shtcharwtinj} composes with the inverse of the map of Lemma \ref{lem:pswtauxlem} to give an isomorphism $\pi_{0}(\overrightarrow{p-1}, \lambda_{\mu, \psi}) \rightarrow \Ind_{\Bbt}^{\Gt}(\mu \cdot \chi_{\psi})$. The following lemma gives the extra information needed to prove irreducibility of the principal series representations which contain these ``extremal'' $\Kt$-weights. 

\begin{lem} \label{lem:changeofweight}
If $\theta$ is a character of $\mathcal{H}(\Tt, \Tt \cap \Kt, (\tilde{\sigma}_{\overrightarrow{p-1}})_{\Ubark})$, and if $\theta'$ is the character of $\mathcal{H}(\Tt, \Tt \cap \Kt, (\tilde{\sigma}_{\vec{0}})_{\Ubark})$ defined by $\theta'(\tau_{-1}^{\vec{0}}) = \theta(\tau_{-1}^{\overrightarrow{p-1}})$, then there is a $\Gt$-linear isomorphism
$$\ind_{\Kt}^{\Gt} \tilde{\sigma}_{\overrightarrow{p-1}} \otimes_{\mathcal{H}(\Gt, \Kt, \tilde{\sigma}_{\overrightarrow{p-1}}), \, \mathcal{S}_{\overrightarrow{p-1}}}\theta \rightarrow \ind_{\Kt}^{\Gt} \tilde{\sigma}_{\vec{0}} \otimes_{\mathcal{H}(\Gt, \Kt, \tilde{\sigma}_{\vec{0}}),\mathcal{S}_{\vec{0}}} \theta'.$$
\end{lem}

\begin{proof}[Proof of Lemma \ref{lem:changeofweight}]
We have $\theta' \circ \mathcal{S}_{\vec{0}} = \theta \circ \mathcal{S}_{\vec{0}, \overrightarrow{p-1}} \circ \mathcal{S}_{\vec{0}}$ and $\theta \circ \mathcal{S}_{\overrightarrow{p-1}} = \theta' \circ \mathcal{S}_{\overrightarrow{p-1}, \vec{0}} \circ \mathcal{S}_{\overrightarrow{p-1}},$ so the Hecke operators $\mathcal{T}_1^{\overrightarrow{p-1}, \vec{0}}$ and $\mathcal{T}_1^{\vec{0}, \overrightarrow{p-1}}$ induce $\Gt$-linear homomorphisms
\begin{align*}
\ind_{\Kt}^{\Gt} \tilde{\sigma}_{\overrightarrow{p-1}} \otimes_{\mathcal{H}(\Gt, \Kt, \tilde{\sigma}_{\overrightarrow{p-1}}), \mathcal{S}_{\overrightarrow{p-1}}} \theta &\rightarrow \ind_{\Kt}^{\Gt} \tilde{\sigma}_{\vec{0}}\otimes_{\mathcal{H}(\Gt, \Kt, \tilde{\sigma}_{\vec{0}}), \mathcal{S}_{\vec{0}}} \theta',\\
\ind_{\Kt}^{\Gt}\tilde{\sigma}_{\vec{0}}\otimes_{\mathcal{H}(\Gt, \Kt, \tilde{\sigma}_{\vec{0}}), \mathcal{S}_{\vec{0}}} \theta' & \rightarrow \ind_{\Kt}^{\Gt}\tilde{\sigma}_{\overrightarrow{p-1}} \otimes_{\mathcal{H}(\Gt, \Kt, \tilde{\sigma}_{\overrightarrow{p-1}}), \mathcal{S}_{\overrightarrow{p-1}}}\theta
\end{align*}
respectively.
By Lemma \ref{lem:cowmap}, $\mathcal{T}_1^{\vec{0}, \overrightarrow{p-1}} \circ \mathcal{T}_1^{\overrightarrow{p-1}, \vec{0}} = (\mathcal{T}_1^{\overrightarrow{p-1}})^2,$ so the composition of induced maps
\begin{align*}
\ind_{\Kt}^{\Gt}\tilde{\sigma}_{\overrightarrow{p-1}} \otimes_{\mathcal{H}(\Gt, \Kt, \tilde{\sigma}_{\overrightarrow{p-1}}), \mathcal{S}_{\overrightarrow{p-1}}}\theta \rightarrow \ind_{\Kt}^{\Gt} \tilde{\sigma}_{\vec{0}}\otimes_{\mathcal{H}(\Gt, \Kt, \tilde{\sigma}_{\vec{0}}), \mathcal{S}_{\vec{0}}} \theta' \rightarrow \ind_{\Kt}^{\Gt} \tilde{\sigma}_{\overrightarrow{p-1}} \otimes_{\mathcal{H}(\Gt, \Kt, \tilde{\sigma}_{\overrightarrow{p-1}}), \mathcal{S}_{\overrightarrow{p-1}}} \theta
\end{align*} 
is given by $\theta \circ \mathcal{S}_{\overrightarrow{p-1}}(\mathcal{T}_1^{\overrightarrow{p-1}})^2$, i.e., by $ (\theta(\tau_{-1}^{\overrightarrow{p-1}}))^2.$ Likewise, the map induced by $\mathcal{T}_1^{\overrightarrow{p-1}, \vec{0}} \circ \mathcal{T}_1^{\vec{0}, \overrightarrow{p-1}} = (\mathcal{T}_1^{\vec{0}})^2$ is multiplication by $(\theta'(\tau_{-1}^{\vec{0}}))^2 = (\theta(\tau_{-1}^{\overrightarrow{p-1}}))^2$. Since $\theta$ is a character of $\mathcal{H}(\Tt, \Tt \cap \Kt, (\tilde{\sigma}_{\overrightarrow{p-1}})_{\Ubark}) \cong E[(\tau^{\overrightarrow{p-1}}_{-1})^{\pm 1}]$, the value $\theta(\tau^{\overrightarrow{p-1}}_{-1})$ lies in $E^\times$. Hence the maps induced by $\mathcal{T}_1^{\overrightarrow{p-1}, \vec{0}}$ and $(\theta(\tau_{-1}^{\overrightarrow{p-1}}))^{-2}\cdot \mathcal{T}_{1}^{\vec{0}, \overrightarrow{p-1}}$ are inverse to each other, and the map induced by $\mathcal{T}_1^{\overrightarrow{p-1}, \vec{0}}$ is the desired isomorphism. 
\end{proof}
Let $\rho$ be a nonzero $\Gt$-subrepresentation of $\Ind_{\Bbt}^{\Gt}(\mu \cdot \chi_{\psi})$. Then $\rho$ contains at least one $\Kt$-weight. Since $\rho \subset \Ind_{\Bbt}^{\Gt}(\mu \cdot \chi_{\psi})$, the $\Kt$-weights of $\rho$ lie in the set $\{ \tilde{\sigma}_{\vec{0}}, \tilde{\sigma}_{\overrightarrow{p-1}}\}$, and the corresponding weight spaces are at most 1-dimensional. Suppose that $\rho$ contains $\tilde{\sigma}_{\overrightarrow{p-1}}$. Then the image of the inclusion $\rho \hookrightarrow \Ind_{\Bbt}^{\Gt}(\mu \cdot \chi_{\psi})$ contains the image of $\tilde{\sigma}_{\overrightarrow{p-1}}$ in $\Ind_{\Bbt}^{\Gt}(\mu \cdot \chi_{\psi})$  under the composition $$\tilde{\sigma}_{\overrightarrow{p-1}} \hookrightarrow \ind_{\Kt}^{\Gt}\tilde{\sigma}_{\overrightarrow{p-1}} \twoheadrightarrow \pi(\overrightarrow{p-1}, \lambda_{\mu, \psi}) \rightarrow^{\cong} \Ind_{\Bbt}^{\Gt}(\mu \cdot \chi_{\psi}).$$ Since the latter image generates $\Ind_{\Bbt}^{\Gt}(\mu \cdot \chi_{\psi})$ as a $\Gt$-module, we have $\rho = \Ind_{\Bbt}^{\Gt}(\mu \cdot \chi_{\psi})$. 

Suppose towards a contradiction that $\rho$ does not contain $\tilde{\sigma}_{\overrightarrow{p-1}}$. Then $\rho$ contains $\tilde{\sigma}_{\vec{0}}$, so Frobenius reciprocity provides a nonzero $\Gt$-linear map $\ind_{\Kt}^{\Gt}\tilde{\sigma}_{\vec{0}} \rightarrow \rho$. Pulling back through the quotient $\ind_{\Kt}^{\Gt}\tilde{\sigma}_{\vec{0}} \twoheadrightarrow \pi(\vec{0}, \lambda_{\mu, \psi})$, we also get a nonzero $\Gt$-linear map 
\begin{equation} \label{eqn:cowprep}
\pi(\vec{0}, \lambda_{\mu, \psi}) \rightarrow \rho.
\end{equation} 
Since $\lambda_{\mu, \psi} \neq 0$, the character $\Theta'_{\mu, \psi}$ of $\mathcal{H}(\Gt, \Kt, \tilde{\sigma}_{\vec{0}})$ extends to a character $\theta'_{\mu, \psi}$ of $\mathcal{H}(\Tt, \TKt, \tilde{\sigma}_{\vec{0}})$ by setting $\mathcal{\theta}_{\lambda, \mu}(\tau_{-1}^{\vec{0}}) = \lambda_{\mu, \psi}^{-1}$, and $\pi(\vec{0}, \lambda_{\mu, \psi}) = \ind_{\Kt}^{\Gt}\tilde{\sigma}_{\vec{0}} \otimes_{\mathcal{H}(\Gt, \Kt, \tilde{\sigma}_{\vec{0}}), \mathcal{S}_{\vec{0}}} \theta_{\mu, \psi}.$ Likewise, $\pi(\overrightarrow{p-1}, \lambda_{\mu, \psi}) = \ind_{\Kt}^{\Gt}\tilde{\sigma}_{\overrightarrow{p-1}}\otimes_{\mathcal{H}(\Gt, \Kt, \tilde{\sigma}_{\overrightarrow{p-1}}), \, \mathcal{S}_{\overrightarrow{p-1}}}\theta_{\mu, \psi}$, where $\theta_{\mu, \psi}(\tau_{-1}^{\overrightarrow{p-1}}) = \lambda_{\mu, \psi}$. Then by Lemma \ref{lem:changeofweight} we have $\pi(\vec{0}, \lambda_{\mu, \psi}) \cong \pi(\overrightarrow{p-1}, \lambda_{\mu, \psi})$, and precomposing (\ref{eqn:cowprep}) with this isomorphism we have a nonzero $\Gt$-linear map $$\pi(\overrightarrow{p-1}, \lambda_{\mu, \psi}) \rightarrow \rho.$$ Precomposing further with the quotient $\ind_{\Kt}^{\Gt}\tilde{\sigma}_{\overrightarrow{p-1}} \twoheadrightarrow \pi(\overrightarrow{p-1}, \lambda_{\mu, \psi})$ produces a nonzero $\Gt$-linear map $\ind_{\Kt}^{\Gt}\tilde{\sigma}_{\overrightarrow{p-1}} \rightarrow \rho$. By Frobenius reciprocity, the latter map corresponds to an injection $\tilde{\sigma}_{\overrightarrow{p-1}} \hookrightarrow \rho$, so $\rho$ contains $\tilde{\sigma}_{\overrightarrow{p-1}}$ after all. Thus $\rho = \Ind_{\Bbt}^{\Gt}(\mu \cdot \chi_{\psi})$. This completes the proof of (1) for principal series representations containing $\tilde{\sigma}_{\vec{0}}$ and $\tilde{\sigma}_{\overrightarrow{p-1}}$. 

In the case that the weights of $\Ind_{\Bbt}^{\Gt}(\mu \cdot \chi_{\psi})$ are $\{ \tilde{\sigma}_{\vec{0}}, \tilde{\sigma}_{\overrightarrow{p-1}}\}$, we have already seen that $ \pi(\overrightarrow{p-1}, \lambda_{\mu, \psi})  \cong \Ind_{\Bbt}^{\Gt}(\mu \cdot \chi_{\psi})$. By Lemma \ref{lem:cowmap}, also $\pi(\vec{0}, \lambda_{\mu, \psi}) \cong \Ind_{\Bbt}^{\Gt}(\mu \cdot \chi_{\psi})$. This proves (2) for principal series representations which contain the extremal $\Kt$-weights. 
\end{proof}

\begin{rem}
\label{rem:thesispsirred}
Alternatively, one can prove irreducibility of genuine principal series representations without any comparison to compact inductions. This was done by the author in \cite{peskin:thesis} following the strategy of Abdellatif \cite{abdellatif:thesis} for $SL_2(F)$. The idea is to show that $\Ind_{\Bt}^{\Gt}(\mu \cdot \chi_{\psi})$ fits into an exact sequence of $\Bt$-modules
$$0 \rightarrow W_{\mu, \psi} \rightarrow \Ind_{\Bbt}^{\Gt}(\mu \cdot \chi_{\psi}) \rightarrow \mu \cdot \chi_{\psi} \rightarrow 0$$ 
in which $W_{\mu, \psi}$ is a certain irreducible infinite-dimensional $\Bt$-module, so that each $\Ind_{\Bt}^{\Gt}(\mu \cdot \chi_{\psi})$ is of length 2 as a $\Bt$-module and admits $\mu \cdot \chi_{\psi}$ as its unique 1-dimensional subquotient. If $\Ind_{\Bt}^{\Gt}(\mu \cdot \chi_{\psi})$ were reducible as a $\Gt$-module, then by smooth Frobenius reciprocity would admit $\mu \cdot \chi_{\psi}$ as a 1-dimensional subquotient. But this is impossible, since the abelianization of $\Gt$ is trivial (cf. Lemma \ref{lem:commutators}) and on the other hand $\mu \cdot \chi_{\psi}$ is genuine and therefore nontrivial. 
\end{rem}

Next we give a dictionary between parameters with respect to $\Kt$ and to $\Kt'$. 

\begin{thm} \label{thm:dictionary}
Given $\vec{r} \in \{0, \dots, p-1\}^f$, let $\vec{r'}$ denote any vector in $\{0, \dots, p-1\}^f$ such that 
 $$\sum_{i = 0}^{f-1}r_i'p^i \equiv \sum_{i = 0}^{f-1}\left(r_i + \frac{p-1}{2}\right)p^i \pmod{q-1}$$ (thus $\vec{r'}$ is uniquely determined by $\vec{r}$ if $\vec{r}\neq \overrightarrow{\frac{p-1}{2}}$, and $\vec{r'} \in \{ \vec{0}, \overrightarrow{p-1}\}$ if $\vec{r} = \overrightarrow{\frac{p-1}{2}}$).
\begin{enumerate}
\item For any $\vec{r} \in \{ 0, \dots, p-1\}^f$ and $\lambda \in E^\times$, $$\pi(\vec{r}, \lambda) \cong \pi'(\vec{r'}, \lambda);$$
in particular, $$\pi(\vec{0}, \lambda) \cong \pi(\overrightarrow{p-1}, \lambda) \cong \pi'(\overrightarrow{\frac{p-1}{2}}, \lambda)$$ and $$\pi'(\vec{0}, \lambda) \cong \pi'(\overrightarrow{p-1}, \lambda) \cong \pi(\overrightarrow{\frac{p-1}{2}}, \lambda);$$
\item the isomorphisms of (1) are the only equivalences between nonsupersingular cokernel modules. 
\end{enumerate}
\end{thm}
\begin{proof}[Proof of Theorem \ref{thm:dictionary}]

\begin{enumerate}
\item
Let $\vec{r} \in \{0, \dots, p-1\}^f$ and $\lambda \in E^\times$, and fix an additive $\CC$-valued character $\psi$ of conductor $m$. Let $\psi$ be a nontrivial additive $\CC$-valued character of $F$ of conductor $m$. Denote the parity of $m$ by $s(m) \in \{0, 1\}$. There exists a character $\mu$ of $F^\times$ such that $\mu \big \vert_{\OF^\times} \cong \delta_{\vec{r}} \cdot (-, \varpi)_F^m$ and $\mu(\varpi) = (-1, \varpi)_F \cdot\lambda\cdot \gamma_{\mathfrak{k}}(\bar{\psi})^{1 - 2 s(m)}$. Then by Theorem \ref{thm:raoweilindexof} (2), the genuine character $\mu \cdot \chi_{\psi}$ of $\Tt$ satisfies 
\begin{equation}
\label{eqn:dict1}
\mu \cdot \chi_{\psi} \big \vert_{\Tt \cap \Kt} \cong \epsilon \cdot \delta_{\vec{r}} \cdot (-, \varpi)_F^{2m} = \tilde{\delta}_{\vec{r}} = (\sigrt)_{\Ubark}
\end{equation}
and
\begin{equation}
\label{eqn:dict2}
\mu \cdot \chi_{\psi} (\tilde{h}(\varpi)) = \lambda.
\end{equation}
The isomorphism (\ref{eqn:dict1}) implies, due to Proposition \ref{prop:psweights} (1) and (2), that the $\Kt$-weight $\sigrt$ is contained in $\Ind_{\Bbt}^{\Gt}(\mu \cdot \chi_{\psi})$. Hence (\ref{eqn:dict1}) and (\ref{eqn:dict2}) together with Theorem \ref{thm:pswtisom} (2) imply that $\pi(\vec{r}, \lambda) \cong \Ind_{\Bbt}^{\Gt}(\mu \cdot \chi_{\psi})$. 

Now let $\vec{r'} \in \{ 0, \dots, p-1\}^f$ satisfy the condition of the theorem with respect to $\vec{r}$. Then $\mu' := \delta_{\overrightarrow{\frac{p-1}{2}}}\cdot \mu = (-, \varpi)_F \cdot \mu$ is a character of $F^\times$ such that $\mu' \big \vert_{\OF^\times} \cong \delta_{\vec{r'}} \cdot (-, \varpi)_F^m$ and $\mu'(\varpi) = \mu(\varpi)$. By Theorem \ref{thm:pswtisom}, $\pi(\vec{r'}, \lambda) \cong \Ind_{\Bbt}^{\Gt}(\mu' \cdot \chi_{\psi})$. And by Lemma \ref{lem:pipi'isom} we have $\pi'(\vec{r'}, \lambda) \cong \left(\pi(\vec{r'}, \lambda) \right)^{\tilde{\alpha}},$ so
\begin{align*}
\pi'(\vec{r'}, \lambda) & \cong \left(\Ind_{\Bbt}^{\Gt}(\mu' \cdot \chi_{\psi}) \right)^{\tilde{\alpha}} \cong \Ind_{\Bbt}^{\Gt}\left( \mu' \cdot \chi_{\psi} \right)^{\tilde{\alpha}}.
\end{align*}
For $(t(x), \zeta) \in \Tt$,
\begin{align*}
(\mu' \cdot \chi_{\psi})^{\tilde{\alpha}}((t(x), \zeta)) & = (\mu' \cdot \chi_{\psi})\left((\tilde{\alpha})^{-1}(t(x), \zeta\right)\tilde{\alpha})\\
& = (\mu' \cdot \chi_{\psi})\left( (t(x), \zeta \cdot (x, \varpi)_F) \right)\\
& = \zeta\cdot (x, \varpi)_F \cdot \mu'(x) \cdot \gamma_F(x, \psi)^{-1}\\
& = \zeta \cdot (x, \varpi)_F \cdot \delta_{\overrightarrow{\frac{p-1}{2}}} \cdot \mu(x)\cdot \gamma_F(x, \psi)^{-1}\\
& = \zeta \cdot \mu(x) \cdot \gamma_F(x, \psi)^{-1}\\
& = (\mu \cdot \chi_{\psi})((t(x), \zeta)), 
\end{align*}
so $(\mu' \cdot \chi_{\psi})^{\tilde{\alpha}} \cong \mu \cdot \chi_{\psi}$. Thus $\pi'(\vec{r'}, \lambda) \cong \Ind_{\Bbt}^{\Gt}(\mu \cdot \chi_{\psi}) \cong \pi(\vec{r}, \lambda).$ 

(2) Let $\vec{r}$, $\vec{s} \in \{ 0, \dots, p-1\}^f$ and let $\lambda, \nu \in E^\times$. Suppose that $\pi(\vec{r}, \lambda) \cong \pi(\vec{s}, \nu)$. Then $\pi(\vec{r}, \lambda)$ is a nonsupercuspidal $\Gt$-representation containing the $\Kt$-weights $\tilde{\sigma}_{\vec{r}}$ and $\tilde{\sigma}_{\vec{s}}$, so by Proposition \ref{prop:psweights} (1) and (2) either $\vec{r} = \vec{s}$ or $\{ \vec{r}, \vec{s}\} = \{ \vec{0}, \overrightarrow{p-1}\}$. In either case, the eigenvalue of $\mathcal{T}^{\vec{r}}$ on $\Hom_{\Kt}(\sigrt, \pi(\vec{r}, \lambda)\big \vert_{\Kt})$ is equal to the eigenvalue of $\mathcal{T}^{\vec{r}}$ on $\Hom_{\Kt}(\sigst, \pi(\vec{s}, \nu))$, i.e., $\lambda = \nu$. Thus the only equivalences between cokernel modules with respect to $\Kt$ are the isomorphisms $\pi(\vec{0}, \lambda) \cong \pi(\overrightarrow{p-1}, \lambda)$. Conjugating by $\tilde{\alpha}$ and using Lemma \ref{lem:pipi'isom}, we see that the only equivalences between cokernel modules with respect to $\Kt'$ are the isomorphisms $\pi'(\vec{0}, \lambda) \cong \pi'(\overrightarrow{p-1}, \lambda).$ 

Finally, let $\vec{r}$, $\vec{s} \in \{ 0, \dots, p-1\}^f$ and $\lambda, \nu \in E^\times$, and suppose that $\pi(\vec{r}, \lambda) \cong \pi'(\vec{s}, \nu)$. By Theorem \ref{thm:pswtisom}, there is a genuine character $\mu \cdot \chi_{\psi}$ of $\Tt$ such that $\pi(\vec{r}, \lambda) \cong \Ind_{\Bbt}^{\Gt}(\mu \cdot \chi_{\psi})$. Then by Proposition \ref{prop:SHAaction}, we must have $\lambda = \nu = \mu \cdot \chi_{\psi}(\tilde{h}(\varpi)).$ And by Lemma \ref{lem:pipi'isom}, $\pi'(\vec{s}, \lambda) \cong \left(\pi(\vec{s}, \lambda) \right)^{\tilde{\alpha}}$, so 
$$\pi(\vec{s}, \lambda) \cong \left( \pi'(\vec{s}, \lambda)\right)^{\tilde{\alpha}^{-1}} \cong \left(\Ind_{\Bbt}^{\Gt} \mu \cdot \chi_{\psi} \right)^{\tilde{\alpha}^{-1}} \cong \Ind_{\Bbt}^{\Gt} \left(\mu \cdot \chi_{\psi} \right)^{\tilde{\alpha}^{-1}}.$$ By a similar calculation to the one in the proof of (1), $(\mu \cdot \chi_{\psi})^{\tilde{\alpha}^{-1}} \cong  \delta_{\overrightarrow{\frac{p-1}{2}}} \cdot \mu \cdot \chi_{\psi}$. Since $\sigst$ is a $\Kt$-weight of $\Ind_{\Bbt}^{\Gt}( \delta_{\overrightarrow{\frac{p-1}{2}}} \cdot \mu\cdot \chi_{\psi} )$, by Lemma \ref{lem:natlisom} we have an isomorphism of $\Tt \cap \Kt$-representations $(\sigst)_{\Ubark}  \cong  \delta_{\overrightarrow{\frac{p-1}{2}}} \cdot \mu \cdot \chi_{\psi} \big \vert_{\Tt \cap \Kt}$. And since $\sigrt$ is a $\Kt$-weight of $\Ind_{\Bbt}^{\Gt}(\mu \cdot \chi_{\psi})$, we have $(\sigrt)_{\Ubark} \cong \mu \cdot \chi_{\psi} \big \vert_{\Tt \cap \Kt}$. Hence $\tilde{\delta}_{\vec{s}} = \delta_{\overrightarrow{\frac{p-1}{2}}} \cdot \tilde{\delta}_{\vec{r}}$, which is only true if $\vec{s}$ satisfies the condition 
$$\sum_{i = 0}^{f-1}s_ip^i \equiv \sum_{i = 0}^{f-1} \left(r_i + \frac{p-1}{2}\right)p^i \pmod{q-1}.$$
\end{enumerate}
\end{proof}

\begin{cor} \label{cor:dictionary2}
Let $\psi: F \rightarrow \CC$ be a nontrivial additive character of conductor $m$, let $\mu: F^\times \rightarrow E^\times$ be a smooth multiplicative character, and let $\chi_{\psi}$ be the genuine $E$-valued character of $\Tt$ defined in (\ref{eqn:genchar}). 
\begin{enumerate}
\item Suppose either that $2 \big \vert m$ and $\mu \big \vert_{\OF^\times} = 1$, or that $2 \not \big \vert m$ and $\mu \big \vert_{\OF^\times} = (-, \varpi)_F$. Then the parameters of $\Ind_{\Bbt}^{\Gt}(\mu \cdot \chi_{\psi})$ with respect to $\Kt$ are $(\vec{0}, \lambda_{\mu, \psi})$ and $(\overrightarrow{p-1}, \lambda_{\mu, \psi})$, and $\Ind_{\Bbt}^{\Gt}(\mu \cdot \chi_{\psi})$ has the unique parameter $(\overrightarrow{\frac{p-1}{2}}, \lambda_{\mu, \psi})$ with respect to $\Kt'$. 

\item Suppose either that $2 \big \vert m$ and $\mu \big \vert_{\OF^\times} = (-, \varpi)_F$, or that $2 \not \big \vert m$ and $\mu \big \vert_{\OF^\times} = 1$. Then $\Ind_{\Bbt}^{\Gt}(\mu \cdot \chi_{\psi})$ has the unique parameter $(\overrightarrow{\frac{p-1}{2}}, \lambda_{\mu, \psi})$ with respect to $\Kt$, and the parameters of $\Ind_{\Bbt}^{\Gt}(\mu \cdot \chi_{\psi})$ with respect to $\Kt'$ are $(\vec{0}, \lambda_{\mu, \psi})$ and $(\overrightarrow{p-1}, \lambda_{\mu, \psi})$.  

\item Otherwise, $\Ind_{\Bbt}^{\Gt}(\mu \cdot \chi_{\psi})$ has a unique parameter with respect to $\Kt$, and this parameter is of the form $(\vec{r}, \lambda_{\mu, \psi})$ for some $\vec{r} \notin \{\vec{0}, \overrightarrow{\frac{p-1}{2}}, \overrightarrow{p-1}\}$. $\Ind_{\Bbt}^{\Gt}(\mu\cdot \chi_{\psi})$ also has a unique parameter with respect to $\Kt'$, equal to $(\vec{r'}, \lambda_{\mu, \psi})$ where $\vec{r'}$ is the unique vector in $\{0, \dots, p-1\}^f$ such that $\sum_{i = 0}^{f-1}r_i'p^i \equiv \sum_{i = 0}^{f-1}\left(r_i + \frac{p-1}{2}\right) p^i \pmod{q-1}$. 
\end{enumerate}
\end{cor}

\begin{proof}[Proof of Corollary \ref{cor:dictionary2}]
First we check that under no circumstance can a genuine principal series representation have a supersingular parameter. If $\pi = \Ind_{\Bbt}^{\Gt}(\mu \cdot \chi_{\psi})$ has a supersingular parameter with respect to $\Kt$, then by Lemma \ref{lem:cokernelnonzero} (2), there is a weight $\tilde{\sigma}_{\vec{r}}$ of $\Kt$ such that $\pi$ contains $\tilde{\sigma}_{\vec{r}}$ and such that $\mathcal{T}^{\vec{r}}_1$ acts by 0 on $\Hom_{\Kt}(\tilde{\sigma}_{\vec{r}}, \varpi\big \vert_{\Kt})$. But by Proposition \ref{prop:SHAaction} (1), this implies that $(\mu \cdot\chi_{\psi})(\tilde{h}(\varpi)) = 0$, which is impossible, so $\pi$ has no supersingular parameter with respect to $\Kt$. A similar argument using Lemma \ref{lem:kt'cokernelnonzero} (2) Proposition \ref{prop:SHAaction} (2) shows that $\pi$ has no supersingular parameter with respect to $\Kt'$.  

All points regarding the nonsupersingular parameters follow from Theorem \ref{thm:dictionary} together with Proposition \ref{prop:psweights}. 
\end{proof}

Finally, we hold the character $\mu: F^{\times} \rightarrow E$ fixed and put the statement of Corollary \ref{cor:dictionary2} in a form which highlights the dependence on $\psi$ of the parameters for principal series representations. For $a \in F^\times$, let $\psi_a$ denote the character $x \mapsto \psi(ax)$. A consequence of Theorem \ref{thm:raoweilindex} is the formula $\gamma_F(c, \psi_a) = (a, c)_F \cdot \gamma_F(c, \psi)$. Thus $\chi_{\psi}$ and $\chi_{\psi_a}$ define identical bijections
$$\{ \text{smooth characters } F^\times \rightarrow E^\times\} \leftrightarrow \{ \text{genuine characters } \Tt \rightarrow E^\times \}$$
if and only if $a \in (F^\times)^2$. Likewise, the principal series representations $\Ind_{\Bbt}^{\Gt}(\mu \cdot \chi_{\psi})$ and $\Ind_{\Bbt}^{\Gt}(\mu \cdot \chi_{\psi_a})$ have identical parameters with respect to $\Kt$ and $\Kt'$ if and only if $a \in (F^\times)^2$. The full dependence is as follows: 

\begin{cor}\label{cor:dictionary3}
Keep the notation of Corollary \ref{cor:dictionary2}, and in addition let $a = u \cdot \varpi^{v_F(a)} \in F^\times$ with $u \in \OF^\times$. Suppose that $(\vec{r}, \lambda)$ is a parameter of $\Ind_{\Bbt}^{\Gt}(\mu \cdot \chi_{\psi})$ with respect to $\Kt$, and that $(\vec{r'}, \lambda)$ is a parameter of $\Ind_{\Bbt}^{\Gt}(\mu \cdot \chi_{\psi})$ with respect to $\Kt'$. Then 
\begin{enumerate}
\item If $2 \big \vert v_F(a)$, then $(\vec{r}, (u, \varpi)_F \cdot \lambda)$ is a parameter of $\Ind_{\Bbt}^{\Gt}(\mu \cdot \chi_{\psi_a})$ with respect to $\Kt$ and $(\vec{r'}, (u, \varpi)_F \cdot \lambda)$ is a parameter of $\Ind_{\Bbt}^{\Gt}(\mu \cdot \chi_{\psi_a})$ with respect to $\Kt'$. 

\item If $2 \not \big \vert v_F(a)$, then $(\vec{r'}, (-u, \varpi)_F \cdot \lambda)$ is a parameter of $\Ind_{\Bbt}^{\Gt}(\mu \cdot \chi_{\psi_a})$ with respect to $\Kt$ and $(\vec{r}, (-u, \varpi)_F \cdot \lambda)$ is a parameter of $\Ind_{\Bbt}^{\Gt}(\mu \cdot \chi_{\psi_a})$ with respect to $\Kt'$. 
\end{enumerate}
\end{cor}

\begin{proof}[Proof of Corollary \ref{cor:dictionary3}]
The conductors of $\psi$ and $\psi_a$ have equal parity if $2 \big \vert v_F(a)$ and opposite parity otherwise. The statements about the first  components of the given parameters then follow from \ref{cor:dictionary2}. The second component of each parameter of $\Ind_{\Bbt}^{\Gt}(\mu \cdot \chi_{\psi_a})$ is equal to $\mu \cdot \chi_{a}(\tilde{h}(\varpi)).$ Since $(\vec{r}, \lambda)$ and $(\vec{r'},\lambda)$ are parameters for $\Ind_{\Bbt}^{\Gt}(\mu \cdot \chi_{\psi})$, we have $\mu \cdot \chi_{\psi}(\tilde{h}(\varpi)) = \lambda$. We also have $\mu \cdot \chi_{\psi}(\tilde{h}(\varpi)) = (-1, \varpi)_F \cdot \mu(\varpi) \cdot \gamma_F(\varpi, \psi)^{-1}$, while $\mu \cdot \chi_{\psi_a}(\tilde{h}(\varpi)) = (-1, \varpi)_F \cdot \mu(\varpi) \cdot \gamma_F(\varpi, \psi_a)^{-1}$, so 
$$\mu \cdot \chi_{\psi_a}(\tilde{h}(\varpi)) = \frac{\gamma_F(\varpi, \psi_a)}{\gamma_F(\varpi, \psi)} \cdot \lambda = (a, \varpi)_F \cdot \lambda.$$
If $2 \big \vert v_F(a)$ then $(a, \varpi)_F = (u, \varpi)_F$, and if $2 \not \big \vert v_F(a)$ then $(a, \varpi)_F = (u\varpi, \varpi)_F = (-u, \varpi)_F$. 
\end{proof}

\section{Proof of the classification theorem}
\label{sec:classnproof}

\begin{proof}[Proof of Theorem \ref{thm:irredadmissclassn}]
Let $\pi$ be a smooth, genuine, irreducible, admissible representation of $\Gt$. Then by Proposition \ref{prop:possparamexists}, $\pi$ has a parameter with respect to each of $\Kt$ and $\Kt'$. Suppose that $\pi$ is nonsupercuspidal; then by Theorem \ref{thm:pswtisom}, $\pi$ is an irreducible genuine principal series representation. So by Corollary \ref{cor:dictionary2}, $\pi$ has a nonsupersingular parameter with respect to each of $\Kt$ and $\Kt'$, and has no supersingular parameters.

Conversely, suppose that $\pi$ is a smooth, genuine, irreducible $\Gt$-representation which has a nonsupersingular parameter. If $\pi$ has the nonsupersingular parameter $(\vec{r}, \lambda)$ with respect to $\Kt$, then by definition $\pi$ is a quotient of $\pi(\vec{r}, \lambda)$. Since $\lambda \in E^\times$, there exists a character $\mu: F^\times \rightarrow E^\times$ such that $\mu(\varpi) = \lambda$ and such that $\mu \big \vert_{\OF^\times} \cong \delta_{\vec{r}}$. If $\psi: F \rightarrow \CC$ is a nontrivial additive character with even conductor $m$, then by Proposition \ref{prop:SHAaction} (1) the genuine $E$-valued character $\mu \cdot \chi_{\psi}$ of $\Tt$ satisfies $\mu \cdot \chi_{\psi}(\tilde{h}(\varpi)) = \lambda$, so by Theorem \ref{thm:pswtisom} (2) we have $\pi(\vec{r}, \lambda) \cong \Ind_{\Bbt}^{\Gt}(\mu \cdot \chi_{\psi})$. Thus $\pi$ is a quotient of a principal series representation, and in fact $\pi \cong \Ind_{\Bbt}^{\Gt}(\mu \cdot \chi_{\psi})$ since by Theorem \ref{thm:pswtisom} (1) the latter is also irreducible. If $\pi$ has a nonsupersingular parameter $(\vec{r}, \lambda)$ with respect to $\Kt'$, then the same argument (only changed to take $\psi$ of odd conductor) shows again that $\pi$ is a genuine principal series representation. Corollary \ref{cor:dictionary2} shows again that $\pi$ has no supersingular parameters. 

Thus an irreducible admissible genuine representation $\pi$ of $\Gt$ has a nonsupersingular parameter if and only if all of its parameters are nonsupersingular, if and only if $\pi$ is a genuine principal series representation. 
\end{proof}


\bibliographystyle{abbrv}
\bibliography{sphericalrefs}

\vspace{.5cm}

\noindent \textit{Address:}\\Department of Mathematics, Weizmann Institute of Science\\
234 Herzl Street, Rehovot 7610001 Israel

\noindent \textit{Email:} \url{laura.peskin@weizmann.ac.il}

\end{document}